\documentclass[12pt,twoside]{amsart}
\usepackage[margin=3cm]{geometry}
\usepackage[colorlinks=false]{hyperref}
\usepackage[english]{babel}
\usepackage{graphicx,titling}
\usepackage{float}
\usepackage{amsmath,amsfonts,amssymb,amsthm}
\usepackage{lipsum}
\usepackage[T1]{fontenc}
\usepackage{fourier}
\usepackage{color}
\usepackage[latin1]{inputenc}
\usepackage{esint}
\usepackage{caption}
\usepackage{ccicons}

\makeatletter
\def\blfootnote{\xdef\@thefnmark{}\@footnotetext}
\makeatother

\usepackage{combelow}
\newcommand\ccnote{
    \blfootnote{\copyright\,\, Andr\'e~Guerra, Bogdan~Rai\cb{t}\u{a}, and Matthew~Schrecker}
    \blfootnote{\ccLogo\, \ccAttribution\,\, Licensed under a \href{https://creativecommons.org/licenses/by/4.0/}{Creative Commons Attribution License (CC-BY)}.}
}

\usepackage[export]{adjustbox}
\numberwithin{equation}{section}
\usepackage{setspace}
\renewcommand{\le}{\leqslant}
\renewcommand{\leq}{\leqslant}

\renewcommand{\geq}{\geqslant}
\renewcommand{\mathbb}{\varmathbb}
\usepackage{fancyhdr}
\newtheorem{theorem}{Theorem}[section]
\newtheorem{lemma}[theorem]{Lemma}
\newtheorem{corollary}[theorem]{Corollary}
\newtheorem{proposition}[theorem]{Proposition}
\newtheorem{definition}[theorem]{Definition}
\newtheorem{remark}[theorem]{Remark}
\usepackage{amssymb}
\usepackage{amsmath}
\usepackage{amssymb}
\usepackage{amsthm}
\usepackage{bbm}
\usepackage{bm}
\usepackage{mathtools}
\usepackage{mathrsfs}
\usepackage{centernot}
\usepackage[T1]{fontenc}

\usepackage[usenames,dvipsnames,svgnames,table]{xcolor}
\usepackage{esint}
\allowdisplaybreaks
\usepackage{combelow}
\usepackage{enumitem}
\setenumerate{label={\rm (\alph{*})}}
\usepackage{graphicx}
\usepackage{tikz}
\makeatletter
\newcommand*\bigcdot{\mathpalette\bigcdot@{.5}}
\newcommand*\bigcdot@[2]{\mathbin{\vcenter{\hbox{\scalebox{#2}{$\m@th#1\bullet$}}}}}
\makeatother
\theoremstyle{plain}
\newtheorem{bigthm}{Theorem}

\newtheorem{bigprop}[bigthm]{Proposition}

\newtheorem{conjecture}[theorem]{Conjecture}
\newtheorem{op}[theorem]{Open Problem}
\newtheorem{question}[theorem]{Question}

\expandafter\let\expandafter\oldproof\csname\string\proof\endcsname
\let\oldendproof\endproof
\renewenvironment{proof}[1][\proofname]{%
  \oldproof[\upshape \bfseries #1]%
}{\oldendproof}

\def\XXint#1#2#3{{\setbox0=\hbox{$#1{#2#3}{\int}$ }
		\vcenter{\hbox{$#2#3$ }}\kern-.6\wd0}}

\newcommand{\di}{\operatorname{div}}
\newcommand{\Div}{\operatorname{Div}}
\newcommand{\Curl}{\operatorname{Curl}}
\newcommand{\dif }{\operatorname{d}\!}

\newcommand{\tr}{\operatorname{Tr}}
\newcommand{\R}{\mathbb{R}}
\newcommand{\B}{\mathbb{B}}
\newcommand{\C}{\mathbb{C}}
\newcommand{\T}{\mathbb{T}}

\let\d\relax
\newcommand{\d}{\partial}

\newcommand{\locc}{\operatorname{loc}}

\renewcommand{\geq}{\geqslant}
\newcommand{\lebe}{\operatorname{L}}
\newcommand{\hold}{\operatorname{C}}
\newcommand{\curl}{\operatorname{curl}}
\renewcommand{\leq}{\leqslant}

\newcommand{\sym}{\operatorname{sym}}

\newcommand{\supp}{\operatorname{supp}}
\newcommand{\lin}{\operatorname{Lin}}

\newcommand{\cala}{\mathcal{A}}
\usepackage{tikz-cd}
\newcommand{\de}{\delta}

\newcommand{\eps}{\varepsilon}
\newcommand{\al}{\alpha}
\newcommand{\la}{\lambda}

\newcommand{\p}{\partial}
\newcommand{\V}{\mathbb{V}}
\newcommand{\W}{\mathbb{W}}
\newcommand{\A}{\mathcal{A}}
\newcommand{\D}{\textup{D}}
\newcommand{\e}{\varepsilon}

\newcommand{\beq}{\begin{equation}}
\newcommand{\eeq}{\end{equation}}
\newcommand{\beqs}{\begin{equation*}}
\newcommand{\eeqs}{\end{equation*}}
\newcommand{\beqa}{\begin{equation}\begin{aligned}}
\newcommand{\eeqa}{\end{aligned}\end{equation}}
\newcommand{\beqas}{\begin{equation*}\begin{aligned}}
\newcommand{\eeqas}{\end{aligned}\end{equation*}}

\address{Andr\'e~Guerra, Institute for Theoretical Studies, ETH Z\"urich, CLV, Clausiusstrasse 47, 8006 Z\"urich, Switzerland}
\email{andre.guerra@eth-its.ethz.ch}
\medskip
\address{Bogdan~Rai\cb{t}\u{a}, Department of Mathematics and Statistics, Georgetown University, 3700 O St NW, 20057, Washington DC, USA; Department of Mathematics, Alexandru-Ioan Cuza University,  Blvd. Carol I, no. 11,  Ia\cb{s}i 700506,
Romania} 
\email{br607@georgetown.edu}
\medskip
\address{Matthew Schrecker, Department of Mathematical Sciences, University of Bath, Claverton Down, Bath BA2 7AY, UK}
\email{mris21@bath.ac.uk}

\begin{document}

\thispagestyle{empty}

\begin{minipage}{0.28\textwidth}
\begin{figure}[H]
\includegraphics[width=2.5cm,height=2.5cm,left]{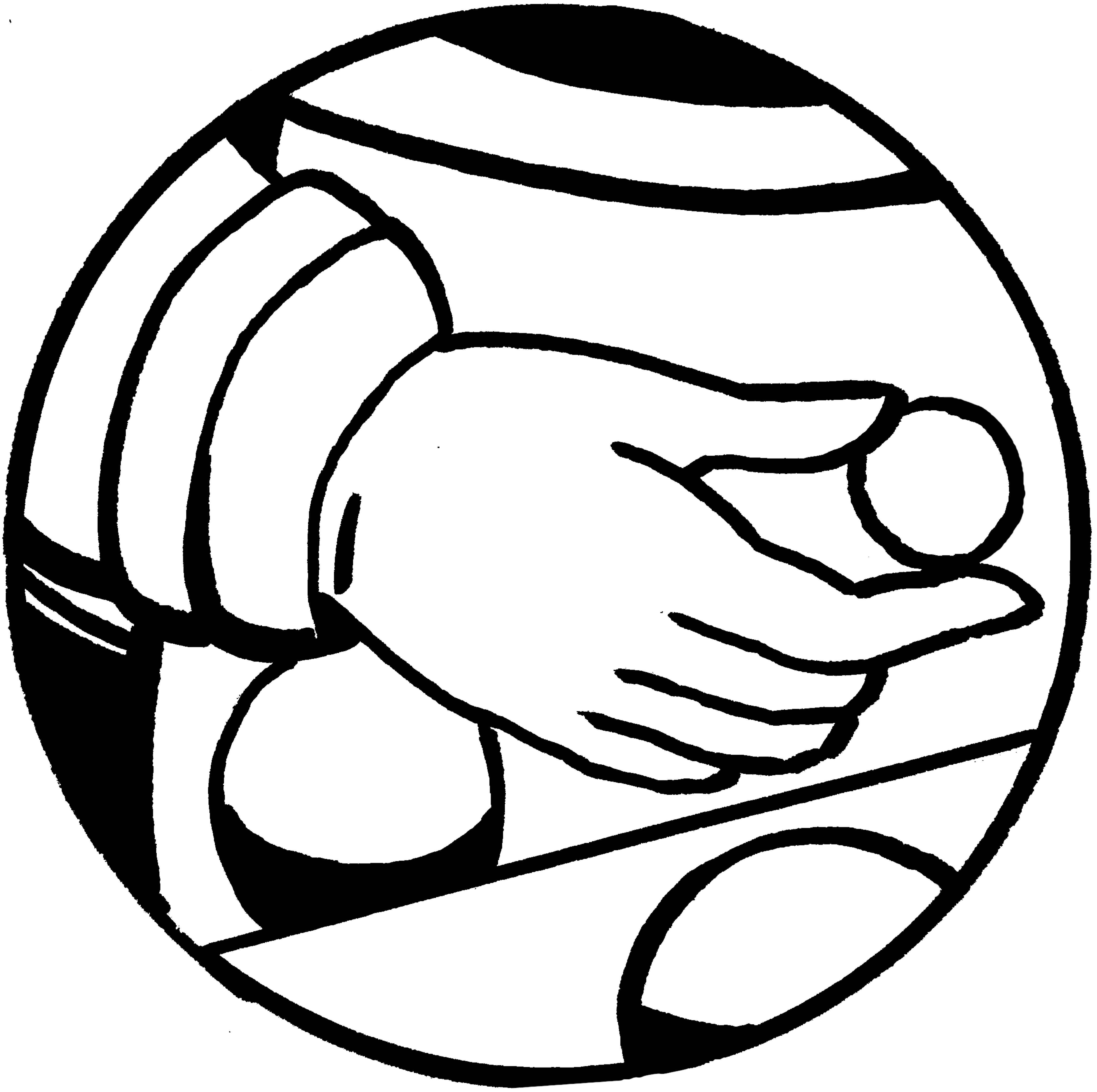}
\end{figure}
\end{minipage}
\begin{minipage}{0.7\textwidth} 
\begin{flushright}
Ars Inveniendi Analytica (2024), Paper No. 1, 56 pp.
\\
DOI 10.15781/7187-xq59
\\
ISSN: 2769-8505
\end{flushright}
\end{minipage}

\ccnote

\vspace{1cm}


\begin{center}
\begin{huge}
\textit{Compensation phenomena for concentration effects via }

\textit{nonlinear elliptic estimates}

\end{huge}
\end{center}

\vspace{1cm}


\begin{minipage}[t]{.28\textwidth}
\begin{center}
{\large{\bf{Andr\'e Guerra}}} \\
\vskip0.15cm
\footnotesize{ETH Z\"urich}
\end{center}
\end{minipage}
\hfill
\noindent
\begin{minipage}[t]{.28\textwidth}
\begin{center}
{\large{\bf{Bogdan Rai\cb{t}\u{a}}}} \\
\vskip0.15cm
\footnotesize{Georgetown University}
\end{center}
\end{minipage}
\hfill
\noindent
\begin{minipage}[t]{.28\textwidth}
\begin{center}
{\large{\bf{Matthew Schrecker}}} \\
\vskip0.15cm
\footnotesize{University of Bath} 
\end{center}
\end{minipage}

\vspace{1cm}


\begin{center}
\noindent \em{Communicated by Guido De Philippis}
\end{center}
\vspace{1cm}


\noindent \textbf{Abstract.} \textit{We study compensation phenomena for fields satisfying both a pointwise and a linear differential constraint. This 
    effect takes the form of nonlinear elliptic estimates, where constraining the values of the field to lie in a cone compensates for the lack of ellipticity of the differential operator. We give a series of new examples of this phenomenon
    for a geometric class of cones and operators such as the divergence or the curl.
    One of our main findings is that the maximal gain of integrability is tied to both the differential operator and the cone, contradicting in particular a recent conjecture from  \cite{Arroyo-Rabasa2021a}. This 
    extends the recent theory of compensated integrability due to D.\ Serre. In particular, we find a new family of integrands that are $\Div$-quasiconcave  under convex constraints.}
\vskip0.3cm

\vspace{0.5cm}


\tableofcontents
\section{Introduction}

\subsection{Nonlinear elliptic estimates}\label{sec:results}

The interactions between nonlinearities, uniform estimates and weak convergence lie at the heart of many problems in the modern theory of nonlinear PDE. Famously, the Sobolev embedding theorem allows one to deduce uniform higher integrability for a sequence of functions whose gradients are bounded in an appropriate space, often leading to uniform bounds and weak (or strong) convergence for solutions of nonlinear PDE. In many cases, however, the full gradient may not be controlled, but instead one only has bounds on a weaker, non-elliptic operator, such as the divergence or the curl, for which general Sobolev inequalities fail. In this paper, we study higher integrability under the assumptions of boundedness of such a weaker operator together with a compensating pointwise constraint.

To be more precise, the main theme of this paper is nonlinear estimates of the form
\begin{align}\label{eq:main_inequ}
\|v\|_{L^q(\B^n)}\leq C\|\A v\|_{L^p(\B^n)}\quad\text{for }v\in C_c^\infty(\B^n,\mathcal K),
\end{align}
where $v\colon \R^n \to \V$ satisfies a pointwise constraint $$v\in \mathcal K, \qquad \mathcal K\subseteq \mathbb V \text{ is a closed cone},$$ and $\mathcal{A}$ is a linear, constant coefficient differential operator $\A v\colon \R^n\to \W$ for some  finite-dimensional vector spaces $\mathbb V$ and $\mathbb W$.
We emphasize that, in \eqref{eq:main_inequ}, the nonlinearity is in the constraint set $\mathcal K$. Here and throughout $\B^n$ denotes the open unit ball. If \eqref{eq:main_inequ} holds, we expect that
\begin{equation}
    \label{eq:nowavecone}
    \mathcal K \cap \Lambda_{\mathcal A} = \{0\}, \qquad \text{where } \Lambda_{\cala}\equiv \bigcup_{\xi \in \mathbb R^{n}\backslash\{0\}}\ker \cala(\xi)
\end{equation}
and $\A(\xi)$ denotes the principal symbol of $\A$. The set $\Lambda_\A$ is known as the \textit{wave cone} of $\A$, and it describes the set of directions along which the operator is not elliptic. Since \eqref{eq:main_inequ} requires \eqref{eq:nowavecone}, which is a type of ellipticity condition of $\A$ with respect to $\mathcal K$, we refer to \eqref{eq:main_inequ} as a \textit{nonlinear elliptic estimate} for $\A$.

The simplest instance of \eqref{eq:main_inequ} is when $\mathcal K$ is  a \textit{linear} space and, in this situation, \eqref{eq:nowavecone} yields genuine ellipticity \cite[Lemma 2.7]{Muller1999a}. It goes back to the work of Calder\'on and Zygmund that, if $1<p<\infty$ and $q\leq\frac{np}{n-\ell p}$, where $\ell$ is the order of $\A$, then \eqref{eq:nowavecone} implies \eqref{eq:main_inequ}. The case $p=1$, where the Calder\'on--Zygmund theory breaks down, is surprisingly subtle, having been settled only much more recently by Van Schaftingen \cite{VanSchaftingen2011}, after the pioneering work of Bourgain and Brezis \cite{Bourgain2007}. Besides ellipticity, the estimate holds if in addition $\A$ satisfies an algebraic condition known as \textit{cancellation}.
In addition, also when $p=1$, similar estimates were recently  obtained under the assumption that $\mathcal K$ is a perturbation of a linear space \cite{Arroyo-Rabasa2021a}.

The methods developed to study \eqref{eq:main_inequ} in the case where $\mathcal K$ is linear (or close to linear) make fundamental use of the Fourier transform, which interacts poorly with pointwise constraints. Thus, in order to study the genuinely nonlinear case we are interested in here, completely different tools are needed.

To the best of our knowledge the only known instance of \eqref{eq:main_inequ} when $\mathcal K$ is genuinely nonlinear is due to Serre \cite{Serre2018a,Serre2019}.
He considered 
$\cala=\Div$ and $\mathcal K$ 
the set $\textup{Sym}^+_n$ of symmetric semi-positive definite matrices and
showed that
\begin{align}\label{eq:serre}
    \|(\det A)^{\frac 1 n}\|_{L^{\frac{n}{n-1}}(\R^n)}\leq C\|\Div A\|_{L^1(\R^n)}\quad\text{for } A\in C_c^\infty(\R^n,\mathrm{Sym}^+_n).
\end{align}
In particular, for closed cones contained in the interior of $\textup{Sym}^+_n$  we recover \eqref{eq:main_inequ} with $p=1$ and $q=\frac{n}{n-1}$. Subsequently, De Rosa--Tione showed in \cite{DeRosa2019a} that the exponent $\frac{n}{n-1}$ corresponds to the \textit{maximal gain of integrability}, 
in the sense that we cannot take  $q>\frac{n}{n-1}$, even when $p=\infty$, see already \eqref{eq:qmax} below. We also refer the reader to \cite{Serre2018a,Serre2019,Serre2021,Serre2019a} for applications of \eqref{eq:serre} to hyperbolic problems and to \cite{Bate2019} for a related result when $n=2$.

In this paper we succeed in finding new examples of nonlinear elliptic estimates \eqref{eq:main_inequ} for non-elliptic operators, such as $\Div$ or $\Curl$.
The first of our main results is the following: 
\begin{bigthm}\label{thm:det2dintro}
Let $\mathcal K\equiv \{A\in \R^{2\times 2}: a_{11},a_{22}\geq 0\}$ and let $A_1,A_2$ be the rows of $A\in \R^{2\times 2}$. Then
$$\int_{\R^2}\det A \dif x \leq \|\di A_1\|_{L^1(\R^2)}\|\di A_2\|_{L^1(\R^2)}\quad 
\text{ for all } A\in C^\infty_c(\R^2,\mathcal K).$$
In particular, writing $Q_2^+(K)=\{A\in \R^{2\times 2}: \|A\|^2\leq K \det A\}$, we have
\begin{equation*}
\|A \|_{L^2(\R^2)}^2 \leq C\|\di A_1\|_{L^1(\R^2)}\|\di A_2\|_{L^1(\R^2)}\quad \text{ for all } A\in C_c^\infty(\R^2,\mathcal K\cap Q_2^+(K)).
\end{equation*}
\end{bigthm}
Noting that $\|\di A_1\|_{L^1}\|\di A_2\|_{L^1}\leq \|\Div A\|_{L^1}^2$, Theorem \ref{thm:det2dintro} should be compared with  Serre's inequality \eqref{eq:serre} above. 
In particular, Theorem \ref{thm:det2dintro} shows that, in the two dimensional case, \eqref{eq:serre} holds for a larger cone which is not necessarily contained in the space of symmetric matrices.
The cone in Theorem \ref{thm:det2dintro} is in some sense optimal, as the above estimates do not hold in the absence of one-sided constraints. However, for $p>1$, such one-sided constraints are not necessary, and we have
\begin{equation}
    \label{eq:strongdiv}
\|A \|_{L^{p^*}(\R^2)} \leq C\|\Div  A\|_{L^p(\R^2)}\quad \text{ for all } A\in C_c^\infty(\R^2, Q_2^+(K))
\end{equation}
if and only if
\begin{equation*}
    \label{eq:prangestrongdiv}
    1<p<\frac{2K}{2K-1},
\end{equation*}
see Proposition \ref{prop:strongdiv} below. Moreover, not even the estimate
$$\|A\|_{L^{\frac{2K}{2K-1}}(\mathbb B^2)}\leq C \|\Div A \|_{L^\infty(\mathbb B^2)} \quad \text{ for all } A\in C_c^\infty(\mathbb B^2, Q_2^+(K))$$
holds.
Note that if $K>1$ then $\frac{2K}{2K-1}<2$ so \eqref{eq:strongdiv} does not hold in the range $1<p<n = 2$. 

The above shows that, while pointwise constraints may allow elliptic-type estimates for non-elliptic operators, the gain of integrability fails in general to be as good as in the linear elliptic theory: there is a number $q_{\max}=q_{\max}(\A,\mathcal K)$ such that 
$$q>q_{\max} \implies \text{there is no } C \text{ such that } \|v\|_{L^q(\B^n)}  \leq C \|\A v\|_{L^\infty(\B^n)} \text{ for } v\in C^\infty_c(\B^n,\mathcal K).$$
We emphasize that $q_{\max}$ depends on both $\mathcal A$ and $\mathcal K$ and we refer to it as the \textit{maximal gain of integrability} induced by $\A$ on $\mathcal K$. Thus
\begin{equation}
    \label{eq:qmax}
q_{\max}(\Div, \textup{Sym}^+_n) = \frac{n}{n-1}, \qquad q_{\max} (\Div, Q_2^+(K))   = \frac{2K}{2K-1}.
\end{equation}

After this paper was completed, we learned that De Philippis--Pigati had independently proved Theorem \ref{thm:det2dintro} in order to establish a Michael--Simon inequality for anisotropic varifolds \cite{DePhilippis2022}. Moreover, in the very interesting paper \cite{Colombo2022}, Colombo--Tione have shown that Theorem \ref{thm:det2dintro} also yields regularity of certain very weak solutions of the $p$-Laplace equation.

Our next result shows that, in higher dimensions, there are stronger estimates for the divergence in smaller convex cones. To be precise, we consider the G\r{a}rding cones
$$\Gamma_k=\{A\in \textup{Sym}_n\,\colon \,F_j(A)>0\text{ for all }j=1,\ldots,k\},$$
where $F_j(A)$ is the following normalization of the $j$-th elementary symmetric polynomial on the vector of eigenvalues $\lambda(A)$ of $A$:
\begin{equation}
    \label{eq:defFk}
F_j(A)= \frac{1}{{n \choose j}}\sum_{i_1<\cdots<i_j}\lambda_{i_1}(A)\cdots\lambda_{i_j}(A).
\end{equation}
The G\r arding cones, introduced in the seminal work \cite{Garding1959}, play an important role in the theory of fully nonlinear elliptic equations \cite{Caffarelli1985,Wang2009}, of quasiconvex functions \cite{Faraco2008} and of curvature flows in differential geometry, for example in Huisken--Sinestrari \cite{Huisken1999} (see also \cite{Gerhardt1990,Urbas1990}). Besides the G\r{a}rding cones, we will consider their dual convex cones
$$\Gamma_k^*=\{A\in\textup{Sym}_n\,\colon\,\langle A,B\rangle\geq 0\text{ for all }B\in\Gamma_k\}.$$
For instance, we have
\begin{gather*}
\Gamma_n=\Gamma_n^*=\textup{Sym}^+_n\qquad \text{and } \qquad
\Gamma_1=\{A\in \textup{Sym}_n\colon \tr(A)\geq 0\},\quad  \Gamma_1^*=\{t I_n \colon t\geq 0\}.
\end{gather*}
The dual cones have been exploited by Krylov \cite{Krylov1987} (cf.~\cite{Ivochkina2012}) in the study of Bellman equations and by Kuo--Trudinger \cite{Kuo2007} to derive Alexandrov--Bakelman--Pucci type estimates.
We refer the reader to Section~\ref{sec:cones} below for further properties of these cones. 

In the next statement and throughout we write $p^*=\frac{np}{n-p}$ for the Sobolev exponent of $p$.

\begin{bigthm}\label{thm:maindiv}
          Let  $n\geq2$, $k\in\{2,\ldots,n\}$, $p^*>\frac{k}{k-1}$ (or $p^*\geq\frac{n}{n-1}$ if $k=n$), and let $\mathcal K$ be a closed, convex cone contained in $\textup{int}\,\Gamma_k^*\cup\{0\}$. Then there exists $C=C(\mathcal K)>0$ such that
     $$\|A\|_{L^{\frac{k}{k-1}}(\B^n)}\leq C\|\Div A\|_{L^p(\B^n)}\quad 
     \text{for all } A\in C_c^\infty(\B^n, \mathcal K).$$
     Moreover, the exponent on the left hand side is the best possible, since for each $\e>0$ there is a cone $\mathcal K^\e \subset \mathrm{int}\, \Gamma_k^*\cup\{0\}$ such that the following estimate fails for any $C>0$:
      $$\|A\|_{L^{\frac{k}{k-1}+\e}(\B^n)}\leq C\|\Div A\|_{L^\infty(\B^n)}\quad \text{for all } A\in C^\infty_c(\B^n,\mathcal K^\e).$$
      In addition, for $k<n$, the 
        restriction to the subcritical range $p^*>\frac{k}{k-1}$ is necessary, 
      as the analogue of \eqref{eq:serre} with the appropriate nonlinear functional replacing $\det$ does not hold. 
\end{bigthm}

Theorem \ref{thm:maindiv} provides a \textit{discrete interpolation} between  Serre's inequality, when $k=n$, and the classical Sobolev embedding $W^{1,p}_0(\B^n)\hookrightarrow L^\infty(\B^n)$ for $p>n$, when $k=1$. Indeed, as $\Gamma_1^*$ is a half-line through the identity, fields constrained to it can be identified with scalar fields, and the divergence can be identified with the gradient. As $k$ decreases, the cones $\Gamma_k^*$ also decrease, and we approach the scalar case, getting higher and higher integrability. We make this point more precise: if we interpolate the exponent $k$ of the cone $\Gamma_k$, we obtain the parameter $\theta$ solving
$$
\frac{1}{k}=\frac{1-\theta}{n}+\frac{\theta}{1}.
$$
Suppose now that we use the same $\theta$ to interpolate the (critical) exponents $q(k)=\tfrac{k}{k-1}$ on the LHS, and  $p(k)=\tfrac{nk}{nk-n+k}$ on the RHS. One can easily check that
$$
\frac{1}{q(k)}=\frac{1-\theta}{q(n)}+\frac{\theta}{q(1)}, \quad\frac{1}{p(k)}=\frac{1-\theta}{p(n)}+\frac{\theta}{p(1)}.
$$
Therefore the same fraction $\theta$ that we used to ``interpolate'' between $\Gamma_1$ and $\Gamma_n$ leads precisely to the interpolation of the exponents on each side of the inequality.

{Moreover, it should be noted that the estimate of Theorem B degenerates as the cone $\mathcal{K}$ expands towards the boundary of $\Gamma_k^*$; the closer the cone is to the boundary, the larger the constant $C$. Similarly, in the non-inequality of the theorem, the smaller $\eps$ is, the larger the corresponding $\mathcal K^\e$ will be; cf.\ Propositions \ref{prop:Div} and \ref{prop:optimal_div}.}

In contrast to Theorem \ref{thm:det2dintro}, we cannot obtain global estimates since we are confined to subcritical exponents $\tfrac{k}{k-1}<p^*$ if $k<n$. This can be seen from the rescaling
 $$\|A\|_{L^{\frac{k}{k-1}}(r\B^n)}\leq Cr^{n(\frac{k-1}{k}-\frac{1}{p^*})}\|\Div A\|_{L^p(r\B^n)}\quad 
     \text{for all } A\in C_c^\infty(r\B^n, \mathcal K),$$
     where clearly the constant blows up as $r\rightarrow\infty$. This is not particular to our nonlinear estimates, it is a feature solely due to homogeneity and the failure of the end-point estimate (see Proposition \ref{prop:endpointk<n}).

Through a nonlinear duality argument, Theorem \ref{thm:maindiv} yields elliptic estimates for the curl:

\begin{bigthm}\label{thm:maincurl}
     Let  $n\geq2$, $k\in\{2,\ldots,n\}$, $p^*>k$ (or $p^* =n$ if $k=n$) and let $\mathcal K$ be a closed, convex cone contained in $\textup{int}\,\Gamma_k\cup\{0\}$. Then there exists $C=C(\mathcal K)>0$ such that
     $$\|A\|_{L^{k}(\B^n)}\leq C\|\Curl A\|_{L^p(\B^n)} \quad \text{for all }A\in C_c^\infty(\B^n,\mathcal K),$$
     and in fact we have the additional improvement
     $$\|A\|_{L^{k}\log L(\B^n)}\leq C\|\Curl A\|_{L^p(\B^n)} \quad \text{for all }A\in C_c^\infty(\B^n,\mathcal K).$$
     Moreover, the exponent on the left hand side is the best possible, since for each $\e>0$ there is a cone $\mathcal K^\e \subset \mathrm{int}\, \Gamma_k\cup\{0\}$ such that the following estimate fails:
      $$\|A\|_{L^{k+\e}(\B^n)}\leq C\|\Curl A\|_{L^\infty(\B^n)}\quad \text{for all } A\in C^\infty_c(\B^n,\mathcal K^\e).$$
\end{bigthm}

We emphasize the nonlinear duality between Theorems \ref{thm:maindiv} and \ref{thm:maincurl}: on the cone $\Gamma_k$ (respectively $\Gamma_k^*$) with operator $\Curl$ (respectively $\Div$), the maximal gain of integrability is $k$ (respectively $k'=\frac{k}{k-1}$). As in \eqref{eq:strongdiv}, an interesting feature of Theorems~\ref{thm:maindiv} and \ref{thm:maincurl} is that the exponents $\frac{k}{k-1}$ and $k$ on the left-hand side, which are optimal, depend on the cone of interest and are \emph{strictly} smaller than the Sobolev exponent $p^*=\frac{np}{n-p}$ for all $p$ in the ranges given, just as in \eqref{eq:qmax}.

The next theorem is a higher order variant of Theorem \ref{thm:maindiv}, and in this case we are able to obtain inequalities in a \textit{full range of exponents}, including the critical case. We let $p^{**}=\frac{np}{n-2p}$ denote the second Sobolev exponent and, by convention, when $p\geq\frac{n}{2}$, we let $p^{**}\equiv\infty$. Below constants $C=C(\mathcal K)$ depend on closed convex cones $\mathcal K$.

\begin{bigthm}\label{thm:div^2}
     Let  $n\geq2$, $k\in\{2,\ldots,n\}$, and let $\mathcal K$ be a closed convex cone contained in $\mathrm{int}\,\Gamma_k^*\cup\{0\}$. Then, whenever $p>1$ and $1\leq q\leq\min\{p^{**},\frac{k}{k-1}\}$,
     \begin{equation}
         \label{eq:estimatediv^2}
             \|A\|_{L^{q}(\B^n)}\leq C\|\Div^2A\|_{L^p(\B^n)}\quad\text{for all }A\in C_c^\infty(\B^n,\mathcal K).
      \end{equation}
If $k> \frac{n}{2}$, so that $\frac{k}{k-1}<1^{**}$, we also have $$
\|A\|_{L^{\frac{k}{k-1}}(\B^n)}\leq C\|\Div^2A\|_{L^1(\B^n)}\quad\text{for all }A\in C_c^\infty(\B^n,\mathcal K).
$$
     Moreover, the exponent $\frac{k}{k-1}$ on the left hand side is the best possible, since for each $\e>0$ there is a cone $\mathcal K^\e \subset \mathrm{int}\, \Gamma_k\cup\{0\}$ such that the following estimate fails:
      $$\|A\|_{L^{\frac{k}{k-1}+\e}(\B^n)}\leq C\|\Div^2 A\|_{L^\infty(\B^n)}\quad \text{for all } A\in C^\infty_c(\B^n,\mathcal K^\e).$$
\end{bigthm}
In particular, when $n=k=2$ and $p=1$, Theorem \ref{thm:div^2} disproves \cite[Conjecture 1.9]{Arroyo-Rabasa2021a}, which predicts that \eqref{eq:estimatediv^2} holds whenever $q<1^{**}=\infty$. A result similar to the one in Theorem~\ref{thm:div^2}, featuring the less common operator $\A=\Div^2-\Delta\tr$, is presented in Subsection~\ref{c}.

A rather simple but nonetheless instructive example is obtained by considering the operator $\A=(\Div, \Curl \circ \cdot^\top)$. Although its restriction to symmetric matrices is elliptic, it is not difficult to verify that $\A$ is not elliptic when acting on the full space $\R^{n\times n}$. We have:

\begin{bigthm}\label{thm:degen}
    Let  $\mathcal K$ be a closed convex cone contained in $\{A\in \R^{n\times n}:\tr(A)> 0\}\cup\{0\}$. For $1\leq p <n,$ there exists $C=C(p,\mathcal K)>0$ such that
    $$ 
\|A\|_{L^{p^*}(\R^n)}\leq C\left(\|\Div A\|_{L^p(\R^n)}+\|\Curl (A^\top) \|_{L^p(\R^n)}\right) \quad \text{ for all } A\in C^\infty_c(\R^n,\mathcal K).
$$
\end{bigthm}

Thus, in Theorem \ref{thm:degen}, contrary to Theorems \ref{thm:maindiv}, \ref{thm:maincurl} and \ref{thm:div^2}, we obtain an arbitrarily high gain of integrability, just as in the usual elliptic case.

Altogether, the previous results contradict the naive expectation that the inequality \eqref{eq:main_inequ} behaves like a Sobolev inequality, i.e.\ one cannot simply take $q=p^{*\ell}=\frac{np}{n-\ell p}$, where $\ell$ is the order of $\A$. Our examples, together with those of \cite{Arroyo-Rabasa2021a,DeRosa2019a,Serre2018a}, suggest that the phenomenon that we are witnessing can be described as follows:

\begin{conjecture}\label{conj:big}
Let $\A$ be an $\ell$-th order operator with constant coefficients and $\mathcal K^o\subset \V$ be an open convex cone such that $\mathcal K^o\cap\Lambda_\A=\emptyset$. 
There exists a critical exponent $$q_{\max}\equiv q_{\max}(\A,\mathcal K^o)\in[1,\infty]$$ such that, for all closed convex subcones $\mathcal{K}\subset\mathcal K^o\cup\{0\}$ and $1\leq p<n$, we have
$$q\leq q_{\max},\, q<p^{*\ell} \quad \implies \quad
 \|v\|_{L^{q}(\B^n)}\leq C\|\A v\|_{L^p(\B^n)}\quad\text{for all }v\in C_c^\infty(\B^n,\mathcal K),$$
 where $C=C(p,q,\A,\mathcal K)$.
Moreover, the inequality is optimal in the sense that, for any $\e>0$, there is a closed convex cone $\mathcal K^\e\subset \mathcal K^o\cup\{0\}$ such that the following estimate fails:
$$ \|v\|_{L^{q_{\max}+\e}(\B^n)}\leq C\|\A v\|_{L^\infty(\B^n)}\quad\text{for all }v\in C_c^\infty(\B^n,\mathcal K^\e).$$
\end{conjecture}
 For the pairs $(\A,\mathcal{K})$ discussed so far, the results above verify partially or completely Conjecture \ref{conj:big} in several cases, see already Table \ref{table} below.

We expect that a theory developed around our conjecture will prove useful as a weak convergence method for nonlinear PDE \cite{Evans1990}. 
We are partly motivated by Tartar's framework in continuum mechanics \cite{Tartar1979,Tartar1983}, whereby one couples nonlinear pointwise constraints, represented by $\mathcal K$ and corresponding to \textit{constitutive relations}, with the linear PDE constraint given by $\A$ and corresponding to \textit{balance relations}. 
The classical theorem of Vitali asserts that strong $L^q$-convergence is equivalent to \textit{lack of oscillations}, that is convergence in measure, and \textit{lack of $L^q$-concentrations}, that is $q$-uniform integrability. 
Our estimates concern the latter, a direction which has perhaps been less studied in modern PDE.

The maximal gain of integrability is particularly relevant for ruling out nontrivial concentration effects: while for $q<q_{\max}$ we have that $q$-uniform integrability follows from $L^{q_{\max}}$-boundedness, an improvement to $q_{\max}$-uniform integrability requires additional work. For instance, the endpoint uniform integrability is achieved in Theorem \ref{thm:maincurl} by obtaining an extra logarithm of integrability. Alternatively, one can impose in addition that the fields $v_j$ are $\A$-free and derive $L^{q_{\max}+\delta}$ bounds from reverse H\"older inequalities, see Corollaries \ref{cor:divUI}, \ref{cor:curlUI} and \ref{cor:divL2UI}.

\subsection{The case of exact constraints: quasiconcavity and optimality}\label{sec:qcintro}
In  the examples discussed so far, the convex cone under consideration has the structural property
\begin{equation}
    \label{eq:deffunction}
\mathrm{int}\,\mathcal K =\{G> 0\}, \quad G\colon \mathcal K \to [0,+\infty) \text{ is concave and positively 1-homogeneous},
\end{equation}
cf.\ Table \ref{table}.  For example, for $A\in \Gamma_k$ and  $B\in \Gamma_k^*$, we define $\rho_k(A)=F_k(A)^{\frac{1}{k}}$ and
$$\rho_k^*(B)=\inf\Big\{\frac{1}{n}\langle A,B\rangle:A\in\Gamma_k,\:\rho_k(A)=1\Big\},$$
and then we have $$\textup{int}\,\Gamma_k=\{\rho_k>0\} \qquad \text{ and } \qquad \textup{int}\,\Gamma_k^*=\{\rho_k^*>0\}.$$ Moreover, $\rho_k$ and $\rho_k^*$ are respectively concave on $\Gamma_k$ and $\Gamma_k^*$, see Section \ref{sec:symmetric} for further details.

Given a pair $(\A,\mathcal K)$ for which \eqref{eq:nowavecone} and \eqref{eq:deffunction} hold, a natural question is  whether there is a number $\alpha(\A,\mathcal K)>1$ such that $G^\alpha$ is $\Lambda_{\A}$-concave in $\mathcal K$, i.e.
$$G^\alpha \textup{ is concave on the segment } [A,B]\subset \mathcal K, \textup{ whenever } B-A\in \Lambda_{\A}.$$
Surprisingly, if we let $\alpha_{\max}$ be the largest exponent with the above property, for the examples above we have that
\begin{equation}
    \label{eq:a=q}
\alpha_{\max}(\A,\mathcal K)=q_{\max}(\A,\mathcal K).
\end{equation}

It has been known since the work of Tartar \cite{Tartar1979} that $\Lambda_{\A}$-concavity is necessary for $\cala$-quasiconcavity, which in general is a strictly stronger property \cite{Ball1981,Grabovsky2018,Sverak1992a,Tartar1979}. Nonetheless, for the  examples above we  manage to prove that $G^{\alpha_{\max}}$ is $\A$-quasiconcave when restricted to fields in $\mathcal K$, that is, identifying the torus $\mathbb T^n$  with $[0,1]^n$, we have
\begin{equation}
    \label{eq:quasiconcavity}
\int_{\mathbb T^n} G(v)^{\alpha_{\max}}\dif x \leq G\Big(\int_{\mathbb T^n} v \dif x\Big)^{\alpha_{\max}} \quad \text{for all } v\in C^\infty(\T^n,\mathcal K) \text{ with }\A v=0.
\end{equation}
By definition of $\alpha_{\max}$, this inequality does not hold for any exponent $\alpha>\alpha_{\max}$. Inequality \eqref{eq:quasiconcavity} provides a higher integrability estimate, \textit{with a sharp constant}, for $\A$-free fields constrained to $\mathcal K$. 
 Quasiconvexity and quasiconcavity are essential notions in the vectorial Calculus of Variations, intimately related to the weak semicontinuity of multiple integrals \cite{Acerbi1984,Fonseca1999,Meyers1965,Morrey1952}. To date, the interaction of quasiconvexity with pointwise constraints is not well understood, but see the contributions in \cite{Astala2022,Conti2015,DeRosa2019,Koumatos2016,Skipper2021}.

To be more specific, let us start with the case $(\A, \mathcal K)=(\Div, \Gamma_k^*)$, where \eqref{eq:deffunction} holds for $G=\rho_k^*$. When $k=n$, $\rho_n^*=\det^{1/n}$ and Serre \cite{Serre2018a} obtained the quasiconcavity inequality
\begin{align}\label{eq:serre_qc}
    \int_{\mathbb T^n}(\det A)^{\frac{1}{n-1}}\dif x\leq \det\left(\int_{\T^n}A\dif x\right)^{\frac{1}{n-1}}\quad \text{for }A\in C^\infty(\T^n,\mathrm{Sym}^+_n) \text{ with }\Div A=0.
\end{align}
We succeed in extending this inequality to all natural numbers $k\leq n$, obtaining new examples of restricted $\Div$-quasiconcave integrands: 
\begin{bigthm}\label{thm:main_qc}
     Let $2\leq k\leq n$. We have that
$$
\int_{\mathbb T^n}\rho_k^*(A)^{\frac{k}{k-1}} \dif x\leq \rho_k^*\left(\int_{\mathbb T^n} A\dif x\right)^{{k}/{(k-1)}}\quad \text{for }A\in C^\infty(\T^n,\Gamma_k^*)\text{ with }\Div A=0.
$$
The exponent $\frac{k}{k-1}$ is sharp, as $(\rho_k^*)^\alpha$ is not $\Lambda_{\Div}$-concave if $\alpha > \frac{k}{k-1}$.
\end{bigthm}
Equality may be achieved in Theorem \ref{thm:main_qc}, see Section \ref{sec:qc} for a characterisation of this case. Moreover, Theorem \ref{thm:main_qc} trivially implies a similar result for $(\A,\mathcal K)=(\Div^2,\Gamma_k^*)$. 

Let us now turn to the case $(\A, \mathcal K)=(\Curl, \Gamma_k)$, where \eqref{eq:deffunction} holds for  $G=\rho_k$. The functions $F_k=\rho_k^k$ are \textit{null Lagrangians} \cite{Ball1981}, that is,
$$\int_{\T^n} F_k(A)\,\dif x=F_k\Big(\int_{\T^n}  A\,\dif x\Big)\quad \text{ for all }A\in C^\infty(\T^n,\textup{Sym}_n)\text{ with } \Curl A=0.$$
Equivalently, both $F_k$ and $-F_k$ are $\Curl$-quasiaffine. This gives a simple proof of the fact that $\alpha_{\max}(\Curl, \Gamma_k)=k$ and of the corresponding quasiconcavity of $F_k$. Thus \eqref{eq:a=q} also holds in this case.
A related result was obtained previously by Faraco--Zhong \cite{Faraco2003}, who proved that the integrands $1_{\Gamma_k} F_k$ are $\Curl$-quasiconvex, where $1_{\Gamma_k}$ denotes the characteristic function of $\Gamma_k$. Their result generalizes a previous result of \v Sver\'ak \cite{Sverak1992} for the case $k=n$, and  in fact it served as an early inspiration for considering the cones $\Gamma_k$.

The case $(\A,\mathcal K)=((\Div, \Curl\circ \cdot^\top),\{\tr\geq 0\})$ is even simpler. In this case, the wave cone is not spanning: $\textup{span}\, \Lambda_{\A}=\{\tr=0\}.$ In general, $\A$-free fields are constrained to $\textup{span}\,\Lambda_{\A}$ and thus, in this case, there are no $\A$-free fields taking values in $\mathcal K$. Thus $\tr^\alpha$ is trivially $\A$-concave on $\mathcal K$ for any $\alpha\geq 1$, and so $\alpha_{\max}=+\infty$, confirming  \eqref{eq:a=q}. 

We conclude this collection of examples by returning to the example discussed in \eqref{eq:strongdiv}, in which case $(\mathcal A, \mathcal K)=(\Div, Q_2^+(K))$. The Burkholder integrand ${B_K\colon\R^{2\times 2}\to \R}$ introduced in \cite{Burkholder1984}, see also \cite{Iwaniec2002}, and  defined  by $${B_K(A)=( K\det A -\|A\|^2)^{\frac{K-1}{2K}}\|A\|^{\frac 1 K}}$$ is 1-homogeneous and satisfies $Q_2^+(K) = \{B_K\geq 0\}$, but this cone is not convex. Nonetheless, it is well-known that $B_K^{\alpha}$ is $\Lambda_{\Curl}$-concave for $\alpha \leq \alpha_{\max}=q_{\max} = \frac{2K}{K-1}$, and it is a long-standing conjecture \cite{Iwaniec2002} that $B_K^{\alpha_{\max}}$ is $\Curl$-quasiconcave. Parallel to the other examples in this paper, it was recently shown in \cite{Astala2022} that the restriction of $B_K^{\alpha_{\max}}$ to $Q_2^+(K)$ is $\Curl$-quasiconcave, see also \cite{Astala1994,Astala2012,GuerraKristensen2021}.

We collect the key facts discussed in Sections \ref{sec:results} and \ref{sec:qcintro} for each of our examples:
\renewcommand{\arraystretch}{1.4}
\begin{table}[h]
\centering
 \begin{tabular}{|c |c |c| c| c|} 
 \hline 
 $\cala$ & $\Lambda_\cala$ & $\mathcal K=\{G\geq 0\}$ & $G$ & $q_{\max}$ \\ [0.5ex] 
 \hline\hline
 $\Div$ & $\{A\in \R^{2\times 2}: \textup{rank}\,A \leq 1\}$ & $Q_2^+(K)$ & $B_K$ & $\frac{2K}{K-1}$\\ \hline
 $\Div$ & $\{A \in \textup{Sym}_n:\textup{rank}\,A\leq n-1\}$ & $\Gamma_k^*$ & $\rho_k^*$ & $\frac{k}{k-1}$ \\ \hline
  $\Curl$ & $\{A \in \textup{Sym}_n:\textup{rank}\,A\leq 1\}$ &  $\Gamma_k$ & $\rho_k$& $k$ \\ \hline
 $\Div^2$ & ${\textup{Sym}_n\setminus\mathrm{int} (\textup{Sym}^+_n\cup\textup{Sym}^-_n)}$ & $\Gamma_k^*$ & $\rho_k^*$ & $\frac{k}{k-1}$ \\ \hline
 $(\Div, \Curl\circ\cdot^\top)$ & $\{A \in \R^{n\times n}:\textup{rank}A\leq 1, \tr(A)=0\}$ & $\{\tr\geq 0\}$ & $\tr$ & $\infty$\\ [1ex] 
 \hline
 \end{tabular}
\caption{ \label{table} A summary of the examples in Theorems \ref{thm:maindiv} to \ref{thm:main_qc}}
\end{table}

\subsection{Double cones and non-convex constraints}
With the exception of Theorem \ref{thm:det2dintro} and inequality \eqref{eq:strongdiv}, all of the results stated above concern smooth fields constrained to \textit{convex} cones. Using the convexity of the cones, the smoothness assumption may be relaxed by a simple mollification argument, see Remark \ref{rmk:roughfields} below. It thus appears natural to ask whether or not such estimates hold for ``double'', non-convex cones, such as $\textup{Sym}^+_n\cup\textup{Sym}^-_n$. While we do not give an answer to such a question here, we provide some evidence that the situation in this case is more complicated than the convex setting.

To fix ideas, consider Serre's inequality \eqref{eq:serre}. We look for an estimate
\begin{align*}
\|A\|_{L^{n/(n-1)}(\R^n)}\leq C\|\Div A\|_{\mathcal{M}(\R^n)}\quad\text{for all }A\in \textup{BV}^{\Div}(\R^n,\mathcal K_n),
\end{align*}
where $\mathcal K_n\subset \text{int}(\mathrm{Sym}^+_n\cup\,\mathrm{Sym}^-_n)\cup\{0\}$ and $\textup{BV}^{\Div}=\{A\in L^1(\R^n,\R^{n\times n}):\Div A\in\mathcal{M}(\R^n)\}$. However, we show in Proposition \ref{prop:nodensity} below the following:
\begin{bigprop}
The space $C^\infty([-1,1]^2,\mathcal K_2)$ is not strictly dense in $\textup{BV}^{\Div}([-1,1]^2,\mathcal K_2)$.
\end{bigprop}
Thus, in order to prove any estimate for functions in this last space, one \textit{cannot} hope to prove such an estimate for smooth functions and then pass to a limit (as we do in the case of convex constraint), and instead one must work directly with measures. 

In addition, the results of \v Sver\'ak \cite{Sverak1992} and more generally of Faraco--Sz\'ekelyhidi \cite{Faraco2008} show that if $A\in L^2(\B^2, \mathcal K_2)$ satisfies $\Div A=0$ (or $\Curl A=0$) and suitable boundary conditions, then $A$ must be valued solely in either $\textup{Sym}^+_2$ or $\text{Sym}_2^-$, see Proposition \ref{prop:FSz}. In other words, there is a separation between the sets $\textup{Sym}^+_2$ and $\textup{Sym}^-_2$ for such fields.

The combined effect of separation for exact constraints and lack of density of smooth fields significantly complicates any attempt to study the compensated integrability phenomenon even in the simplest non-convex case of double cones.

\subsection{Summary, ideas of proof, and open problems}
The outlook towards double cones in the previous subsection shows that there are many difficult questions that one can ask in this direction. For the present paper, we focus on convex cones. To summarize, our main results verify the postulated Conjecture~\ref{conj:big} in a number of fundamental cases, which are quite different in nature.  We show that our inequalities sometimes behave as in the linear case, sometimes as in Serre's inequality \eqref{eq:serre}, sometimes  capturing endpoint exponents, sometimes not. Some connections with geometric function theory and quasiconvexity are drawn. In particular, new nonlinear quantities $\rho_k^*$ are linked with both quasiconvexity and higher integrability. We see that the endpoint exponents for $\Lambda_{\Div}$-convexity of $(\rho_k^*)^\alpha$ are exactly the sharp higher integrability exponents, and we link maximal integrability and absence of concentration effects, through both Gehring's lemma and $L\log L$ bounds. This highlights the importance of finding \textit{maximal integrability} gains $q_{\max}$.

To aid the reader, we classify our estimates by method of proof here:
\begin{enumerate}
    \item Slicing: Theorem \ref{thm:det2dintro} (see also Theorem \ref{thm:slicing}, Proposition \ref{prop:curlinequality}). For these proofs, we were inspired by the original slicing argument to prove the classical Gagliardo--Nirenberg inequality. Interestingly, while this idea works quite well for both $\Div$ and $\Curl$ in 2D, it appears to adapt only for $\Curl$  in higher dimensions.
    \item\label{itm:nonlinear_dual} Nonlinear duality: Theorems \ref{thm:maindiv}, \ref{thm:div^2}, \ref{thm:main_qc},  and Propositions \ref{prop:curl} and \ref{prop:div^2L_2A} are similar in flavour to Serre's original argument from \cite{Serre2018a,Serre2019}. A crucial observation is that in Serre's case it holds that $\rho_n^*=\rho_n$ and $\Gamma_n^*=\Gamma_n$, so our idea of using convex duality is not visible there. This means that in Theorems \ref{thm:maindiv} and \ref{thm:div^2} we may bound $\rho_k^*(A)$ on the LHS, but not $\rho_k(A)$; see also Propositions \ref{prop:Div} and \ref{prop:Div^2}. 
    \item\label{itm:null_lag} Theorem \ref{thm:maincurl} is special: it follows from Theorem \ref{thm:maindiv} by use of our new characterization of null Lagrangians in Section \ref{sec:null-lag}, but it does \textit{not} follow from the kind of duality argument as in \ref{itm:nonlinear_dual} above, but rather a novel nonlinear duality.  We believe that the characterization of null Lagrangians in Corollary \ref{cor:null_lag} is, suprisingly, new.
    \item Theorem \ref{thm:degen} is also special: it follows from a simple argument and it stands to show that  estimates as in the linear case are possible in examples other than \cite{Arroyo-Rabasa2021a}. 
\end{enumerate}

We will conclude the introduction with a list of open problems that would help take this theory further.
The first  stems from the idea of interpolation following Theorem \ref{thm:maindiv}.
\begin{op}
    Let $s\in[1,n]$, $p^*>\tfrac{s}{s-1}$. Do there exist closed convex cones $\Gamma_s\subset\R^{n\times n}_{\sym}$ such that for closed cones $\mathcal K\subset\mathrm{int\,}\Gamma_s^*$ we have the estimate
    $$
    \|A\|_{L^{\frac{s}{s-1}}(\B^n)}\leq C\|\Div A\|_{L^p(\B^n)}\quad 
     \text{for all } A\in C_c^\infty(\B^n, \mathcal K)
    $$
    and we have optimality, $q_{\max}(\Gamma_s^*,\Div)=\tfrac{s}{s-1}$? When $s=k$, we retrieve the G\aa rding cones.
\end{op}
Each of the results listed in \ref{itm:nonlinear_dual} above is proved by bounding nonlinear quantities related to $\Gamma_k$ and $\Gamma_k^*$. There is one exception, coming from \ref{itm:null_lag} above:
\begin{op}
    Is it possible to bound $\rho_k(A)$ in Proposition \ref{prop:curlmain} as we bound $\rho_k^*(A)$ in Proposition \ref{prop:Div}?
\end{op}
We already highlighted this point in \ref{itm:null_lag} above. In the beginning of the section we also mentioned the concentration phenomena that are prevented by our estimates, see also Corollaries \ref{cor:divUI} and \ref{cor:curlUI}. In this respect we point out the very interesting recent paper \cite{DRT23}, where the concentration measures arising in connection with sequences bounded by \eqref{eq:serre} are analyzed. It would be interesting to understand this in connection with recent trace theorems in the linear case, see \cite{BDG,DieGm,GR17,GRVS1,GRVS2} and also the classic \cite{Adams2003}:
\begin{op}
    Can one prove trace inequalities
    $$
    \|v\|_{L^q(\dif \mu)}\leq C\|\A v\|_{L^p(\dif x)}\quad\text{for }v\in C_c^\infty(\B^n,\mathcal K),
    $$
    where $\mu\in\mathcal{M}^+$ satisfies a lower dimensional condition $\mu(B_r(x))\leq Cr^d$ for some $d<n$?
\end{op}
Finally, we remarked in Section \ref{sec:GNcurl2d} that the following is unclear:
\begin{op}
    Does Ornstein's non-inequality hold in the $\mathcal{K}$-constrained case?
\end{op}
We speculate that it does, but we were surprised that we could not find a proof using the numerous techniques in the literature; see Remark \ref{question:Ornstein} for details.

\subsection{Outline of the paper}
Section \ref{sec:notation} collects some notation that will be used throughout.

In Section \ref{sec:symmetric} we develop a duality theory for the key nonlinear objects that we will consider, which are the cones $\Gamma_k$ and their duals $\Gamma_k^*$. Although the cones $\Gamma_k$ are standard objects in the literature of degenerate nonlinear elliptic equations, see e.g.~\cite{Caffarelli1985,Wang2009}, their duals have been much less studied. In this section we also prove a new quantitative characterization of null Lagrangians, which is the second essential element in the nonlinear duality between Theorems~\ref{thm:maindiv} and \ref{thm:maincurl}. We also collect some relevant facts concerning weak solutions to $k$-Hessian equations which will be used throughout Sections \ref{sec:qc} to \ref{sec:otherexamples}.

In Section \ref{sec:GN} we state and prove Theorem \ref{thm:det2dintro}: the argument is based on the classical slicing method of Gagliardo and Nirenberg. Since the setting is particularly simple, we also take the opportunity to discuss a number of issues intrinsic to our setup, such as Ornstein's non-inequalities, the necessity of entry-wise constraints, and the non-density of smooth fields in the relevant function spaces when the cones are non-convex.

In Section \ref{sec:qc} we prove Theorem \ref{thm:main_qc}. The proof employs the duality theory developed in Section \ref{sec:symmetric} together with an auxiliary $k$-Hessian equation with periodic boundary condition. Besides being a result of independent interest, the proof of Theorem \ref{thm:main_qc} gives a simple instance of the strategy that we will employ in the following sections. 

In Section \ref{sec:div} we enrich the strategy of the previous section by employing estimates for auxiliary $k$-Hessian equations, and prove Theorems~\ref{thm:maindiv} and \ref{thm:div^2}. We also deduce higher integrability results for solenoidal fields under pointwise constraints.

In Section \ref{sec:curl} we prove Theorem~\ref{thm:maincurl}. 
We further obtain higher integrability statements for irrotational fields under pointwise constraints.

Finally, Section \ref{sec:otherexamples} addresses three more examples of operator-cone pairs for which we obtain elliptic estimates. The first example we consider yields Theorem \ref{thm:degen}. The other examples are given by
a second order operator related to the composition of $\Div$ and $\Curl$, 
and a weak sub-curl type operator, which only uses some components of $\Curl$.

\bigskip 

{\bf Acknowledgments.}
A.\,Guerra was supported by the Infosys Membership at the Institute for Advanced Study. M. Schrecker's  research is supported by the EPSRC Early Career Fellowship EP/S02218X/1. We thank Luc Nguyen for helpful discussions concerning $k$-Hessian equations and Denis Serre for his interesting comments. We also thank the  anonymous referees for the careful reading of the manuscript and comments, which greatly improved the quality of the paper.

\section{Notation}\label{sec:notation}

Throughout, we will work with the $L^p$ spaces, $1\leq p\leq\infty$ and will denote by $p'$ the H\"older conjugate of an exponent $p$. Given $p\in[1,\frac{n}{\ell})$, we write $p^{*\ell}=\frac{np}{n-\ell p}$ for its $\ell$-th order Sobolev exponent, which is denoted by $p^*$ or $p^{**}$ when $\ell=1$ or 2, respectively. We also define $(\frac{n}{\ell})^{*\ell}=\infty$. For $q\geq\frac{n}{n-1}$, we write 
$$
q_*=\frac{nq}{n+q}
$$
for the exponent such that $(q_*)^*=q$.

Given a matrix $A\in \R^{n\times n}$, we write $A=(a_{ij})_{i,j=1}^n$ and we denote by $A_i=(a_{ij})_{j=1}^n$ its rows. In particular, if $A\colon \R^n\to \R^{n\times n}$ is a matrix field, we have $\Div A=(\di A_i)_{i=1}^n$. 

We write $\textup{Diag}_n$ for the space of diagonal matrices in $\R^{n\times n}$,  $\textup{diag}(\la_1,\ldots,\la_n)$ for the diagonal matrix with entries $\la_1,\ldots,\la_n$,
$\textup{Sym}_n\equiv \{A\in \R^{n\times n}: A=A^\top\}$ for the space of symmetric matrices, and 
\begin{align}
\begin{split}
    & \textup{Sym}^+_n\equiv \{A\in \textup{Sym}_n: Av\cdot v\geq 0 \textup{ for all } v\in \R^n\},\\
    & \textup{Sym}^-_n\equiv \{A\in \textup{Sym}_n: Av\cdot v\leq 0 \textup{ for all } v\in \R^n\},
\end{split}
\label{eq:SPD}
\end{align}
for the cones of symmetric positive semi-definite and negative semi-definite matrices. We also use the notation
$$|A|\equiv \sqrt{\tr(AA^\top)}, 
\qquad \|A\|\equiv \max_{|v|=1} |Av|, \qquad $$
for the Euclidean and the operator norms of $A\in \R^{n\times n}.$ The Euclidean norm is induced from the inner product
$$\langle A,B\rangle \equiv \tr(AB^\top).$$
The cones of \textit{orientation-preserving} and \textit{orientation-reversing $K$-quasiconformal} matrices will also play an important role in this paper: given $K\geq 1$, we write
\begin{equation}
    \label{eq:quasiregmatrices}
    Q_n^\pm (K)\equiv \{A\in \R^{n\times n}: \|A\|^n \leq \pm K \det A\}.
\end{equation}
If $A\in Q_n^+(1)$ then we say that $A$ is \textit{conformal}, and if $A\in Q_n^-(1)$ it is said to be \textit{anti-conformal}.
When $n=2$, the cones $Q_2^\pm(1)$ are linear spaces and we may also write, see e.g.\ \cite[(16.33)]{Astala2009},
\begin{equation}
    \label{eq:auxquasiregular}
    Q_2^\pm(K)= \left\{A\in \R^{2\times 2}:|A|^2\leq \pm \Big(K+\frac 1 K\Big) \det A\right\}.
\end{equation}
Finally, for any function $F\colon \R^{n\times n}\to \R$, we will write $\nabla F(A)$ for the gradient of $F$ with respect to the matrix variables $A$, reserving the notation $\D$ for the gradient operator in the spatial variables on $\R^n$.

\section{Symmetric cones and $k$-Hessian equations}
\label{sec:symmetric}

In Section \ref{sec:cones} we develop a comprehensive duality theory for the symmetric cones. In particular, we establish some new duality results which go beyond the standard theory as described in \cite{Wang2009}. Section \ref{sec:null-lag} contains a new quantitative characterisation of null Lagrangians that will be essential for the nonlinear duality between $\Div$ and $\Curl$.  In Section \ref{sec:kHessian} we collect known results about existence and regularity of solutions to $k$-Hessian equations. The results of this section will be extensively used in Sections \ref{sec:qc}, \ref{sec:div} and \ref{sec:curl}.

\subsection{Symmetric polynomials and the cones $\Gamma_k$ and $\Gamma_k^*$}
\label{sec:cones}

We are interested in the cones
$$\Gamma_k\equiv \{\lambda\in \R^n:\sigma_j(\lambda)\geq 0 \text{ for } j=1,\dots, k\},$$
where 
\begin{equation}
    \label{eq:defsigmak}
    \sigma_k(\lambda)\equiv \sum_{i_1<\dots<i_k} \lambda_{i_1}\dots \lambda_{i_k}.
\end{equation}
For each $k=1,\dots,n$, the cone $\Gamma_k$ is clearly closed, convex and symmetric.
Note that $\Gamma_1=\{\lambda\in \R^n:\sum_i \lambda_i \geq 0\}$ is a half-space, while $\Gamma_n=\{\lambda \in \R^n: \lambda_i\geq 0 \text{ for all } i\}$ is the positive cone. In general, we have the following alternative characterizations of $\Gamma_k$, see for instance \cite{Ivochkina1983,Trudinger1999}:

\begin{lemma}[Characterization of $\Gamma_k$]
For $k=1,\dots, n$, we have
$$\Gamma_k=\{\lambda\in \R^n:0\leq \sigma_k(\lambda)\leq \sigma_k(\lambda+\eta) \text{ for all } \eta\in \Gamma_n\}.$$
Moreover, $\Gamma_k$ is the closure of the component of $\{\sigma_k>0\}$ which contains $\Gamma_n$.
\end{lemma}

Given a cone $\Gamma\subset \R^n$, its \textit{dual cone} is the closed, convex cone
$$\Gamma^*\equiv \{\mu\in \R^n:\lambda\cdot \mu\geq 0 \text{ for all }\lambda \in \Gamma\}.$$
For $l> k$ we have $\Gamma_l\subset\Gamma_k$ and thus $\Gamma_l^*\supset \Gamma_k^*$. It is easy to see that $\Gamma_1^*=\{t (1,\dots, 1):t\geq 0\}$ is a ray, while $\Gamma_n$ is self-dual, i.e.\ $\Gamma_n=\Gamma_n^*$. 

For our purposes, it is convenient to consider the functions $\rho_k$, defined for $\lambda \in \Gamma_k$ by
$$\rho_k(\lambda)\equiv \left(\frac{\sigma_k(\lambda)}{{n\choose k}}\right)^\frac 1 k$$
which are normalized so that $\rho_k(1,\dots,1)=1$. We set $\rho_k(\lambda)=-\infty$ if $\lambda\not \in \Gamma_k$.

We have defined the cones $\Gamma_k,\,\Gamma_k^*$ as subsets of $\R^n$. However, it is also natural to view them as subsets of $\text{Sym}_n$: we say that a symmetric matrix $A$ is in $\Gamma_k$ if $\lambda(A)\in \Gamma_k$, where $\lambda(A)$ is the vector of eigenvalues of $A$, arranged in descending order, and similarly for $\Gamma_k^*$. We also extend $\rho_k$ to $\text{Sym}_n$ by 
\beq\label{def:rhok}
\rho_k(A)\equiv \rho_k(\lambda(A)).
\eeq
Recalling \eqref{eq:defFk}, note that, for $A\in \Gamma_k$,
\beq\label{def:Fk}
F_k(A)= \rho_k(A)^k,
\eeq
where $F_k$ is, up to a multiplicative constant, the sum of the principal $k\times k$-minors of $A$. In particular, $F_k$ is a null Lagrangian (see Appendix \ref{app:LlogL}).

At this point it is reassuring to observe that the cones $\Gamma_k^*\subset \textup{Sym}_n$ are exactly what one would expect:

\begin{lemma}
\label{lemma:sanitycheck}
For $k=1,\dots, n,$ we have
$\Gamma_k^*=\{B\in  \textup{Sym}_n:\langle A,B\rangle \geq 0 \text{ for all } A\in \Gamma_k\}.$
\end{lemma}

\begin{proof}
By definition we have $\Gamma_k^*=\{B\in  \textup{Sym}_n: \lambda(A)\cdot\lambda(B) \geq 0 \text{ for all } A\in \Gamma_k\}$. Let $$\Sigma=\{B\in  \textup{Sym}_n:\langle A,B\rangle \geq 0 \text{ for all } A\in \Gamma_k\}.$$
For any $A\in \textup{Sym}_n$ we have $A \in \Gamma_k$ if and only if $QAQ^\top\in \Gamma_k$ for all $Q\in \textup{O}(n)$. It follows that the sets $\Gamma_k^*$ and $\Sigma$ have the same property. Since $\Gamma_k^*\cap\textup{Diag}_n=\Sigma\cap\textup{Diag}_n$, the conclusion follows.
\end{proof}

Associated to the dual cones $\Gamma_k^*\subset \R^n$, we have the dual functions 
\beq\label{def:rhokstar}
\rho_k^*(\mu)\equiv \inf\left\{\frac{\mu\cdot \lambda}{n}:\lambda\in \Gamma_k, \rho_k(\lambda)\geq 1\right\}
\eeq
which are defined for $\mu \in \Gamma_k^*$. The definition of $\rho_k^*$ can of course be extended to $\mu \not\in\R^n\backslash \Gamma_k^*$ but for such $\mu$ one always has $\rho_k^*(\mu)=-\infty$. As above, we also define $\rho_k^*(A)\equiv \rho_k^*(\lambda(A))$ for $A\in \textup{Sym}_n$.
The following classical result is a fundamental property of these functions and we refer the reader to \cite{Caffarelli1985} for a proof:
\begin{proposition}
\label{prop:concavity}
The function $\rho_k$ is concave in $\Gamma_k$, and so $\rho_k^*$ is concave in $\Gamma_k^*$.
\end{proposition}

Clearly $\rho_1^*(t \,\textup{diag}(1,\dots, 1))=t$, while $\rho_n^*=\rho_n$. The case $k=2$ also admits an explicit description as $\nabla F_2$ is linear: elementary calculations reveal that
\begin{align}
\label{eq:rho2*}
& \Gamma_2^*\equiv \bigg\{A\in \textup{Sym}_n:|A|\leq\frac{1}{\sqrt{n-1}}\tr A\bigg\},
& \rho_2^*(A)=\frac{1}{\sqrt n}\left((\tr A)^2 -(n-1)|A|^2\right)^\frac 1 2.
\end{align}

The next lemma shows that the families $\{\rho_k\}_{k=1}^n$ and $\{\rho_k^*\}_{k=1}^n$ are ordered: 

\begin{lemma}\label{lemma:rhokorder}
For all $1\leq l<k\leq n$, we have
$$\rho_k(A)\leq \rho_l(A)\text{ for }A\in\Gamma_k,\qquad \rho_l^*(A)\leq \rho_k^*(A)\text{ for }A\in\Gamma_l^*.$$
\end{lemma}

\begin{proof}
The first inequality is exactly Maclaurin's inequality, see e.g.\ \cite{Hardy1952}, and the second inequality follows from the first. Indeed, if $B\in\Gamma_k\subset \Gamma_l$ is such that $\rho_k(B)\geq 1$, then also $\rho_l(B)\geq 1$, so that
$$\rho_l^*(A)\leq \frac{\la(A)\cdot\la(B)}{n}.$$
Taking the infimum over $B\in \Gamma_k$ such that $\rho_k(B)\geq 1$, we have
$$\rho_l^*(A)\leq \rho_k^*(A),$$
as required.
\end{proof}

We have the following alternative characterizations of the dual functions:

\begin{proposition}
\label{prop:simplifiedrhok*}
Fix $k\in\{1, \dots, n\}$ and $B\in \Gamma_k^*\backslash \p \Gamma_k$. We have:
\begin{align*}
    \rho_k^*(B)=\min\left\{\frac{1}{n}\langle A,B\rangle:A\in\Gamma_k,\:\rho_k(A)=1\right\}.
\end{align*}
\end{proposition}

In order to prove the proposition, we first establish the following lemma:

\begin{lemma}\label{lemma:simplifiedrhok*}
For $k\in \{1,\dots, n\}$ and $B\in \Gamma_k^*$, we have
\begin{equation*}
\rho_k^*(B)=\inf\left\{\frac{\lambda(B) \cdot \lambda(A)}{n}:A\in \Gamma_k,\, \rho_k(A)= 1\right\}.
\end{equation*}
\end{lemma}

\begin{proof}
The case $k=1$ can be verified directly, so fix $k\geq 2$. We use the fact that for a convex subset $K\subseteq \R^n$ we have
$$\textup{conv}(\textup{rel-}\p K)= K $$
if and only if $K$ is not an affine subspace or the intersection of an affine subspace with a closed halfspace, see for instance \cite[Lemma 1.4.1]{Schneider2013}. 
 This identity can be applied to the set $K=\{A\in \Gamma_k: \rho_k(A)\geq 1\}$ as it follows easily from the definition of $\rho_k$ that, for $k\geq 2$, $K$ is neither affine nor a half-space. Moreover, by Proposition \ref{prop:concavity}, $K$ is convex. By definition, $\rho_k^*(B)$ is the infimum of the concave function $\ell(A)=\frac 1 n \lambda(B) \cdot \lambda(A)$ over the set $K=\textup{conv}(\p K)$, and so
$\rho_k^*(B)=\inf_{\p K} \ell$. Since we have $\p K=\{A\in \Gamma_k:\rho_k(A)=1\}$, the conclusion follows.
\end{proof}

\begin{proof}[Proof of Proposition \ref{prop:simplifiedrhok*}]
We claim that the infimum of $\langle A, B\rangle$ over the set of $A\in \Gamma_k$ such that $\rho_k(A)=1$ is attained. Indeed, according to von Neumann's trace inequality, see e.g.\ \cite{Mirsky1959}, we have the coercivity inequality
\begin{equation}
    \label{eq:coercivity}
    \lambda_\textup{min} (B)\|A\|=\lambda_\textup{min} (B)\lambda_\textup{max} (A)\leq \langle A, B\rangle.
\end{equation}
Note that $\lambda_{\min}(B)>0$ since $B$, being in the interior of $\Gamma_k^*\subset \Gamma_n$, is a positive definite matrix. The claim now follows from \eqref{eq:coercivity} and a compactness argument. 

Let $A_0$ be the point where the infimum of $\langle A, B\rangle$ over the set $\{A\in \Gamma_k:\rho_k(A)=1\}$ is attained; we necessarily have
$B=c \nabla \rho_k(A_0)$ for some Lagrange multiplier $c\in \R$. 
It follows that $A_0$ and $B$ are simultaneously diagonalizable, so we may assume that both of them are diagonal matrices. 
Therefore
$$\inf_{\substack{A\in\Gamma_k \\ \rho_k(A)=1}}\langle A,B\rangle=\langle A_0,B\rangle=\la(A_0)\cdot \la(B)\geq \inf_{\substack{A\in\Gamma_k \\ \rho_k(A)=1}}\la(A)\cdot \la(B).$$
To prove the reverse inequality we note that, as above, the right-hand is in fact attained at a matrix $A_1\in \Gamma_k$ with $\rho_k(A_1)$, which we may assume to be diagonal. Thus, if $Q\in \textup{O}(n)$ is such that $Q^\top B Q$ is a diagonal matrix, we have
$$\la(A_1)\cdot\la(B)=\langle A_1,Q^\top BQ\rangle=\tr(A_1Q^\top BQ)=\tr(Q^\top A_1QB)\geq \inf_{\substack{ A\in\Gamma_k \\ \rho_k(A)=1}}\langle  A,B\rangle,$$
since both $\Gamma_k$ and $\rho_k$ are invariant under the operation $A\mapsto Q^\top A Q$. The conclusion now follows from Lemma \ref{lemma:simplifiedrhok*}.
\end{proof}

As an immediate consequence of Proposition \ref{prop:simplifiedrhok*}, we obtain the main result of this subsection, which will play a key role in sections \ref{sec:qc} and \ref{sec:div}.

\begin{proposition}[Duality pairing]
\label{prop:dualitypair}
For $k=1,\dots, n$ and any $A\in \Gamma_k$, $B\in \Gamma_k^*$, we have
$$\rho_k(A)\rho_k^*(B)\leq \frac 1 n \langle A,B\rangle.$$
Moreover, for each $B\in \Gamma_k^*$ there is some $A\in \Gamma_k$ such that equality is attained.
\end{proposition}

\begin{proof}
 We may assume that $\rho_k(A), \rho_k^*(B)>0$ as otherwise, according to Lemma \ref{lemma:sanitycheck}, there is nothing to prove. In this case $B\in \Gamma_k^*\backslash\p \Gamma_k^*$ and, by homogeneity, we may also assume that $\rho_k(A)=1$.  The proposition now follows from Proposition \ref{prop:simplifiedrhok*}.
\end{proof}

The first part of Proposition \ref{prop:dualitypair} was  stated in \cite[Proposition 2.1]{Kuo2007}, although their proof omits a few details.

The following lemma establishes a useful relationship between $\rho_k$ and its dual function:

\begin{lemma}
\label{lemma:attainment}
For $A\in \Gamma_k\backslash \p\Gamma_k$ we have $\nabla \rho_k(A)\in \Gamma_k^*$ and 
$\rho_k^*(n \nabla \rho_k(A))=1.$
\end{lemma}

\begin{proof}
 Let $A,\tilde A \in \Gamma_k$ be arbitrary. Using the concavity of $\rho_k$ provided by Proposition \ref{prop:concavity} we have
$$\rho_k(\tilde A)-\rho_k(A)\leq \langle \nabla \rho_k(A), \tilde A-A\rangle = \langle \nabla \rho_k(A), \tilde A\rangle - \rho_k(A)$$
where the last equality follows from Euler's identity for homogeneous functions; thus
\begin{equation}
    \label{eq:welldef}
    \rho_k( \tilde A)\leq \langle \nabla \rho_k(A), \tilde A\rangle.
\end{equation}
We first note that \eqref{eq:welldef}  shows that $0 \leq \langle \nabla \rho_k(A), \tilde A\rangle$ for all $\tilde A \in \Gamma_k$: thus $\nabla \rho_k(A)\in \Gamma_k^*$.

Let us now assume that $\rho_k(\tilde A)=1$. Taking the infimum over such $\tilde A$ in \eqref{eq:welldef} gives $1\leq \rho_k^*(n\nabla\rho_k(A))$, where we have used the characterization of $\rho_k^*$ given by Proposition \ref{prop:simplifiedrhok*}.
To prove the opposite inequality, assume that $\rho_k(A)=1$; this is justified by homogeneity, since $A\in \Gamma_k\backslash \p \Gamma_k$. Note that, by definition,
$
\rho_k^*(n\nabla\rho_k(A))\leq \frac{1}{n}\langle  n\nabla\rho_k(A),A\rangle=\rho_k(A)=1.
$
The proof is complete.
\end{proof}

By combining Proposition \ref{prop:dualitypair} with Lemma \ref{lemma:attainment}, we obtain:

\begin{proposition}\label{prop:diffeo}
The map $\nabla F_k\colon \mathrm{int}\,\Gamma_k\to \mathrm{int}\,\Gamma_k^*$ is a diffeomorphism.
\end{proposition}

\begin{proof}
Clearly $F_k$ is smooth and so, to prove that it is a local diffeomorphism, it is enough to show that $\nabla F_k(A)$ is non-singular for all $A\in \Gamma_k\backslash \p \Gamma_k$. By Lemmas \ref{lemma:rhokorder} and \ref{lemma:attainment}, recalling \eqref{def:Fk}, we estimate
\begin{align*}
\det(\nabla F_k(A))^{\frac{1}{n}} =\rho_n^*(\nabla F_k(A))& =\frac{k}{n}\rho_k^{k-1}(A) \rho_n^*(n\nabla \rho_k(A))\\ & \geq \frac{k}{n}\rho_k^{k-1}(A) \rho_k^*(n \nabla \rho_k(A))=\frac{k}{n}\rho_k^{k-1}(A)>0,
\end{align*}
and so by the Inverse Function Theorem $\nabla F_k$ is a local diffeomorphism. It remains to prove that this map is bijective.

Let us begin with surjectivity. Applying Proposition \ref{prop:dualitypair}, for each $B\in \mathrm{int}\,\Gamma_k^*$, let $A\in \Gamma_k$ be such that $\rho_k(A)\rho_k^*(B)=\frac 1 n \langle A,B\rangle$. By Lemma \ref{lemma:attainment} and the Lagrange multiplier argument in the proof of Proposition \ref{prop:simplifiedrhok*}, we have
$$B=n \rho_k^*(B) \nabla \rho_k(A) = \frac{n}{k} \frac{\rho_k^*(B)}{\rho_k(A)^{k-1}}\nabla F_k(A)=\nabla F_k\bigg(\Big(\frac{n\rho_k^*(B)}{k}\Big)^{\frac{1}{k-1}}\frac{A}{\rho_k(A)}\bigg).$$

To prove injectivity, we note that it is sufficient to show that the set $\{\rho_k\geq 1\}$ is strictly convex, as then $\nabla\rho_k$ is the normal to $\{\rho_k=1\}$ and hence it is injective. By Lemma \ref{lemma:attainment}, if $\nabla F_k(A)=\nabla F_k(\tilde A)$ for $A, \tilde A \in \mathrm{int}\, \Gamma_k$, then $\rho_k(A)=\rho_k(\tilde A)$ and so $\nabla \rho_k(A)=\nabla \rho_k(\tilde A)$ as well. Thus the injectivity claim follows.

The strict convexity of $\{\rho_k\geq 1\}$ follows from  \cite[Theorem 5]{Garding1959} (see also the comment at end of page 964 there): a positive level set of a hyperbolic polynomial $P$ bounds a strictly convex region, modulo the lineality space $L_P$. In our case, $P=F_k$ and we have $L_{F_k}=\{0\}$: in G\r{a}rding's terminology, $F_k$ is complete, as it depends on all variables.
\end{proof}

\subsection{A quantitative characterisation of null Lagrangians}\label{sec:null-lag}

For an arbitrary minor $M$, the next proposition establishes a form of nonlinear duality between the divergence of $\nabla M$ and the curl of the matrix field on which it acts:

\begin{proposition}\label{prop:divrhok*}
Let $M\colon \R^{n\times n}\to \R$ be a fixed minor. There is $C=C(n)$ such that, for all $A\in C^1(\R^n,\R^{n\times n})$,
$$|\Div \nabla M(A) |\leq C|\Curl A| |\nabla^2 M(A)|.$$
\end{proposition}

Recall that we use the notation $\nabla$ for derivatives with respect to the $A$ variables and reserve $\D$ for spatial derivatives.

\begin{proof}
The proposition is most easily proved through differential forms. Let $A_i\in \R^n$ be the rows of the matrix field $A\in \R^{n\times n}$;  we think of $A_i$ as a 1-form given by $A_i=\sum_{j=1}^n a_{ij} \dif x^j$ and so we may identify $\Curl A$ with $(\dif A_i)_{i=1}^n$, since 
$$\dif A_i=\sum_{1\leq j<k\leq n} \Big(\frac{\p a_{ik}}{\p x^j}- \frac{\p a_{ij}}{\p x^k} \Big) \dif x^j \wedge \dif x^k.$$

We first deal with the case where $k=n$ and $M=\det$. We consider the linear map $\Lambda^{n-1} A\colon \Lambda^{n-1}(\R^n)\to \Lambda^{n-1}(\R^n)$, defined as usual on decomposable elements through
\begin{equation}
    \label{eq:exteriorpower}
\Lambda^{n-1}(A)(e_1\wedge \dots \wedge \widehat{e_i}\wedge \dots \wedge e_n) = A_1\wedge \dots \wedge \widehat{A_i}\wedge \dots \wedge A_n,
\end{equation}
where $\hat{\cdot}$ means that the corresponding term is omitted. 
We may identify the above $(n-1)$-form with the $i$-th row of $\textup{cof}\,A$, which as expected we denote by $( \textup{cof}\, A)_i$: indeed, writing $(e_i)$ for the canonical basis of $\R^n$,
\begin{equation}
    \label{eq:cofactoridentification}
A_1\wedge \dots \wedge \widehat{A_i} \wedge \dots \wedge A_n = \sum_{j=1}^n \det A_{ij} e_1\wedge \dots \wedge \widehat{e_j} \wedge \dots \wedge e_n,
\end{equation}
where $A_{ij}$ is the $(n-1)\times (n-1)$ matrix obtained from $A$ by deleting the $i$-th row and the $j$-th column. 
Taking exterior derivatives in \eqref{eq:exteriorpower}, we compute
\begin{align*}
  &\dif\left(\Lambda^{n-1}(A)(e_1\wedge \dots \wedge \widehat{e_i}\wedge \dots \wedge e_n)\right)=\\ 
  & \qquad = \sum_{j< i} (-1)^{j-1}A_1 \wedge \dots \wedge A_{j-1} \wedge \dif A_j \wedge A_{j+1} \wedge \dots \wedge \widehat{A_i} \wedge \dots \wedge A_n   \\
  & \qquad +\sum_{j>i} (-1)^j A_1 \wedge \dots  \wedge \widehat{A_i} \wedge \dots 
 \wedge A_{j-1} \wedge \dif A_j \wedge A_{j+1}\wedge \dots \wedge A_n.
\end{align*}
Thus, applying the triangle inequality, we obtain
$$|\di( \textup{cof}\, A)_i|= \left|\dif\left(\Lambda^{n-1}(A)(e_1\wedge \dots \wedge \widehat{e_i}\wedge \dots \wedge e_n)\right)\right|\leq (n-1) |\Curl A| |\nabla^2 \det (A)|,$$
where we used the fact that, for any indices $i_1<\dots<i_{n-2},$
$$|A_{i_1} \wedge \dots \wedge A_{i_{n-2}}| \leq |\nabla^2 \det(A)|.$$
This last inequality can be deduced as follows: similarly to \eqref{eq:cofactoridentification}, the left-hand side is the norm of the vector $M_{i_1\dots i_{n-2}}(A)$ of those $(n-2)\times (n-2)$ minors of $A$ which only depend on the $i_1,\dots, i_{n-2}$ rows of $A$; the right-hand side is the norm of the vector of \textit{all} $(n-2)\times (n-2)$ minors of $A$.
Hence
\begin{equation}
    \label{eq:particularcase}
|\Div(\textup{cof}(A))|\leq C(n) |\Curl A||\nabla^2 \det A|
\end{equation}
and since $\nabla \det = \textup{cof}$, the proof of this case is complete.

We now prove the general case. Without loss of generality we assume that $M$ is the $k\times k$ minor which depends on the first $k$ rows and columns. The main observation is that
\begin{equation}
    \label{eq:dividentity}
\Div_{x_1,\dots,x_n} \nabla M(A)=\Div_{x_1,\dots, x_k} \nabla^{(k)} M(A),
\end{equation}
where the subscripts denote the variables with respect to which one takes the divergence and $\nabla^{(k)}$ is the gradient operator with respect only to the matrix variables $(a_{ij})_{i,j=1}^k$, thus $\nabla^{(k)}M \colon \R^{n\times n}\to \R^{k\times k}$. To prove \eqref{eq:dividentity}, we write  $\pi_k\colon \R^{n\times n}\to \R^{k\times k}$ for the linear operator such that $\pi_k(A)_{ij}=a_{ij}$ if $1\leq i,j\leq k$, and so
$$M(A)=\textup{det}_k(\pi_k(A)), \qquad \text{where } \textup{det}_k\colon \R^{k\times k}\to \R \text{ is  the determinant.}$$
We see that $\nabla M(A)\in \R^{n\times n}$ is such that $(\nabla M(A))_{ij}=0$ unless $1\leq i,j\leq k$, and thus \eqref{eq:dividentity} is proved.

It follows from \eqref{eq:dividentity} that, to estimate $\Div\nabla M(A)$ at a point $\bar x=(\bar x_1,\ldots,\bar x_n)$, we may consider the restriction of $A$ to the copy of $\R^k$ given by $\{(x_{k+1},\ldots,x_{n})=(\bar x_{k+1},\ldots,\bar x_n)\}$, and apply \eqref{eq:dividentity} and \eqref{eq:particularcase} to obtain
\begin{align*}
    |\Div_{x_1,\dots,x_n} \nabla M(A)|& =|\Div_{x_1,\dots, x_k} \nabla^{(k)} M(A)|
    \\ &= |\Div_{x_1,\dots, x_k} \nabla^{(k)} \textup{det}_k(\pi_k(A))| 
    \leq C |\Curl_{x_1,\ldots,x_k}\pi_k(A)||(\nabla^{(k)})^2 M(A)|,
\end{align*}
Observing that
$$|\Curl_{x_1,\ldots,x_k}\pi_k(A)|\leq |\Curl_{x_1,\dots,x_n} A|$$
and also, as in the proof of \eqref{eq:dividentity} above, 
$$|(\nabla^{(k)})^2M(A)|=|\nabla^2M(A)|,$$
we conclude the proof of the proposition.
\end{proof}

\begin{corollary}\label{cor:divrhok*}
Let $2\leq k \leq n$. For all $A\in C^1(\R^n,\R^{n\times n})$, 
$$\left|\Div \left(\rho_k^{k-1}(A)\nabla \rho_k(A) \right)\right|\leq C(n,k)|\Curl A| |A|^{k-2}.$$
\end{corollary}

\begin{proof}
Up to a multiplicative constant, $\rho_k^k(A)$ is the sum of the $k\times k$ principal minors of $A$. Since $\rho_k^{k-1} \nabla \rho_k = \frac 1 k \nabla \rho_k^k$, the conclusion follows from  Proposition \ref{prop:divrhok*} by noting that, for a $k\times k$ minor $M$, we have $|\nabla^2 M(A)|\leq C |A|^{k-2}$ by Hadamard's inequality.
\end{proof}

To end this subsection, we note that Proposition \ref{prop:divrhok*} yields a quantitative version of the well-known fact (see e.g.\ \cite{Ball1981}) that a $C^1$ integrand $F\colon \R^{n\times n}\to \R$ is a \textit{null Lagrangian}, i.e.\ it satisfies 
$$\Div \nabla F(\D u)=0 \quad \text{for } u\in C^2(\R^n,\R^n),$$ if and only if $F$ can be written as a linear combination of minors. Indeed, we have:

\begin{corollary}[Characterization of homogeneous null Lagrangians]\label{cor:null_lag}
Let $F\in C^1(\R^{n\times n})$ be an integrand and $k\in [2,\infty)$. The following, are equivalent:
\begin{enumerate}
    \item $k\leq n$ is an integer and $F$ is a linear combination of $k\times k$ minors.
    \item\label{it:estimate} The estimate $|\Div \nabla F(A) |\leq C|\Curl A| |A|^{k-2}$ holds for all $A\in C^1(\R^n,\R^{n\times n})$.
\end{enumerate}
\end{corollary}
In particular, the estimate above can be viewed as a perturbed Euler--Lagrange equation for homogeneous null Lagrangians.
\begin{proof}
If we assume that $F$ is a homogeneous null Lagrangian, the estimate follows at once from Corollary \ref{cor:divrhok*}. Conversely, if the estimate holds, we can first plug in $A=\D u$ for $u\in C^2(\R^n,\R^n)$ to see that $F$ is a null Lagrangian. By the fact mentioned above, $F$ is therefore an affine combination of minors, of possibly different orders. We then replace $A$ with $tA$ for $t>0$ in the inequality of \ref{it:estimate} and notice that the right hand side is $(k-1)$-homogeneous, so the left hand side must be too, as it is a sum of homogeneous terms. It follows that $F$ must be $k$-homogeneous, which concludes the proof.
\end{proof}
\subsection{$k$-Hessian equations}
\label{sec:kHessian}
We will frequently make use of the existence and regularity theory for the so-called $k$-Hessian equation, and so we collect here a few results in this direction.

For $k=1,\dots, n$, the $k$-Hessian equation is the equation
\begin{equation}
    \label{eq:khessianprelims}
\rho_k(\D^2u)=f \quad \text{in } B,
\end{equation}
where $B\subset \R^n$ is a ball and $u$ is $k$-admissible. Note that this is a slight variant of the usual definition in the literature, where typically one uses $F_k$ instead of $\rho_k$. The equations are, of course, equivalent up to taking $k$-th powers and adjusting the integrability assumptions on $f$ accordingly. We say that a function $u\in C^2(B)\cap C^0(\overline B)$ is \textit{$k$-admissible} if
$$\lambda(\D^2 u)\in \Gamma_k.$$
When restricted to $k$-admissible solutions, equation \eqref{eq:khessianprelims}  is a fully nonlinear degenerate elliptic equation. 
Note that the $1$-Hessian equation is just the Poisson equation, and a function is $1$-admissible if and only if it is subharmonic. At the other endpoint, the $n$-Hessian equation is the Monge--Ampère equation and $n$-admissible functions are convex. 

In order to consider weak solutions of \eqref{eq:khessianprelims}, it is important to extend the notion of $k$-admissibility to solutions which are not necessarily $C^2$. For our purposes it will be enough to consider \textit{continuous} weak solutions.

\begin{definition}
A function $u\in C^0(B)$ is \textit{$k$-admissible} if, for any matrix $A\in \Gamma_k^*$,
$$\int_B u(x) \langle A, \D^2 \varphi(x) \rangle \dif x \geq 0, 
\qquad  \text{for any } 0\leq \varphi\in C^\infty_0(B).$$
\end{definition}

The definition of $k$-admissible functions coincides with the notion of the inequality $\rho_k(\D^2 u)\geq 0$ holding in the viscosity sense, see for instance \cite{Trudinger1999} or \cite[Definition 8.1]{Wang2009} for further details. 

It is natural to consider the Dirichlet problem for \eqref{eq:khessianprelims}. The classical existence theory, originally due to Caffarelli--Nirenberg--Spruck \cite{Caffarelli1985},
see also \cite{Ivochkina2005},
yields in particular the following result:
\begin{theorem}\label{thm:kHessExistence}
Let $f\in C^{1,1}(\overline{B})$ be non-negative. Then there exists a unique $k$-admissible solution $u\in C^{1,1}(\overline{B})$ to the problem
$$\begin{cases}
\rho_k(\D^2 u)=f & \text{ in }B,\\
u=0 & \text{ on }\d B.
\end{cases}
$$
\end{theorem}

We also refer the reader to Theorem 3.4 in \cite{Wang2009} and the ensuing remarks.
The next result can be found in \cite[Lemma 2.1, Theorem 2.2]{Trudinger1997a}: 

\begin{lemma}\label{lemma:Lpest}
Take $p\geq k$ and $1\leq q\leq \infty$. Let $f\in L^p(B)$ be such that $f\geq 0$. For  a $k$-admissible solution $u\in C^{1,1}(B)\cap C^0(\overline B)$ of the Dirichlet problem $$
\begin{cases} \rho_k(\D^2 u)=f& \text{ in } B,\\
u=0& \text{ on }\d B, \end{cases}$$
we have the estimate
$\|u\|_{L^q(B)}\leq C\|f\|_{L^p(B)}$
whenever
\begin{enumerate}
    \item $q=p^{**}$ if $1<p<\frac{n}{2}$,
    \item any $q<\infty$ if $p=\frac{n}{2}$,
    \item $q=\infty$ if $p>\frac{n}{2}$.
\end{enumerate}
\end{lemma}

We will also use the following integral gradient estimate due to Trudinger--Wang \cite{Trudinger1999}, see also Corollary 7.1 and the remarks following Definition 8.1 in \cite{Wang2009}:
\begin{proposition}\label{prop:WS}
    Suppose $u\in C^{1,1}(B)\cap C^0(\overline B)$ is $k$-admissible and satisfies $u\leq 0$ in $B$. Then, for all $q\in[1,\frac{nk}{n-k})$ and all $\Omega\Subset B$, there exists $C>0$ such that
    $$\|\D u\|_{L^q(\Omega)}\leq C\int_{B}|u|.$$
\end{proposition}

Besides Theorem \ref{thm:kHessExistence}, we require the following existence result for $k$-Hessian equations on the torus:

	\begin{theorem}
	\label{thm:khessian}
	Let $0<f\in\hold^\infty(\mathbb T^n)$ satisfy $\rho_k(S)=\int_{\mathbb T^n} f \dif x$ for some $S\in \Gamma_k$. Then
	\begin{align}
	\label{eq:khessianper}
	    \rho_k(\D^2(\tfrac12 x^\top Sx+\varphi))=f
	\end{align}
	has a unique periodic solution $\varphi\in\hold^\infty(\mathbb T^n)$ such that $x\mapsto \frac 1 2 x^\top S x  + \varphi(x)$ is $k$-admissible and $\int_{\mathbb T^n} \varphi \dif x =0$.
	\end{theorem}

The case $k=n$  was proved by Li \cite{Li1990} and we note that the case $k=1$ is elementary, as then \eqref{eq:khessianper} is just the Poisson equation. To the best of our knowledge,  Theorem \ref{thm:khessian} has not been proved for general $k$ in the literature, although it follows from well-known arguments and results; the interested reader can find a proof in Appendix \ref{appendix}.

\section{Elliptic estimates in two dimensions}\label{sec:GN}

In this section, we prove our first elliptic estimates for the divergence and the curl operators under pointwise constraints. Here we will discuss only the case where $n=2$ and $p=1$: thus, the estimates we prove can be thought of as a version of the Gagliardo--Nirenberg inequality for certain non-elliptic operators. We also note that the methods we employ here are different from and more elementary than those used in later sections.

In Section \ref{sec:GNdiv2d} we will prove the main result of this section, Theorem \ref{thm:slicing}, which in particular contains Theorem \ref{thm:det2dintro}. In Section \ref{sec:GNcurl2d} we deduce from Theorem \ref{thm:slicing} a Gagliardo--Nirenberg inequality for the curl and show an ad-hoc Ornstein non-inequality demonstrating the non-equivalence of these estimates. We also discuss whether a general version of Ornstein's $L^1$ non-inequality \cite{Ornstein1962} may hold under pointwise constraints.  Section \ref{sec:optimality2d} collects examples which show that Theorem \ref{thm:slicing} is essentially optimal. In Section \ref{sec:density} we discuss in some detail the possibility of extending our results to non-convex cones.

\subsection{Elliptic estimates for the divergence}\label{sec:GNdiv2d}

Given a matrix $A\in \R^{n\times n}$, recall that we denote by $A_i$ its rows and $a_{ij}$ its entries. The following is the main result of this section, from which  Theorem \ref{thm:det2dintro} follows immediately:

\begin{theorem}
\label{thm:slicing}
Consider the cones $\mathcal K_1^+, \dots, \mathcal K_4^+\subset \R^{2\times 2}$ defined by
$$\{a_{11},a_{22}\geq 0\},\quad
\{a_{12}\geq 0 \geq a_{21}\},\quad
\{a_{11},a_{22}\leq 0\},\quad
\{a_{21}\geq 0 \geq a_{12}\},$$
respectively.
Likewise, we define the cones $\mathcal K_1^-, \dots, \mathcal K_4^-$ by
$$\{a_{11}\geq 0 \geq a_{22}\}, 
\quad \{a_{12}, a_{21}\leq 0\},
\quad \{a_{12}, a_{21}\geq 0\},
\quad \{a_{22}\geq 0 \geq a_{11}\},$$
respectively. For any $i=1, \dots, 4$, if $A\in C^\infty_c(\R^2,\mathcal K_i^\pm)$ then
\begin{equation}
    \label{eq:detinequality2d}
    \int_{\R^2} \pm \det A \dif x \leq 
\|\di A_1\|_{L^1}\|\di A_2\|_{L^1}.
\end{equation}
\end{theorem}

In fact, it will be clear from the proof that there are
two continuous one-parameter families of admissible cones for which estimate \eqref{eq:detinequality2d} holds. The specific examples $\mathcal K_i^\pm$ stated above are the ones having the clearest characterisation.

By restricting the cones in Theorem \ref{thm:slicing}, we readily obtain a Gagliardo--Nirenberg-type inequality for the divergence:

\begin{corollary}
\label{cor:GN2ddiv}
Let $\mathcal K_i^\pm\subset \R^{2\times 2}$ be as in Theorem \ref{thm:slicing} and consider a cone $\mathcal K\subset Q_2^\pm(K)\cap\mathcal K_i^\pm$ for some $K\geq 1$ and $i=1,\dots, 4.$ There is a constant $C=C(K)>0$ such that, for any $A\in C^\infty_c(\R^2, \mathcal K)$,
$$\|A\|_{L^2(\R^2)}\leq C \|\Div A\|_{L^1(\R^2)}.$$
\end{corollary}

Notice that, in the above results, we restrict ourselves to smooth matrix fields. We now remark on the possibility of extending these inequalities to larger classes of rough fields, a point which we will also discuss in further detail in Section \ref{sec:density} below.

\begin{remark}[Extension to rough fields]\label{rmk:roughfields}
Let $\mathcal K_i^\pm$ be the cones in Theorem \ref{thm:slicing}. Since these cones are convex, a standard mollification argument shows that estimate \eqref{eq:detinequality2d} holds for fields $A\in L^2(\R^2, \mathcal K_i^\pm)$ such that $\Div A \in \mathscr{M}(\R^2)$. 

In order to extend estimate \eqref{eq:detinequality2d} to fields which are only in $L^1(\R^2)$, we need to   constrain the target further. For instance,  if $A\in L^1(\R^2, \textup{Sym}^+_2)$ vanishes outside a compact set,
and if moreover $\Div A \in \mathscr M(\R^2)$, then estimate \eqref{eq:detinequality2d} still holds. Indeed, using the convexity of $\textup{Sym}^+_2$, a simple mollification shows that there is a sequence $A_\e\in C^\infty_c(\R^2,\textup{Sym}^+_2)$ such that $\Div A_\e\to \Div A$ in $\mathscr M(\R^2)$ in the {strict topology} and $A_\e\to A$ in $L^1(\R^2)$. In particular, $A_\e\to A$ a.e.\ in $\R^2$. Since $\textup{Sym}^+_2\subset\{\det \geq 0\}$ we may apply Fatou's lemma to deduce that
\begin{align*}
\int_{\R^2} \det A \dif x & \leq \liminf_{\e\to 0}\int_{\R^2} \det A_\e \dif x \\
& \leq 
\liminf_{\e\to 0}  \|\di (A_\e)_1\|_{L^1}\|\di (A_\e)_2\|_{L^1}
=  \|\di A_1\|_{L^1}\|\di A_2\|_{L^1}.
\end{align*}
More generally, the same conclusion holds if we replace $\textup{Sym}^+_2$ with any convex subcone of some $\mathcal K_i^\pm$ on which the determinant has a fixed sign.

Similarly, Corollary \ref{cor:GN2ddiv} also extends to compactly supported matrix fields which a priori are only in $L^1(\R^2,\mathcal K)$, as long as $\mathcal K\subset Q_2^\pm(K)\cap \mathcal K_i^\pm$ is assumed to be convex.
\end{remark}

\begin{remark}[Boundary conditions]\label{rmk:bdryconds}
The estimate in Corollary \eqref{cor:GN2ddiv} is equivalent to the more general estimate
$$\|A\|_{L^2(\Omega)}\leq C\left( \|\Div A\|_{\mathscr M(\Omega)}+\|A\nu\|_{\mathscr M(\p \Omega)}\right) \quad \text{for } A\in C^\infty(\overline \Omega,\mathcal K)$$
whenever $\Omega\subset \R^2$ is a bounded Lipschitz domain with unit normal $\nu$. Clearly this last estimate recovers the one in Corollary \ref{cor:GN2ddiv}, but it can also be deduced from it by a standard argument.
\end{remark}

We are now ready to proceed with the proof of Theorem \ref{thm:slicing}.

\begin{proof}[Proof of Theorem \ref{thm:slicing}]
We begin by observing that it suffices to establish \eqref{eq:detinequality2d} when $A$ takes values in $\mathcal K_1^+$. Indeed, this implies immediately an analogous inequality if $A$ takes values in $\mathcal K_1^-$, noting that $A\in \mathcal K_1^+$ if and only if $\hat A\in 
\mathcal K_1^-$, where the rows of $\hat A$ are defined as
$\hat A_1=A_1$ and $\hat A_2=-A_2.$ Thus it is enough to show that the inequality for fields in $\mathcal K_1^\pm$ implies the same inequality for fields in $\mathcal K_i^\pm$, $i=2,3,4$.

Given $A\in C^\infty_c(\R^2,\R^{2\times 2})$, consider a new matrix field $\tilde A(x)\equiv A(R_\theta x) R_\theta$, where $R_\theta$ is a rotation matrix by angle $\theta$. We have $\tilde A_j(x)=R_{-\theta} A_j(R_\theta x)$ and so
$$\Div \tilde A(x)=\Div A(R_\theta x), \qquad \det \tilde A(x)=\det A(R_\theta x).$$
In particular, we see that both sides of the inequality \eqref{eq:detinequality2d} are kept unchanged under the transformation $A\mapsto \tilde A$. Noting that
$$ \mathcal K_1^\pm =R_{\pi/2} \mathcal K_2^\pm= R_{\pi} \mathcal K_3^\pm = R_{3\pi/2}\mathcal K_4^\pm,$$
the conclusion follows.

Thus our task is to show \eqref{eq:detinequality2d} for $\mathcal K_1^+$.
We are going to work in conformal/anti-conformal coordinates, that is, we pick the following orthogonal basis of $\R^{2\times 2}$:
$$E_1\equiv \begin{pmatrix}
1 & 0\\
0 & 1
\end{pmatrix},\quad
E_2\equiv \begin{pmatrix}
0 & 1\\
1 & 0
\end{pmatrix},\quad
E_3\equiv \begin{pmatrix}
1 & 0\\
0 & -1
\end{pmatrix},\quad 
E_4\equiv\begin{pmatrix} 0 & 1 \\
-1 & 0\end{pmatrix}.$$
Note that $E_1,E_4$ span the space  of conformal matrices, while $E_2, E_3$ span the space of anti-conformal matrices. We will write
$$A=\sum_{j=1}^4 b_j(x_1,x_2)E_j$$
and we record the formulae
$$ \det A=b_1^2-b_2^2-b_3^2+b_4^2, \qquad
\Div A=
\begin{pmatrix}
\di A_1\\ \di A_2
\end{pmatrix}
=
\begin{pmatrix}
\partial_{x_1}(b_1+b_3)+\partial_{x_2}(b_2+b_4)\\
\partial_{x_1}(b_2-b_4)+\partial_{x_2}(b_1-b_3)
\end{pmatrix}.$$
Using the fundamental theorem of calculus, we get
\begin{equation*}
\begin{aligned}
    (b_1+b_3)(x_1,x_2)=&\,\int_{-\infty}^{x_1}\p_{x_1}(b_1+b_3)(y_1,x_2)\dif y_1\\
    =&\,\int_{-\infty}^{x_1}\Big(\di A_1(y_1,x_2)-\p_{x_2}(b_2+b_4)(y_1,x_2)\Big)\dif y_1\\
    \leq&\,\int_{-\infty}^{\infty}|\di A_1|(y_1,x_2)\dif y_1-\int_{-\infty}^{x_1}\p_{x_2}(b_2+b_4)(y_1,x_2)\dif y_1,
    \end{aligned}
\end{equation*}
and also
\begin{equation*}
\begin{aligned}
    (b_1-b_3)(x_1,x_2)=&\,\int_{-\infty}^{x_2}\p_{x_2}(b_1-b_3)(x_1,y_2)\dif y_2\\
    =&\,\int_{-\infty}^{x_1}\Big(\di A_2(x_1,y_2)-\p_{x_1}(b_2-b_4)(x_1,y_2)\Big)\dif y_2\\
    \leq&\,\int_{-\infty}^{\infty}|\di A_2|(x_1,y_2)\,\dif y_2-\int_{\-\infty}^{x_2}\p_{x_1}(b_2-b_4)(x_1,y_2)\dif y_2.
    \end{aligned}
\end{equation*}
Note that the assumption $A\in\mathcal K_1^+$ ensures that $b_1\geq |b_3|$; hence we can take the product of the last two inequalities and integrate over $\R^2$ to obtain
\begin{align*}
    \int_{\R^2}\big(b_1^2-b_3^2)\dif x\leq&\,\int_{\R^2}\Big(\int_{-\infty}^{\infty}|\di A_1|(y_1,x_2)\dif y_1-\int_{-\infty}^{x_1}\p_{x_2}(b_2+b_4)(y_1,x_2)\dif y_1\Big)\\
    &\quad\times \Big(\int_{-\infty}^{\infty}|\di A_2|(x_1,y_2)\dif y_2-\int_{-\infty}^{x_2}\p_{x_1}(b_2-b_4)(x_1,y_2)\dif y_2\Big)\dif x_1\dif x_2\\
    =&\, \mathrm{I}-\mathrm{II}-\mathrm{III}+\mathrm{IV},
\end{align*}
where we defined
\begin{align*}
\mathrm{I}\equiv &\,\int_{\R^2}\Big(\int_{-\infty}^{\infty}|\di A_1|(y_1,x_2)\dif y_1\Big)\Big(\int_{-\infty}^{\infty}|\di A_2|(x_1,y_2)\dif y_2\Big)\dif x_1\dif x_2,\\
\mathrm{II}\equiv    &\int_{\R^2}\Big(\int_{-`\infty}^{\infty}|\di A_1|(y_1,x_2)\dif y_1\Big)\Big(\int_{-\infty}^{x_2}\p_{x_1}(b_2-b_4)(x_1,y_2)\dif y_2\Big)\dif x_1\dif x_2,\\
\mathrm{III}\equiv    &\int_{\R^2}\Big(\int_{-\infty}^{\infty}|\di A_2|(x_1,y_2)\dif y_2\Big)\Big(\int_{-\infty}^{x_1}\p_{x_2}(b_2+b_4)(y_1,x_2)\dif y_1\Big)\dif x_1\dif x_2,\\
\mathrm{IV}\equiv &\int_{\R^2}\Big(\int_{-\infty}^{x_1}\p_{x_2}(b_2+b_4)(y_1,x_2)\dif y_1\Big)\Big(\int_{-\infty}^{x_2}\p_{x_1}(b_2-b_4)(x_1,y_2)\dif y_2\Big)\dif x_1\dif x_2.
\end{align*}    
First, we observe that 
\begin{equation*}
\begin{aligned}
    \mathrm{II}=&\,\int_{\R}\Big(\int_{-\infty}^{\infty}|\di A_1|(y_1,x_2)\dif y_1\Big)\Big(\int_{-\infty}^{x_2}\int_{\R}\p_{x_1}(b_2-b_4)(x_1,y_2)\dif x_1\dif y_2\Big)\dif x_2=0
\end{aligned}
\end{equation*}
as $b_2-b_4$ is smooth and compactly supported. 
A similar argument shows that $\mathrm{III}=0.$
Moreover, we check for $\mathrm{I}$ that
\begin{equation*}
    \begin{aligned}
    \mathrm{I}=&\,\int_{-\infty}^{\infty}\int_{-\infty}^\infty\Big(\int_{-\infty}^\infty |\di A_1|(y_1,x_2)\dif y_1\Big)\Big(\int_{-\infty}^\infty|\di A_2|(x_1,y_2)\dif y_2\Big)\dif x_1\dif x_2\\
    =&\,\|\di A_1\|_{L^1(\R^2)}\|\di A_2\|_{L^1(\R^2)}.
    \end{aligned}
\end{equation*}
Finally, as $b_2$ and $b_4$ have compact support, we integrate by parts twice in $\mathrm{IV}$ to obtain
\begin{equation*}
    \begin{aligned}
    \mathrm{IV}=&\,\int_{\R^2}(b_2-b_4)(b_2+b_4)(x_1,x_2)\dif x_1\dif x_2.
    \end{aligned}
\end{equation*}
Collecting everything, we find the estimate
\begin{equation*}
    \begin{aligned}
    \int_{\R^2}(b_1^2-b_3^2)(x_1,x_2)\dif x_1\dif x_2\leq&\,\|\di A_1\|_{L^1(\R^2)}\|\di A_2\|_{L^1(\R^2)}+\int_{\R^2}\big(b_2^2-b_4^2\big)(x_1,x_2)\dif x_1\dif x_2,
    \end{aligned}
\end{equation*}
which we may rearrange to get 
$$\int_{\R^2} \det A \dif x\\
\leq \|\di A_1\|_{L^1(\R^2)}\|\di A_2\|_{L^1(\R^2)}$$
as wished.
\end{proof}

We conclude this subsection with a version of Theorem \ref{thm:slicing} for $p>1$:

\begin{proposition}\label{prop:strongdiv}
Fix $K\geq 1$ and let $p\in (1,2)$ be such that $p^*< \frac{2K}{K-1}$. There is a constant such that, for all compactly supported fields $A\in L^1_{\locc}(\C,Q_2^+(K))$ with $\Div A \in L^p(\C)$,
\begin{equation*}
    \label{eq:strongdivtext}
\|A\|_{L^{p^*}(\C)} \leq C\|\Div A\|_{L^p(\C)}.
\end{equation*}
This estimate fails for $p=1$ (even if $K=1$) or $p^*=\frac{2K}{K-1}$. When $p=1$ we instead have
$$\|A\|_{\textup{weak-}L^{2}(\C)} \leq C \|\Div A\|_{L^1(\C)}.$$
\end{proposition}

\begin{proof}
Let us identify $A\in \R^{2\times 2}$ with a pair $(a_+,a_-)\in \C^2$ according to the usual rule
$A z = a_+ z + a_-\bar z.$ The condition $A\in Q_2(K)$ is equivalent to the equation 
\begin{equation}
    \label{eq:beltrami}
a_-=-\mu a_+, \qquad |\mu|\leq k1_{B_R(0)}=\frac{K-1}{K+1}1_{B_R(0)},
\end{equation}
where $R$ is large enough so that $\supp A \subset B_R(0).$

Let us denote by $\p,\bar \p$ the usual Wirtinger derivatives acting on complex-valued functions and by $\mathcal S$ the Beurling--Ahlfors transform, which is a singular integral operator with the characteristic property
$$\mathcal S \circ \bar \p = \bar \p \circ \mathcal S = \p\quad \text{in } \mathscr D'(\C),$$
see e.g.\ \cite[§4]{Astala2009} for further information on this operator. We also note that $\mathcal S^{-1} = \overline{\mathcal S \,\bar{\cdot}}$ and that both $\mathcal S$ and $\mathcal S^{-1}$ are bounded in $L^q(\C)$ for $1<q<\infty$.

Under the obvious identifications, we compute
$$2\Div A = \bar \p a_+ +\p a_-= \bar \p (a_+ + \mathcal S a_-).$$
Thus, as $\bar \p$ is elliptic (but not canceling), if $p\in (1,2)$ we have
$$\|a_+ + \mathcal S a_-\|_{L^{p^*}(\C)}\leq C \|\Div A\|_{L^p(\R^2)}$$
while for $p=1$ we only get 
$$\|a_++\mathcal S a_-\|_{\text{weak-}L^{2}(\C)}\leq C \|\Div A\|_{L^1(\R^2)}.$$

Since both $a_+$ and $a_-$ are compactly supported, from \eqref{eq:beltrami} we easily deduce the representation formulas
\begin{align*}
a_+& = \mathcal S (I-\mu \mathcal S)^{-1} \mathcal S^{-1} (a_++\mathcal S a_-),\\
a_-& =(I-\mu \mathcal S)^{-1} [\mu (a_++\mathcal S a_-)].
\end{align*}
Here we note that the operator $I-\mu \mathcal S$ is invertible on $L^q(\C)$ if and only if $\frac{2K}{K+1}<q<\frac{2K}{K-1}$ \cite[§14]{Astala2009}. Thus, by our assumption, Since $2= 1^*<p^*<\frac{2K}{K-1}$, the conclusion follows.

The failure of the estimate at the endpoint $p=1$ will be proven in Lemma \ref{lemma:doublecones} below. For the failure at the endpoint $p^*=\frac{2K}{K-1}$, it suffices to consider the spherically symmetric map
\begin{equation*}
    u(x)=\rho(r) \frac{x}{r}, \quad \text{ where } \rho(r)=r^{1-2/p}.
\end{equation*}
Writing $\nu=\frac{x}{r}$, we have
$$\D u=\frac{\rho(r)}{r} \textup{Id} + \Big( \dot \rho(r)-\frac{\rho(r)}{r}\Big) \nu \otimes \nu,$$
and one easily deduces that
$$\D u\in Q_2^-\Big(\frac{p}{p-2}\Big), \qquad u \in W^{1,q}\backslash W^{1,p}(\B^2) \text{ for any }q<p.$$
Note that $\D u\cdot \nu^\bot = \nu^\bot$ on $\mathbb S^1$. Let $p=\frac{2K}{K-1}$, so that $\frac{p}{p-2}=K$. It suffices to take a cutoff $\chi\in C^\infty_c(\mathbb B^2)$ with $\chi=1$ near 0 and $A(x_1,x_2)=\chi(x_1,x_2) \D u(x_2,-x_1)$ to conclude.
\end{proof}

\subsection{Elliptic estimates for the curl}\label{sec:GNcurl2d}

We begin this subsection by rewriting Theorem \ref{thm:slicing} using the curl operator instead of the divergence.

\begin{corollary}
\label{cor:curlinequality2d}
Fix $i\in \{1,\dots, 4\}$ and let $\mathcal K_i^\pm$ be as in Corollary \ref{thm:slicing}. If $A\in C^\infty_c(\R^2,\mathcal K_i^\pm),$
$$\int_{\R^2} \pm \det A \dif x \leq \|\curl A_1\|_{L^1(\R^2)}\|\curl A_2\|_{L^1(\R^2)}$$
\end{corollary}

\begin{proof}
The corollary follows from the identity $\curl A_i = \di A_i^\bot$ and the fact that the family of cones $\{\mathcal K_i^\pm\}_{i=1}^4$ is closed under the map $A=(A_1,A_2)\mapsto (A_1^\bot, A_2^\bot)$.
\end{proof}

Although the previous estimate follows at once from Theorem \ref{thm:slicing}, it is nonetheless a genuinely different inequality. Indeed, we have the following Ornstein non-inequality (note that $\textup{Sym}^+_2\subset \mathcal K_1^+$):

\begin{lemma}\label{lemma:babyOrn}
There is no constant $C$ such that, for all $A\in C^\infty_c(\R^2,\textup{Sym}^+_2)$,
$$\|\Div A\|_{L^1(\R^2)} \leq C \|\Curl A\|_{L^1(\R^2)}.$$ Likewise, there is no constant $C$ such that, for all $A\in C^\infty_c(\R^2,\textup{Sym}^+_2)$,
$$\|\Curl A\|_{L^1(\R^2)} \leq C \|\Div A\|_{L^1(\R^2)}.$$
\end{lemma}

\begin{proof}
Consider the matrix fields
\begin{align*}
    &A_\eps\equiv \begin{pmatrix}
1+\frac{1}{4}\sin(\frac{x_1}{\eps})\sin(x_2) & 0\\
0 & 1+\frac{1}{4}\sin(\frac{x_2}{\eps})\sin(x_1)
\end{pmatrix},\\
&B_\eps\equiv \begin{pmatrix}
1+\frac{1}{4}\sin(\frac{x_2}{\eps})\sin(x_1) & 0\\
0 & 1+\frac{1}{4}\sin(\frac{x_1}{\eps})\sin(x_2)
\end{pmatrix}.
\end{align*}
One easily checks that
\begin{align*}
    \Div A_\eps=\begin{pmatrix}
    \frac{1}{4\eps}\cos(\frac{x_1}{\eps})\sin(x_2)\\
    \frac{1}{4\eps}\cos(\frac{x_2}{\eps})\sin(x_1)
    \end{pmatrix}, \qquad 
    \Curl A_\eps=\begin{pmatrix}
    \frac{1}{4}\sin(\frac{x_1}{\eps})\cos(x_2)\\
    -\frac{1}{4}\sin(\frac{x_2}{\eps})\cos(x_1)
    \end{pmatrix}.
\end{align*}
Thus, on the box $[0,2\pi]^2$, 
$$\|\Div A_{\e}\|_{L^1}=\frac{4}{\e},\quad \|\Curl A_{\e}\|_{L^1}=4.$$
Similarly, we obtain 
$$\|\Div B_{\e}\|_{L^1}=4,\quad \|\Curl B_{\e}\|_{L^1}=\frac 4\e.$$
Note that $A=B=\textup{Id}$ on $\p ([0,2\pi]^2)$. In order to get a compactly supported example, it is enough to consider a smooth cutoff $\eta$ with support in $[0,2\pi]^2$  which is equal to 1 in $[\pi/2,3\pi/2]^2$. We may replace the fields $A_\e, B_\e$ by $\eta A_\e, \eta B_\e$: they still take values in $\textup{Sym}^+_2$ and their divergence and curl is affected only by a constant which is independent of $\e$.
\end{proof}

It would be interesting to establish an analogue of Ornstein's result \cite{Ornstein1962} for constrained fields. We propose a formulation of such a result in the next remark.

\begin{remark}[Ornstein's non-inequality with constraints]\label{question:Ornstein}
Given a finite-dimensional vector space $\mathbb V$, let  $\mathcal K\subset \mathbb V$ be a closed, convex cone which spans $\mathbb V$. Let $\mathcal P_1$, $\mathcal P_2$ be two $\ell$-th order, constant coefficient partial differential operators on $\R^n$ from $\mathbb V$ to $\mathbb W_1, \mathbb W_2$, respectively. 
If there is a constant $C$ such that
$$\| \mathcal P_1 \varphi \|_{L^1(\R^n)} \leq C\|\mathcal P_2 \varphi \|_{L^1(\R^n)} \quad \text{ for all } \varphi \in C^\infty_c(\R^n, \mathcal K),$$
is it the case that $|\mathcal P_1 
\varphi|\leq C'|\mathcal P_2\varphi|$ pointwise for some constant $C'$ and all $\varphi\in C^\infty_c(\R^n,\mathcal K)$?
An inspection of the standard proofs in the unconstrained case in \cite{Conti2005,Faraco2020,Kirchheim2016,Ornstein1962} shows that they do not extend to this setting in an obvious way.
\end{remark}

\subsection{Optimality}\label{sec:optimality2d}

In this subsection we collect several examples and remarks which show that Theorem \ref{thm:slicing} is essentially optimal.

The first point we address is whether one can hope for a gain of higher integrability if one assumes that the divergence of the matrix field is itself more integrable. This question was posed already by  Serre \cite{Serre2018a} and answered negatively by De Rosa--Tione \cite{DeRosa2019a}:

\begin{proposition}[Necessity of $p=1$]
\label{prop:p=1isnecessary}
Fix $p\in (1,n)$. Let $A\in L^p(\B^n,\textup{Sym}^+_n)$ be  such that $A=0$ near $\mathbb S^{n-1}$ and $\Div A\in L^p(\B^n)$. In general, $(\det A)^{ 1/ n}\not \in L^{p^*}(\B^n)$. In fact, for $1<p<1^*$, there is no $\e>0$ such that we always have $(\det A)^{ 1/ n}\in L^{1^*+\e}(\B^n)$.
\end{proposition}

A different question, which was also briefly addressed by  Serre \cite[\S 4]{Serre2019}, is the extent to which the entry-wise constraints are truly necessary. One could imagine, for instance, that Proposition \ref{prop:p=1isnecessary} may still hold if we replace $\mathcal K_1^+$ with the (non-convex) cone of matrices with non-negative determinant; however, this is not the case:

\begin{lemma}[Necessity of entry-wise constraints]
\label{lemma:doublecones}
There is no constant $C$ such that, for all $A\in C^\infty_c(\R^2,\R^{2\times 2})$ with $\det A\geq 0$,
$$\int_{\R^2} \det A \dif x \leq C \|\Div A\|_{L^1(\R^2)}^2.$$
\end{lemma}

\begin{proof}
Let $G(x)=\frac{1}{2\pi}\log|x|$ be the Green's function for the planar Laplacian. We define, for $|x|<1,$ the field
\begin{equation}
    \label{eq:greenexample}
A(x)\equiv \begin{pmatrix}
G_{x_2} & -G_{x_1}\\
G_{x_1} & G_{x_2}
\end{pmatrix}=\frac{1}{\pi|x|^2}\begin{pmatrix}
x_2 & - x_1 \\ x_1 & x_2
\end{pmatrix}.
\end{equation}
Note that $A$ takes values in the two-dimensional linear space $Q_2^+(1)$ of conformal matrices, cf.\ \eqref{eq:quasiregmatrices}. We have
$A \in L^{2,\infty}\backslash L^2(\B^2)$ and moreover
$$\Div A=(0,\delta_0)\in\mathscr{M}(\B^2),$$
while 
$$0<\det A=\frac{1}{\pi|x|^2}\not\in L^1(\B^2).$$

Take $A_\e\equiv \eta (A*\rho_\e)$, where $\rho_\e$ is the standard mollifier and $\eta\in C^\infty_c(\B^2)$ is a non-negative cutoff with $\eta=1$ near $0$. Clearly $A_\e$ still respects the constraint and
$$\int_{\R^2} \det A_\e \dif x \to +\infty \text{ as }\e\to 0, \qquad \|\Div A_\e\|_{L^1(\R^2)}\leq C$$
for some constant $C$ independent of $\e$, as wished.
\end{proof}

\begin{remark}
An example related to the previous lemma can be obtained by swapping the two rows of $A$ in \eqref{eq:greenexample}: by doing so, we get a matrix field $\tilde A$ which shows that there is also no estimate for fields constrained to take values in $\textup{Sym}_2\cap \{\det \leq 0\}.$ Note that $\mathcal K_2^-\cap \textup{Sym}_2$ is a two-dimensional half-space, while $\tilde A$ takes values in the full linear space containing this half-space. Thus entrywise constraints are in some sense necessary for an estimate. 
\end{remark}

\subsection{Density of smooth functions and separation}\label{sec:density}
The previous two remarks provide examples of matrix fields $A$ which take values in $\textup{Sym}_2\cap \{\det \leq 0\}$ and which falsify the inequalities of this section. Notice also that, in both cases, the matrix field $A$ is \textit{constant} on the boundary of the unit ball. A question which the previous remarks leave open is the following:
\begin{question}
\label{question:symmetric}
Can the inequalities of this section hold if we constrain the fields to take values in the double cone $$\textup{Sym}_2\cap \{\det \geq 0\}= \textup{Sym}_2\cap\left(\mathcal K_1^+ \cup \mathcal K_3^+\right)=\textup{Sym}^+_2\cup \textup{Sym}^-_2?$$ 
\end{question}
In fact, for fields satisfying the exact constraint $\Curl A=0$, we have the following  rigidity result, which is a particular case of a theorem of \v Sver\'ak \cite{Sverak1992}, see also the remarkable generalisations due to Faraco--Sz\'ekelyhidi \cite{Faraco2008} as well as \cite{Sverak1993,Szekelyhidi2005,Szekelyhidi2006} for related results:

\begin{proposition}
\label{prop:FSz}
Let $A\in L^2(\B^2,\textup{Sym}_2)$ be such that $\Curl A=0$ and $A|_{\mathbb S^1}$ is constant. Then
$$\text{either } A\in \textup{Sym}^+_2 \text{ a.e.\ in }\Omega \qquad \text{ or } 
A\in \textup{Sym}^-_2 \text{ a.e.\ in }\Omega.$$
\end{proposition}

As before, one can convert the constraint $\Curl A=0$ into the constraint $\Div \tilde A=0$ through the map
$$A=\begin{pmatrix}
a & b \\ b & c 
\end{pmatrix} 
\mapsto 
\begin{pmatrix}
-c & b \\ b & - a
\end{pmatrix}=\tilde A,$$
which interchanges $\textup{Sym}^+_2$ and $\textup{Sym}^-_2$.

A question related to Question \ref{question:symmetric} is whether it is possible to approximate rough fields (with or without exact differential constraints) by smooth fields that retain the pointwise constraint. As mentioned before, when the target is a convex cone, this is easily achieved by a straightforward mollification argument. However, when convexity fails, the situation is far more complex, and Lemma \ref{lemma:doublecones} shows that in general it is not possible to do so. We now give an example which shows that the same phenomenon happens in the setting of Question \ref{question:symmetric}.

\begin{proposition}\label{prop:nodensity}
Let $\mathcal{K}^0$ be an open, convex cone in $\textup{Sym}^+_2$ and denote by $\mathcal{K}$ the double cone $\overline{\mathcal{K}^0}\cup\big(-\overline{\mathcal{K}^0}\big)$. Then there exists a map $A\in L^\infty([-1,1]^2,\mathcal{K})$ such that:
\begin{enumerate}
    \item  $\Div A\in \mathcal{M}([-1,1]^2,\R^2)$;
    \item there does not exist any sequence of smooth maps $(A_\eps)_{\eps>0}$ with $A_\eps\in C^\infty([-1,1]^2,\mathcal{K})$ such that 
    $$A_\eps\to A\text{ in } L^1,\qquad \|\Div A_\eps\|_{L^1}\to\|\Div A\|_{\mathcal{M}}.$$
In other words, there is no smooth sequence approximating $A$ in the strict topology in $\textup{BV}^{\Div}$ that preserves the pointwise constraint.
\end{enumerate} 
\end{proposition}

\begin{proof}
Let $A_1\in \mathcal{K}^0$, $A_2\in-\mathcal{K}^0$ be such that the line segment $[A_1,A_2]$ does not pass through the origin. Define the function $A$ by
$$A(x_1,x_2)\equiv \begin{cases}
A_1 &\text{if }x_1<0,\\
A_2 &\text{if } x_1>0.
\end{cases}$$
One easily sees that $$\Div A=(A_2-A_1)e_1\mathcal{H}^1\big|_{x_1=0},$$
and hence
$$\|\Div A\|_{\mathcal{M}}=2|(A_2-A_1)e_1|.$$
Moreover, by construction of $A_2$ and $A_1$, 
\begin{equation}
    \label{eq:A1A2}
    |(A_2-A_1)e_1|<|A_2e_1|+|A_1e_1|.
\end{equation}
Suppose for a contradiction that there exists a smooth sequence $A_\eps$ as in the statement of the proposition. 
We fix small auxiliary parameters $\delta_0, \lambda>0$ which will be chosen later, as well as a number $\delta\in (\frac{\delta_0 }{2}, \delta_0)$. As $A_\eps\to A$ in $L^1([-1,1]^2)$, there is $\e_0>0$ such that 
\begin{equation}
    \label{eq:approximation}
\|A_\eps - A\|_{L^1([-1,1]^2)} <\lambda\quad \text{ for all }0<\e<\e_0.
\end{equation}
 Let us fix one such $\e>0$ and let $f_\eps\equiv \tr A_\eps$.  Consider the set $$\{b\in \R: f_\e^{-1}(b) \text{ is a smooth manifold}\}.$$ By Sard's Theorem, this set has full measure in $\R$. By the co-area formula,
\begin{equation}
    \label{eq:coarea}
\int_{[-1,1]^2} |\D f_\e(x)|\dif x = \int_\R \mathcal H^1(f_\e^{-1}(t))\dif t,
\end{equation}
 and so there is $0<b<|A_1|/2$ such that $f^{-1}_\e(b)$ is a smooth manifold and 
 \begin{equation}\label{ineq:levelset}
     \mathcal{H}^1(f_\eps^{-1}(b))\leq \frac{\de}{b},
 \end{equation}
for otherwise the integral in the right-hand side of \eqref{eq:coarea} would be infinite.

Let us write  $R_{\de}\equiv [-\de,\de]\times[-1+\de,1-\de]$ and consider the sets
 \begin{gather*}
E_1\equiv \{v\in \R^2: f_\e^{-1}(b) \text{ intersects } v+\p R_\delta\text{ transversely}\},\\
E_2\equiv \Big\{v\in \R^2: \int_{\p R_\delta} |A_\e-A| \dif \mathcal H^1 <\lambda
\Big\}.
 \end{gather*}
 According to the General Position Lemma the set $E_1$ has full measure \cite{Guillemin2010}, and by \eqref{eq:approximation} and Fubini's theorem the same holds for $E_2$, so we take $v=(v_1,v_2)\in E_1\cap E_2$ with $|v_1|+|v_2|\leq \frac{\delta_0}{4}$; in particular, $v+R_\delta\subset [-1,1]^2$. Strictly speaking, to apply the General Position Lemma, we would need the sets $R_\delta$ to be smooth, while they are just Lipschitz. This technicality can be dealt with by rounding off the corners of $R_\delta$ a little bit and dealing with the new terms generated by the round parts as in \eqref{eq:bdry3} below. We leave the details to the interested reader and instead we pretend that $R_\delta$ is a smooth manifold.
  More importantly for our purposes here, note that by construction of $E_1$ the relatively open sets
  $$\Omega_+\equiv \{x\in v+R_\de:f_\eps(x)>b\},\quad \Omega_-\equiv \{x\in v+R_\de:f_\eps(x)<b\}$$
  are Lipschitz.

 We now estimate the $L^1$ norm of $\Div A_\eps$ as follows. First, we break the region of integration into two pieces: $$\int_{v+R_\de}|\Div A_\eps|\dif x=\int_{\Omega_+}|\Div A_\eps|\dif x+\int_{\Omega_-}|\Div A_\eps|\dif x.$$
 Working on $\Omega_+$, we estimate
 \begin{align*}
    \int_{\Omega_+}|\Div A_\eps|\dif x\geq \Big|\int_{\Omega_+}\Div A_\eps \dif x\Big|=\Big|\int_{\d\Omega_+}A_\eps \nu\dif x\Big|,
 \end{align*}
 where we have used that $\d\Omega_+$ is Lipschitz and $A_\eps$ is smooth to perform the integration by parts; here $\nu$ denotes the unit normal to $\p \Omega_+$.
We divide the boundary,  $\d\Omega_+$, into four portions:
$$\d\Omega_+=\Sigma_1^+\cup\Sigma_2^+\cup\Sigma_3^+\cup\Sigma_4^+,$$
where
\begin{align*}
    \Sigma_1^+\equiv &\{x\in\d\Omega_+:x_1=-\de+v_1\},\\
    \Sigma_2^+\equiv&\{x\in\d\Omega_+:x_1=\de+v_1\},\\
    \Sigma_3^+\equiv&\{x\in\d\Omega_+:x_2=1-\de+v_2\}\cup\{x\in\d\Omega_+\,:\,x_2=\de+v_2\},\\
    \Sigma_4^+\equiv&\{x\in\d\Omega_+:f_\eps(x)=b\}.
\end{align*}
Note that, equivalently, we have 
\begin{align*}
&\Sigma_1^+=\{x\in v+R_\delta: x_1=-\delta+v_1 \text{ and } f_\e(x)>b\},\\ 
&\Sigma_2^+=\{x\in v+R_\delta: x_1=\delta + v_1 \text{ and } f_\e(x)>b\}.
\end{align*}
We now estimate the $\mathcal H^1$-measure of these two regions. Since $b>0$, whenever $f_\e>b$ we have $|A_2-A_\e|\geq \textup{dist}(A_2, \{\tr=0\})= |\tr (A_2)|\geq |A_2|$, that is, we have the inclusion $\{f_\e>b\}\subset \{|A_2-A_\e|\geq |A_2|\}$. Hence
\begin{equation}
    \label{eq:Sigma2size}
\mathcal H^1(\Sigma_2^+)\leq \frac{\lambda}{|A_2|}
\end{equation}
by Chebyshev's inequality and the definition of $E_2$; note that, as $|v_1|\leq\frac14 \delta_0$, we have $v_1+\delta>0$, so that $A|_{x_1=v_1+\delta}=A_2$. Similarly, since $b\leq |A_1|/2,$ we have the inclusion $\{f_\e\leq b\}\subset\{|A_1-A_\e|\geq |A_1|/2\}$ and so
\begin{equation}
    \label{eq:Sigma1size}
    \mathcal H^1(\Sigma_1^+)\geq 2\Big(1-\delta -\frac{\lambda}{|A_1|}\Big).
\end{equation}

We estimate the integral of $A_\e \nu$ on each of the components of $\p \Omega_+$ separately and we use throughout the fact that $v\in E_2$.
Starting with the first term, we have
\begin{equation}
    \label{eq:bdry1}
\Big|\int_{\Sigma_1^+}A_\eps\nu\Big|=\Big|-\int_{\Sigma_1^+}\big(A_1e_1+(A_\eps-A_1)e_1\big)\Big|\geq |A_1e_1|(2-2\de -C\la)-\la
\end{equation}
where we used \eqref{eq:Sigma1size}.
For the second term, by \eqref{eq:Sigma2size},
\begin{equation}
\Big|\int_{\Sigma_2^+}A_\eps\nu\Big|=\Big|\int_{\Sigma_2^+}\big(A_2e_1+(A_\eps-A_2)e_1\big)\Big|\leq C\la+\la.
\end{equation}
The third term is estimated using that $\mathcal{H}_1(\Sigma_3^+)\leq 4\de$ as 
\begin{equation}
\label{eq:bdry3}
\Big|\int_{\Sigma_3^+}A_\eps\nu\Big|=\Big|\int_{\Sigma_3^+}\big(Ae_2+(A_\eps-A)e_2\big)\Big|\leq 4\max\{|A_1|,|A_2|\}\de+\la.
\end{equation}
Finally, on $\Sigma_4^+$, we have that $f_\eps=b$, and hence, as $\mathcal{K}\subset\textup{Sym}^+_2\cup\textup{Sym}^-_2$, we have the simple estimate $|A_\eps|\leq |f_\eps|=b$. So using \eqref{ineq:levelset} we obtain
\begin{equation}
    \label{eq:bdry4}
\Big|\int_{\Sigma_4^+}A_\eps\nu\Big|\leq b\mathcal{H}_1(\Sigma_4^+)\leq\de.
\end{equation} Thus, combining estimates \eqref{eq:bdry1}--\eqref{eq:bdry4}, we obtain
$$\Big|\int_{\d\Omega_+}A_\eps \nu\dif x\Big|\geq \Big|\int_{\Sigma_1^+}A_\eps\nu\Big|-\sum_{k=2}^4\Big|\int_{\Sigma_k^+}A_\eps\nu\Big|\geq |A_1e_1|(2-2\de-C\la)-C(\de+\la).$$
Arguing similarly on $\Omega_-$, we obtain
\begin{align*}
    \int_{v+R_\de}|\Div A_\eps|\dif x&=\int_{\Omega_+}|\Div A_\eps|\dif x+\int_{\Omega_-}|\Div A_\eps|\dif x\\
    &\geq \big(|A_1e_1|+|A_2e_2|\big)(2-2\de-C\la)-C(\de+\la).
\end{align*} 
By \eqref{eq:A1A2}, if we choose $\de$ and $\la$ sufficiently small, we conclude that 
$$\|\Div A_\eps\|_{L^1}\not\to2|(A_2-A_1)e_1|=\|\Div A\|_{\mathcal{M}},$$
a contradiction to the strict convergence.
\end{proof}

\section{New examples of div-quasiconcave integrands}
\label{sec:qc}

The purpose of this section is to prove that appropriate powers of the integrands $\rho^*_k\colon \Gamma_k^*\to[0,+\infty)$, as defined in Section \ref{sec:cones}, are $\Div$-quasiconcave, see Appendix 
\ref{app:LlogL} for the precise definition. When $k=n$ we recover inequality \eqref{eq:serre_qc}, due to Serre \cite{Serre2018a,Serre2019}, and for $k\neq n$ we obtain new examples of \textit{compensated integrability} for divergence-free fields constrained to the cones $\Gamma_k^*$:

\begin{theorem}
\label{thm:quasiconcavity}
Let $A\in L^1(\mathbb T^n,\Gamma_k^*)$ be such that $\Div A=0$. Then
\begin{equation}
    \label{eq:quasiconcavityineq}
\int_{\mathbb T^n}\rho_k^*(A)^{\frac{k}{k-1}} \dif x\leq \rho_k^*\left(\int_{\mathbb T^n} A\dif x\right)^{\frac{k}{k-1}}.
\end{equation}
The inequality also holds for $\Div^2$-free fields.\\ Moreover, the inequality is sharp: in the class of smooth fields $A\in C^\infty(\T^n,\Gamma_k^*)$, equality is attained if and only if $A=\nabla F_k(S+\D^2 \varphi)$ for a periodic $\varphi\in C^\infty(\T^n)$ and $S\in\Gamma_k$ such that $S+\D^2\varphi\in\Gamma_k$.
\end{theorem}
The authors would like to thank D.~Serre for pointing out to us that such matrix fields achieve equality in Theorem \ref{thm:quasiconcavity}.

We recall that $\Div^2$ is an abbreviation for the operator $\di\circ\Div$. Note that, when $k=1$, fields in $\Gamma_1^*$ are scalar multiples of the identity,  so the condition of being $\Div$-free is equivalent to being constant. Hence the inequality is trivial in this case.

\begin{proof}
Since the cones $\Gamma_k^*$ are convex, it is enough to consider the case where $A\in C^\infty(\mathbb T^n,\Gamma_k^*)$ is $\Div$-free as else we may mollify. We write  $f\equiv \rho_k^*(A)^{{k}/{(k-1)}}.$ We may also assume that $A$ is uniformly inside $\Gamma_k^*$, i.e.\ that $f$ is uniformly positive, as otherwise we replace $A$ with $A_\e\equiv A+\e \,\textup{Id}$, apply the inequality to $A_\e$, and then send $\e\to 0$.

Since a $\Div$-free field is also $\Div^2$-free, it suffices to prove the inequality in the latter case.
Take an arbitrary $S\in \Gamma_k$ such that $$\rho_k(S)=\Big(\int_{\mathbb T^n} f \dif x\Big)^{1/ k}.$$ Let $\psi=\frac 1 2 Sx\cdot x+\varphi$ be the solution of the equation $\rho_k(\D^2 \psi)=f^{1 /k}$ on $\mathbb{T}^n$ provided by Theorem \ref{thm:khessian}; in particular, $\int_{\T^n}\varphi=0$. Using Proposition \ref{prop:dualitypair}, we then estimate
\begin{align*}
    \int_{\mathbb T^n} f \dif x =  \int_{\mathbb T^n} \rho_k(\D^2\psi)\rho_k^*(A)\dif x\leq \frac1 n\int_{\mathbb T^n} \langle S+\D^2\varphi,A\rangle \dif x= \frac1 n\langle S, \bar A\rangle,
\end{align*}
where $\bar A\equiv \int_{\mathbb T^n} A \dif x$ and we have used that $A$ is $\Div^2$-free to integrate by parts twice in the $\varphi$ term. Note that, by convexity of $\Gamma_k^*$, we have $\bar A\in \Gamma_k^*$.
We now choose $S$, according to Proposition \ref{prop:dualitypair}, so that $\frac1 n\langle S, \bar A\rangle=\rho_k^*(\bar A)\rho_k(S)$,  hence
\begin{equation*}
      \label{eq:choiceS}
\frac1 n\langle S,\bar A\rangle =  \rho_k(S) \rho_k^*(\bar A)= \left(\int_{\mathbb T^n} f\dif x\right)^{ 1/ k} \rho_k^*(\bar A).
\end{equation*}
It follows that
$$\left(\int_{\mathbb T^n} f \dif x\right)^{{k}/{(k-1)}}\leq \rho_k^*(\bar A).$$
Recalling that $f=\rho_k^*(A)^{{k}/{(k-1)}}$ and the definition of $\bar A$, the inequality follows by taking the $\frac{k}{k-1}$ power.

To show the equality case, we inspect the above proof. In order to have equality in \eqref{eq:quasiconcavityineq}, we must have that 
$$\rho_k(\D^2 \psi) \rho_k^*(A) = \frac 1 n\langle \D^2 \psi, A \rangle \quad \text{for all $x$ in } \T^n.$$
 By the proof of Proposition \ref{prop:simplifiedrhok*}, at any point $x_0\in\T^n$ where this equality is achieved, $\D^2\psi(x_0)$ is the (unique) minimiser of $B\mapsto \frac{1}{n}\langle B,A(x_0)\rangle$ in the set $\{B\in\Gamma_k:\rho_k(B)=\rho_k(\D^2\psi(x_0))\}.$ Moreover, by the same proof, this implies that 
 $$A(x_0)=c\nabla\rho_k(\D^2\psi(x_0)).$$
 Recalling now the definition of $\psi$, we have that 
 \beqas
 \rho_k(\D^2\psi(x_0))=f(x_0)^{1/k}=\rho_k^*(c\nabla\rho_k(\D^2\psi(x_0)))^{{1}/{(k-1)}}=\Big(\frac{c}{n}\Big)^{{1}/{(k-1)}}
 \eeqas
 by Lemma \ref{lemma:attainment}. Solving for $c$, we have obtained
 $$A=n\rho_k(\D^2\psi)^{k-1}\nabla\rho_k(\D^2\psi)=\frac{n}{k}\nabla F_k(S+\D^2\varphi).$$
Moreover, by the proof of the inequality above, any $A$ of this form achieves equality in \eqref{eq:quasiconcavityineq}, as wished.
\end{proof}

\begin{remark}[$k=2$]
When $k=2$ one can prove a stronger version of Theorem \ref{thm:quasiconcavity} using standard arguments, as we now show. We consider the quadratic form 
$$F_2(A)\equiv \frac{1}{\sqrt n} \left(\tr(A)^2-(n-1)|A|^2\right),$$
which extends the representation \eqref{eq:rho2*} of $(\rho_2^*)^2$ to any $A\in \textup{Sym}_n$: in particular, we have $F_2(A)=(\rho_2^*(A))^2$ whenever $A\in \Gamma_2^*.$
It turns out that $F_2$ is $\Div$-quasiconcave in $\textup{Sym}_n$, \textit{regardless of pointwise constraints}. To prove this, we recall a celebrated theorem due to Tartar \cite{Tartar1979} which asserts that a quadratic form is $\Div$-quasiconcave if and only if it is non-positive on $\Lambda_\textup{Div}\equiv\{A\in\textup{Sym}_n:\det A=0\}.$ To check this last condition, let $A\in \Lambda_{\Div}$ have eigenvalues $(\lambda_1(A),
\dots, \lambda_{n-1}(A),0)$. Then a straightforward calculation  shows that
$$F_2(A)=-\frac{1}{\sqrt n} \sum_{1\leq i<j\leq n-1} \left(\lambda_i(A) -\lambda_j(A)\right)^2\leq 0$$
as wished. It follows that
$$\int_{\mathbb T^n} F_2(A) \dif x\leq F_2\left(\int_{\mathbb T^n} A \dif x\right)$$
for any $A\in C^\infty(\mathbb T^n, \textup{Sym}_n)$ such that $\Div A=0.$
\end{remark}

We now show that the exponents obtained in Theorem \ref{thm:quasiconcavity} are optimal for the quasiconcavity of the integrands $\rho_k^*.$ Our argument is based on the classical fact \cite{Tartar1979} that if an integrand $F$ is $\Div$-quasiconcave then it is necessarily concave along any line parallel to a singular matrix, that is, for any $A\in \textup{Sym}_n,$
$$t\mapsto F(A+t X) \text{ is concave, whenever } \det X=0.$$
Although this fact is typically proved in the case where the matrix fields are unconstrained, the proof extends to the case where the integrand is $\Div$-quasiconcave only with respect to fields which take values in $\Gamma_k^*$, see e.g.\ \cite[\S1.2]{Serre2018a} for a similar argument.

We will prove the following result:

\begin{proposition}\label{prop:lambdaconcavity}
    For $k=2, \dots, n$, the integrands $(\rho_k^*)^\alpha\colon \Gamma_k^*\to [0,+\infty)$ are $\Lambda_{\Div}$-concave if and only if $\alpha\le \frac{k}{k-1}$.
\end{proposition}

\begin{proof}
It suffices to prove the `only if' part of proposition, as the `if' part is a consequence of Theorem \ref{thm:quasiconcavity}. It is enough to show that if $(\rho_k^*)^\alpha$ is $\Lambda_{\Div}$-concave on diagonal matrices then $\alpha\le\frac{k}{k-1}$. 

Let us consider, for $t\geq 0,$ the straight line $(A_t)_{t\geq 0}$, where $A_t=\textup{diag}(1+t,1,\dots, 1)\in \Gamma_n$. For any $A\in \Gamma_k$, we have
$$\rho_k(A)^{k-1} \nabla \rho_k(A) = \frac{ 1}{k} \nabla F_k(A).$$
It follows that the image of the line $(A_t)_{t\geq 0}$ under $\rho_k^{k-1}(\cdot)\nabla \rho_k(\cdot)$ is a straight line parallel to $\Lambda_{\Div}=\{A\in \textup{Sym}_n: \det A =0\}.$ Indeed,
$${n \choose k}\nabla F_k(A_t)=\textup{diag}\left(a_{n,k} , a_{n,k} + q_{n,k} t, \dots, a_{n,k} + q_{n,k} t\right)$$
where
$$a_{n,k}={n-1 \choose k-1}, \quad q_{n,k}={n-2 \choose k-2}.$$
According to Lemma \ref{lemma:attainment}, we have 
\begin{align*}
    n^\alpha \rho_k^*(\rho_k(A_t)^{k-1} \nabla \rho_k(A_t))^\alpha = \rho_k(A_t)^{(k-1)\alpha} & = c_{n,k}^{-(k-1)\alpha}\left({n-1\choose k-1} (1+t) + {n-1\choose k}\right)^{\frac{(k-1)\alpha}{k}}\\
    & = (b_{n,k}+m_{n,k} t)^{\frac{(k-1)\alpha}{k}}
\end{align*}
where  $c_{n,k}\equiv {n \choose k}^\frac 1 k$ and $b_{n,k}, m_{n,k}>0$ are appropriately chosen constants.   
By the previous paragraph, if $(\rho_k^*)^\alpha$ is $\Lambda_{\Div}$-concave then the function $t\mapsto \rho_k^*(\rho_k(\lambda_t)^{k-1} \nabla \rho_k(\lambda_t))^\alpha$ is concave, and so the function $t\mapsto (b_{n,k}+m_{n,k} t)^{(k-1)\alpha/k}$ must be sublinear. Thus we must have $\alpha \leq \frac{k}{k-1}$, as wished.
\end{proof}

\section{Elliptic estimates for the divergence}\label{sec:div}
In the previous section we saw that divergence-free fields constrained to take values in the cones $\Gamma_k^*$ have improved integrability. A natural question, especially in light of Serre's inequality \eqref{eq:serre}, see \cite{Serre2018a,Serre2019}, is whether a similar result holds for fields which do not satisfy the exact constraint $\Div A =0$, but instead have divergence bounded in some Lebesgue space. In other words, do we have elliptic estimates for
$$\mathcal K =\Gamma_k^*, \qquad \cala=\Div \text{ or } \cala=\Div^2?$$
More precisely, for $\A$ as above, do we have estimates
\begin{equation}
    \label{eq:goaldiv}
    \|\rho_k^*(A)\|_{L^q(\B^n)}\leq C \|\A A\|_{L^p(\B^n)} \quad \text{for } A\in C^\infty_c(\B^n,\Gamma_k^*)\end{equation}
for some set of exponents $1^*\leq q\leq q_{\max}$ and $p\geq 1$ sufficiently large?
As we will see below, the answer is positive. Through a nonlinear duality argument with  solutions of an appropriate $k$-Hessian equation, we will show that \eqref{eq:goaldiv} holds in the full range of exponents when $\A=\Div^2$ and a partial set of exponents (including both $q_{\max}$ and $1^*$) when $\A=\Div$, see Remark \ref{rmk:interpol}. Moreover, in the terminology of Conjecture \ref{conj:big}, we show that
\begin{equation}
    \label{eq:qmaxdiv}
q_{\max}(\Div,\Gamma_k^*)=q_{\max}(\Div^2,\Gamma_k^*)=\frac{k}{k-1}.
\end{equation}
The second equality in particular implies that Conjecture 1.9 in \cite{Arroyo-Rabasa2021a} is false. 

\subsection{Elliptic estimates}
Since $\Gamma_k^*\subset \Gamma_n=\textup{Sym}^+_n$, Serre's estimate \eqref{eq:serre} implies  
\begin{equation}
    \label{eq:SerreDiv}
\|\rho_k^*(A)\|_{L^{1^*}}\leq\|\rho_n(A)\|_{L^{1^*}}\leq C\|\Div A\|_{L^1} \quad \text{for } A\in C_c^\infty(\R^n,\Gamma_k^*),
\end{equation}
 as $\rho_k^*(A)\leq \rho_n^*(A)=\rho_n(A)$ by Lemma \ref{lemma:rhokorder}. Note that estimate \eqref{eq:SerreDiv} already yields \eqref{eq:goaldiv} when $q=1^*$.

We now establish \eqref{eq:goaldiv} at the other endpoint, i.e.\ when $q=\frac{k}{k-1}$:

\begin{proposition}\label{prop:Div}
     Let $2\leq k \leq n$ and $p>\frac{nk}{nk-(n-k)}=\big(\frac{k}{k-1}\big)_*$. Then there exists $C>0$ such that
    \beq\label{ineq:rhokstarend}
    \|\rho_k^*(A)\|_{L^{\frac{k}{k-1}}(\B^n)}\leq C\|\Div A\|_{L^p(\B^n)}\quad\text{for }A\in C_c^\infty(\B^n,\Gamma_k^*).
\eeq
Thus, if $\mathcal K\subset \mathrm{int}\,\Gamma_k^*\cup\{0\}$ is a closed convex cone, we have
 $$
    \|A\|_{L^{\frac{k}{k-1}}(\B^n)}\leq C\|\Div A\|_{L^p(\B^n)}\quad\text{for }A\in C_c^\infty(\B^n,\mathcal K).
$$
\end{proposition}
Recall that $p_*=\frac{np}{n+p}$ is the exponent such that $(p_*)^*=p$. Moreover, we have that $(p')_*=(p^*)'$ whenever $1\leq p\leq n$.

\begin{proof}
 Let $A\in C_c^\infty(\B^n,\Gamma_k^*)$ and  extend $A$ by zero to the ball $2\B^n$. We let
$$f= \rho_k^*(A)^{\frac{1}{k-1}}
$$
and apply Theorem \ref{thm:kHessExistence} to obtain a $k$-admissible function
$\psi\in C^{1,1}(\overline{2\B^n})$  such that 
$$
\begin{cases}
\rho_k(\D^2\psi)=f & \text{ in }2\B^n,\\
\psi=0 & \text{ on }\d(2\B^n).
\end{cases}
$$
In particular, $\D^2\psi\in\Gamma_k$ a.e~ in $2\B^n$ and $\psi$ is subharmonic, cf.\ Lemma \ref{lemma:rhokorder}, and therefore non-positive.
By Proposition \ref{prop:WS}, as $p'<k^*=\frac{nk}{n-k}$ by assumption, we have
$$\|\D\psi\|_{L^{p'}(\B^n)}{\leq C} \int_{2\B^n}|\psi|\dif x.$$
By Lemma \ref{lemma:Lpest}, $\psi$ satisfies the estimate
$$
\int_{2\B^n}|\psi|\dif x{\leq C} \|f\|_{L^{k}(2\B^n)}=\|f\|_{L^{k}(\B^n)}= \left(\int_{\B^n}\rho_k^*(A)^{\frac{k}{k-1}}\dif x\right)^{1/k},
$$
where we have used that $f$ is compactly supported inside $\B^n$ in the first equality.

We then estimate using Proposition \ref{prop:dualitypair} and integration by parts:
\begin{align*}
    \int_{\B^n} \rho_k^*(A)^{\frac{k}{k-1}}&=\int_{\B^n} \rho_k^*(A)\rho_k(\D^2\psi){\leq C} \int_{\B^n} \langle A, \D^2\psi\rangle  =-\int_{\B^n} \Div(A)\cdot \D\psi\\
    &{\leq C} \|\D\psi\|_{L^{p'}(\B^n)}\|\Div(A)\|_{L^p(\B^n)}{\leq C} \left(\int_{\B^n}\rho_k^*(A)^{\frac{k}{k-1}}\right)^{1/k} \|\Div(A)\|_{L^p(\B^n)}.
\end{align*}
The first inequality in the statement follows.
The second inequality follows by observing that for $A\in\mathcal{K}$ we have $|A|{\leq C} \rho_k^*(A)$.
\end{proof}

In fact, the range of exponents in Proposition \ref{prop:Div} is optimal, as we show next.

\begin{proposition}\label{prop:endpointk<n}
Let $2\leq k< n$ and $p=\frac{nk}{nk-(n-k)}=\big(\frac{k}{k-1}\big)_*$. Then, for any $j\in\mathbb{N}$, there exists a matrix field $A_j\in C_c^\infty(\B^n;\Gamma_k^*)$ such that
$$ \|\rho_k^*(A_j)\|_{L^{\frac{k}{k-1}}(\B^n)}\geq j\|\Div A_j\|_{L^p(\B^n)}.$$
\end{proposition}
\begin{proof}
Let $\eta(x)\geq 0$ be a standard radial mollifier. Given $\eps>0$, consider the solution to the problem
$$\begin{cases}
F_k(\D^2 w^\eps_k)=\eta^\eps(x)+\eps=\frac{1}{\eps^n}\eta(\frac{x}{\eps}) +\eps & \text{ in }\B^n,\\
w_k=0 & \text{ on }\partial \B^n.
\end{cases}$$
which, according to \cite[Equations (3.2), (3.16)]{Trudinger1997b}, is given by
$$w_k^\e(x)=-\Big(\frac{n}{{n\choose k}}\Big)^{\frac{1}{k}}\int_{|x|}^1s^{1-\frac{n}{k}}\left(\int_0^s \Big (\frac{1}{\eps^n}\eta(t/\e)+\eps\Big)t^{n-1}\dif t\right)^{\frac{1}{k}}\dif s.$$
Thus $w_k^\e$ is smooth, radial, and $\D^2 w_k^\e(x)\in \textrm{int}\,\Gamma_k$ for all $x\in \B^n$. Moreover, writing $v^\eps_k(|x|)\equiv w_k^\eps(x)$, for $r\geq \e$ we have the estimates
\beqs
|(v^\eps_k)'(r)|{\leq C} r^{1-\frac{n}{k}},\qquad |(v^\eps_k)''(r)|{\leq C}
r^{-\frac{n}{k}},
\eeqs
where the constants depend on $n$ and $k$ but are independent of $\eps<1$. In particular, using the identity
$$\D^2 w_k^\eps = \frac{(v^\eps_k)'(r)}{r}I_n+\Big((v^
\eps_k)''(r)-\frac{(v^\eps_k)'(r)}{r}\Big)\frac{x}{|x|}\otimes\frac{x}{|x|},$$
we find that
\beq\label{eq:D2wkepsbd}
r=|x|\geq \e \quad \implies \quad |\D^2w_k^\eps(x)|{\leq C} r^{-\frac{n}{k}}.
\eeq

We define a field $$\widetilde{A^\eps}\equiv \nabla F_k(\D^2 w^\eps_k)\in C^\infty(\B^n,\Gamma_k^*)$$
which is well-defined as a map into the interior of $\Gamma_k^*$ by Lemma \ref{lemma:attainment}. For $\delta>0$, consider the radial, scalar function
\begin{equation}
    \label{eq:defphidelta}
\phi^\delta(x)\equiv \begin{cases}
\log\log(1+\frac{1}{|x|}) & |x|\geq\de,\\
\log\log(1+\frac{1}{\de}) & |x|<\de,
\end{cases} 
\end{equation}
and, without risk of confusion, let us also denote by $\phi^\de$ a smooth approximation to this function.
We now set
$$A^{\eps,\de}(x)\equiv \phi^\de(x)\widetilde{A^{\eps}}(x)$$
and observe that, when $|x|\geq \de$, for $i=1,\ldots,n$
$$(\Div A^{\eps,\de})_i=\D\phi^\de\cdot (\widetilde{A^{\eps}})_i+\phi^\de(\Div \widetilde {A^{\eps}})_i=-\frac{1}{r(1+r)\log(1+\frac{1}{r})}\frac{x}{r}\cdot(\widetilde{A^{\eps}})_i $$
as $\Div \widetilde{A^\eps}=0$ by Proposition \ref{prop:divrhok*}, while for $|x|<\de$, $\Div A^{\eps,\de}\equiv 0$. We assume in addition that $0<\eps<\de$. Using the bounds from \eqref{eq:D2wkepsbd} and the $k-1$ homogeneity of $\nabla F_k$, we estimate
\begin{align*}
\|\Div A^{\eps,\de}\|^p_{L^{p}(\B^n)}  {\leq C} \int_\de^1 \frac{1}{r^p|\log(1+\frac{1}{r})|^p}r^{-\frac{n(k-1)p}{k}}r^{n-1}\dif r
 =\int_0^1\frac{1}{r|\log(1+\frac{1}{r})|^p}\dif r\leq C,
\end{align*}
where we have used the definition of $p$ to simplify the exponent and the fact that $p>1$ to see convergence of the integral with a bound independent of $\eps>0$. Note that it is in this last step that we crucially use the assumption $k<n$: when $k=n$, we have $p=1$ and the integral diverges.

In addition we calculate
\begin{align*}
\rho_k^*(A^{\eps,\de})^{\frac{k}{k-1}} & =\big(\phi^\de(x)\big)^{\frac{k}{k-1}}\rho_k^*(\nabla F_k(\D^2 w_k^\eps))^{\frac{k}{k-1}}\\
& \approx_{k,n}\big(\phi^\de(x)\big)^{\frac{k}{k-1}}F_k(\D^2w_k^\eps)
=\big(\phi^\de(x)\big)^{\frac{k}{k-1}}(\eta^\eps(x)+\e),
\end{align*}
where we have used the homogeneity of $\rho_k^*$ and Lemma \ref{lemma:attainment}.
Therefore, taking a non-negative continuous test function $\varphi\in C_c^0(\B^n)$, and recalling that $0<\eps<\de$, we test 
\beqas
\int_{\B^n}\rho_k^*(A^{\eps,\de})^{\frac{k}{k-1}}\varphi =&\, \int_{\B^n}\big(\phi^\de(x)\big)^{\frac{k}{k-1}}(\eta^\eps(x)+\eps)\varphi(x)\dif x\\
=&\,\log\log\Big(1+\frac{1}{\de}\Big)\int_{B_\eps(0)}\eta^\eps(x)\varphi(x)\dif x+\eps\int_{\B^n}\phi^\de(x)\varphi(0)\dif x
\eeqas
where we have used that the support of $\eta^\eps$ is $B_\eps$, on which $\phi^\de$ is constant. Then, as $\eta^\eps$ is an approximation to the identity, by taking  $\eps=\e(\delta)$ sufficiently small (note $\phi^\de$ is uniformly integrable), we deduce that 
$$\int_{\B^n}\rho_k^*(A^{\eps,\de})^{\frac{k}{k-1}}\varphi \geq \frac12 \log\log\Big(1+\frac{1}{\de}\Big)\varphi(0)$$
and thus we conclude that 
$$\|\rho_k^*(A^{\eps(\delta),\de})^{\frac{k}{k-1}}\|_{L^{\frac{k}{k-1}}(\mathbb B^n)}\to\infty \text{ as }\de\to0$$
as wished.
\end{proof}

It is interesting to inspect formally the above proof when $k=1$, as it served as motivation for the general case $k<n$. We have that $\nabla F_1$ is constant and thus $A^{\e,\delta}$ can be identified with the scalar field $\phi^\delta$ defined in \eqref{eq:defphidelta}. Formally, when $k=1$ the relation $p^*=\frac{k}{k-1}$ specifies $p=n$, and thus the above argument yields the failure of the embedding $W^{1,n}(\mathbb B^n)\subset L^\infty(\mathbb B^n).$

On the other hand, in the other limiting case, when $k=n$, this counterexample no longer holds, as the constructed field will not have an $L^1$ divergence due to the failure of integrability of $r^{-1}\log(1+\frac1r)^{-1}$ on $(0,1)$. This reflects the fact that the Monge-Amp\`ere equation (the dual equation when $k=n$) enjoys an unusual gain of regularity: when written in $1$-homogeneous form,
$$\rho_n(\D^2 u)=f\geq0\in L^n,$$
with suitable boundary condition, the solution $u$ typically satisfies a Lipschitz bound that improves the usual failure of the embedding 
$$W^{2,n}\not\hookrightarrow W^{1,\infty}$$
that obstructs such an estimate in the linear case (and, indeed, for $k<n$).

\begin{remark}[Interpolation]\label{rmk:interpol}
Since the space $\{A\in L^p(\B^n,\Gamma_k^*): \Div A\in L^p(\B^n)\}$ is \textit{nonlinear}, we are unable to interpolate between the end-point estimates in Proposition \ref{prop:Div} and  Serre's inequality \eqref{eq:SerreDiv}, and so we do not obtain the conjectural estimates
\begin{equation}
    \label{eq:fullrange}
\|\rho_k^*(A)\|_{L^q(\B^n)}\leq C \|\Div A\|_{L^p(\B^n)}, \qquad \text{whenever } 1^*< q < \frac{k}{k-1},\, q_*<p.
\end{equation}
Nonetheless, the monotonicity $\rho_k^*\leq \rho_l^*$ for $k\leq l$ implies a partial result towards \eqref{eq:fullrange}: for each $k\leq l\leq n$, $l\in\mathbb N$, we have the estimate
$$\|\rho_k^*(A)\|_{L^{\frac{l}{l-1}}(\B^n)}\leq C\|\Div A\|_{L^p(\B^n)},$$
where $p>\big(\frac{l}{l-1}\big)_*$. In the case $l=n$, we may take the critical estimate with $p=1$.
\end{remark}

A key difficulty in proving \eqref{eq:fullrange} directly, i.e.\ by the method of proof used to treat the end-point cases, is that the needed gradient estimates for solutions of $k$-Hessian equations are not known. When $\A=\Div^2$, this can be remedied by integrating by parts twice: by doing so, we only require $L^p$ estimates on the solution of the $k$-Hessian equation, and not on its gradient.

\begin{proposition}\label{prop:Div^2}
     Let $2\leq k \leq n$. We have the estimate 
    $$
    \|\rho_k^*(A)\|_{L^{q}(\B^n)}\leq C\|\Div^2A\|_{L^p(\B^n)}\quad\text{for }A\in C_c^\infty(\B^n,\Gamma_k^*),
$$
whenever either of the following conditions hold:
\begin{enumerate}
    \item\label{itm:div2a} $p>1$, $k\leq \frac{n}{2}$, and $q=\min\{p^{**},\frac{k}{k-1}\}$;
    \item\label{itm:div2b} $p=1$, $q<1^{**}$, $q\leq \frac{k}{k-1}$.
\end{enumerate}
Thus, if $\mathcal K$ is a closed convex cone contained in $\mathrm{int}\,\Gamma_k^*\cup\{0\}$, then
$$
    \|A\|_{L^{q}(\B^n)}\leq C\|\Div^2A\|_{L^p(\B^n)}\quad\text{for }A\in C_c^\infty(\B^n,\mathcal K).
$$
\end{proposition}
It may seem that the restriction $k\leq \frac{n}{2}$ is arbitrary in \ref{itm:div2a}, but in fact, when $k>\frac{n}{2}$, we automatically have $\frac{k}{k-1}<1^{**}$, so there is no loss in restricting to the case $p=1$ and we fall in case \ref{itm:div2b}. We also note that the restriction $q\leq \frac{k}{k-1}$ is necessary, as can be seen from Proposition \ref{prop:optimal_div} below. 
\begin{proof}
We begin by setting 
$$f=  \rho_k(A)^{q-1}
$$
and  we apply Theorem \ref{thm:kHessExistence} to find a $k$-admissible function $\psi\in C^{1,1}(\overline{\B^n})$ solving
$$\begin{cases}
\rho_k(\D^2\psi)=f & \text{ in }\B^n,\\
\psi=0 & \text{ on }\d \B^n.
\end{cases}
$$
In particular,  $\D^2\psi\in\Gamma_k$ a.e.\ in $\B^n$. We now apply the $L^p$ estimates of Lemma \ref{lemma:Lpest}, noting $q'\geq k$ by assumption, to estimate 
$$
\|\psi\|_{L^{p'}(\B^n)}{\leq C} \|f\|_{L^{q'}(\B^n)}= \left(\int_{\B^n}\rho_k^*(A)^q\right)^{1/q'};
$$
the assumed restrictions on the exponents $p$ and $q$ ensure that Lemma \ref{lemma:Lpest} is applicable.
We then estimate using Proposition \ref{prop:dualitypair} and integration by parts twice:
\begin{align*}
    \int_{\B^n} \rho_k^*(A)^{q}=&\int_{\B^n} \rho_k^*(A)\rho_k(\D^2\psi){\leq C} \int_{\B^n} \langle A, \D^2\psi\rangle  =\int_{\B^n} \psi\Div^2(A)\\
    &\leq  \|\psi\|_{L^{p'}(\B^n)}\|\Div^2(A)\|_{L^p(\B^n)}{\leq C} \left(\int_{\B^n}\rho_k^*(A)^q\right)^{1/q'} \|\Div^2(A)\|_{L^p(\B^n)}.
\end{align*}
The first inequality now follows directly and the second is a consequence of the fact that if $\mathcal{K}$ is a closed convex cone in $\textup{int}\,\Gamma_k^*\cup\{0\}$, then $|A|{\leq C}\rho_k^*(A)$.
\end{proof}

\subsection{The case of exact constraints: optimality and higher integrability}
\label{sec:curloptimal}

The quasiconcavity result in Theorem \ref{thm:quasiconcavity} suggests that, in order to see that the exponent $\frac{k}{k-1}$ in both Propositions  \ref{prop:Div} and \ref{prop:Div^2} is optimal, it should be enough to consider $\Div$-free fields. This is indeed the case: 

\begin{proposition}\label{prop:optimal_div}
Let $2\leq k\leq n$ and $\e>0$. There is a closed convex cone $\mathcal K^\e\subset \textup{int\,} \Gamma_k^*\cup\{0\}$ and a matrix field $A_\e\in C^\infty(\B^n\backslash\{0\},\mathcal K^\e)$ such that
$$\Div A_\e =0, \qquad A_\e\in L^{\frac{k}{k-1}}\backslash L^{\frac{k}{k-1}+\e}(\B^n).$$
\end{proposition}

\begin{proof}
With $r=|x|$, let $u_\e=u_\e(r)$ be a function to be determined that satisfies $\D^2u\in L^k(\B^n,\Gamma_k)$ and take
    $$A_\e(x)\equiv \frac 1 k (\nabla \rho_k^k)(\D^2 u_\e(x))= \rho_k^{k-1}(\D^2 u_\e(x)) \nabla \rho_k(\D^ 2 u_\e(x)).$$
Employing Lemma \ref{lemma:attainment} we see that $A_\e$ takes values in $\Gamma_k^*$. Since $\rho_k^k$ is, up to a multiplicative constant, the sum of the $k\times k$ principal minors, Proposition \ref{prop:divrhok*} shows that $A_\e$ is divergence-free, since $\D^2 u_\e$ is evidently $\Curl$-free.

Writing $\nu=\frac{x}{r}$, we have
$$\D^2 u_\e(x) = \frac{\dot u_\e(r)}{r} \textup{Id} + \Big(\ddot u_\e(r)-\frac{\dot u_\e(r)}{r}\Big) \nu \otimes \nu,$$
and so in particular the eigenvalues of $\D^2 u_\e(x)$ are $(\ddot u_\e(r), \dot u_\e(r)/r, \dots, \dot u_\e(r)/r)$.
Recalling \eqref{eq:defsigmak}, it follows that
$$\sigma_k(\D^2 u_\e)={n-1\choose k} \Big(\frac{\dot u_\e(r)}{r}\Big)^k + {n-1\choose k-1} \ddot u_\e(r) \Big(\frac{\dot u_\e(r)}{r}\Big)^{k-1};$$
here we interpret ${n-1 \choose n}=0$, in the case $k=n$.
We take 
$$\alpha \equiv \frac{n}{k+\e(k-1)}$$
and we split the analysis into two cases.

\textbf{Case 1: $k\geq \frac n 2$.} In this case we have $0< \alpha <2$ and we take $u_\e(r)\equiv r^{2-\alpha}$. Our claim is that $u_\e$ is strictly $k$-convex (i.e., $\sigma_j(\D^2 u_\e)>0$ for $1\leq j\leq k$). If $k=n$ then it is easy to see that in fact $0<\alpha<1$ for $\e>0$, so $u_\e$ is clearly convex. If $k<n$ note that for $j\leq k$, since ${n-1 \choose j-1} = \frac{j}{n-j} {n-1 \choose j}$, we can calculate
$$\sigma_j(\D^2 u_\e) = {n -1 \choose j} \Big(\frac{\dot u_\e(r)}{r}\Big)^{j-1} (2-\alpha) r^{-\alpha} \left(1-(\alpha-1)\frac{j}{n-j}\right).$$
We claim that $\sigma_j(\D^2 u_\e)>0$ if $j\leq k$. Indeed, since $\alpha<2$, it suffices to examine the last term in the above product, and as $\e>0$ we have
$$(\alpha-1)\frac{j}{n-j} = \frac{j}{n-j} \left(\frac{n-k-\e(k-1)}{k+\e(k-1)}\right) < \frac{j}{n-j} \frac{n-k}{k} \leq \frac{k}{n-k}\frac{n-k}{k} = 1,$$
as wished. It follows that $\D^2 u_\e \in \Gamma_k$ a.e.\ and so, since $u_\e\in W^{2,k}(\mathbb B^n)$, we conclude that $u_\e$ is $k$-convex for all $k\geq \frac n 2.$ In fact, we have shown that $\D^2u_\e$ lies in a closed cone strictly contained in $\textup{int}\,\Gamma_k\cup\{0\}$.

 Using Lemma \ref{lemma:attainment} we see that
$$\rho_k(A_\e(x))=\frac 1 n \rho_k^{k-1}(\D^2 u_\e(x))\approx_{k,n} \sigma_k^{\frac{k-1}{k}}(\D^2 u_\e)\approx_{k,n} r^{-\alpha(k-1)}.$$
In particular, using spherical coordinates we calculate
$$\int_{\B^n} |\rho_k(A_\e(x))|^{\frac{k}{k-1}+\e}  \dif x
\approx_{k,n} \int_0^1 r^{n-1} r^{-\alpha(k+\e(k-1))} \dif r = \int_0^1 r^{-1} \dif r$$
by the definition of $\alpha$. Since $\rho_k(A_\e)\leq C(n,k) |A_\e|$, the conclusion follows.

\textbf{Case 2:  $2\leq k <\frac n 2$.} In this case, we take $u_\e(r)\equiv -r^{2-\alpha}$. If $\e$ is chosen small enough, more precisely if $0<\e<\frac{n-2k}{2(k-1)}$, we have $\alpha>2$. For $j\leq k$, we calculate
$$\sigma_j(\D^2 u_\e) = {n-1 \choose j} \Big(\frac{\dot u_\e(r)}{r}\Big)^{j-1} (\alpha-2) r^{-\alpha} \left(1-(\alpha-1)\frac{j}{n-j}\right)$$
and the conclusion follows by repeating the arguments above.
\end{proof}

\begin{proof}[Proof of Theorems \ref{thm:maindiv} and \ref{thm:div^2}]
The proof of Theorem \ref{thm:maindiv}  follows almost directly from Propositions \ref{prop:Div} and \ref{prop:optimal_div}, while the last claim in the statement is the content of Proposition \ref{prop:endpointk<n}. To verify the first non-inequality, we observe that, given an $\eps>0$, we may multiply the function $A^\eps$ of Proposition \ref{prop:optimal_div} by a test function $\rho\in C_c^\infty(\mathbb{B}^n,[0,1])$ such that $\rho\equiv 1$ on $\frac 1 2 \B^n$. Then, as $A^\eps$ is smooth and bounded on the annulus $\B^n\setminus \frac 1 2 \B^n$,
$$\|\Div(\rho A^\eps)\|_{L^\infty(\B^n)}=\| A^\eps\D\rho\|_{L^\infty(\B^n\setminus \frac 1 2 \B^n)}\leq C,$$
but $\rho A^\eps\not\in L^{{k}/{(k-1)}+\eps}(\B^n)$.
Taking a smooth sequence of approximations to $A^\eps$, we conclude. In particular, \eqref{eq:qmaxdiv} holds.

The proof of Theorem \ref{thm:div^2} follows along very similar lines, using Proposition \ref{prop:Div^2} instead of Proposition \ref{prop:Div}.
\end{proof}

We conclude this section by noting that, despite Proposition \ref{prop:optimal_div},  solenoidal fields uniformly inside $\Gamma_k^*$ have higher integrability:

\begin{corollary}\label{cor:divUI}
Let $\mathcal K\subset\mathrm{int}\,\Gamma_k^*\cup\{0\}$ be a closed convex cone. Let $A_j\in C^\infty(\R^n,\mathcal K)$ be a sequence bounded in $L^1_{\locc}(\R^n)$ and such that $\Div A_j=0$. For a ball $B=B_R(x_0)$ and $\theta\in (0,1),$ we have
$$\left(\frac{1}{|\theta B|}\int_{\theta B}|A_j|^{\frac{k}{k-1}} \dif x\right)^{\frac{k-1}{k}}\leq C(\mathcal K, R,\theta) \frac{1}{|B|}\int_{B\setminus \theta B} |A_j| \dif x.$$
In particular, there exists $\delta=\delta(\mathcal K)$ such that $(A_j)$ is bounded in $L^{{k}/{(k-1)}+\delta}_{\locc}(\R^n)$, and so $(A_j)$ is locally $\frac{k}{k-1}$-uniformly integrable.
\end{corollary}

\begin{proof}
Fix a ball $B\subset \R^n$. If $\Div A=0$, we have, for a scalar field $\rho\in C_c^\infty(B)$,
$$
\Div^2(\rho A)=\di(\rho \Div A +A\D\rho)=\Div A\cdot \D\rho+\langle \D^2\rho, A\rangle=\langle\D^2\rho,A\rangle.
$$
We therefore have that $\Div^2(\rho A_j)\in C_c^\infty(B)$ is bounded uniformly in $L^1(B)$. Fix a constant $\tau \in (0,1)$ to be chosen later and choose $\rho$ non-negative  and such that $\rho=1$ in $\tau B$. Thus, letting $1<q<1^{**}$ be such that $q\leq \frac{k}{k-1}$,  Proposition \ref{prop:Div^2} implies
$$
\left(\frac{1}{|\tau B|}\int_{\tau B} |A_j|^q\dif x\right)^{1/q}\leq C\left(\frac{1}{|B|}\int_{B\setminus \tau B} |A_j|\dif x\right),
$$
with a constant that depends only on the radius of $B$ and $\tau$. If $\frac{k}{k-1}<1^{**}$ we are done. If $\frac{k}{k-1}\geq 1^{**}$, we take some $\tilde \tau \in (0,\tau)$ and $\tilde\rho\in C_c^\infty(\tau B)$ such that $\tilde \rho\equiv 1$ on $\tilde \tau B$. Repeating  the argument above we obtain, for $\tilde q=\min\{q^{**},\frac{k}{k-1}\}$,
$$
\left(\frac{1}{|\tilde \tau B|}\int_{\tilde \tau B} |A_j|^{\tilde q}\dif x\right)^{1/\tilde q}\leq C\left(\frac{1}{|B|}\int_{B\setminus \tilde \tau B} |A_j|\dif x\right),
$$
again with a constant depending only on the radius of $B$ and $\tau,\tilde \tau$.
Iterating this argument finitely many times until we arrive at the exponent $\frac{k}{k-1}$, and choosing $\tau,\tilde \tau, \dots,$ appropriately, we obtain the desired inequality.
The last claim in the statement of the corollary follows from Gehring's Lemma, see for instance \cite{Iwaniec2001}.
\end{proof}

\section{Elliptic estimates for the curl}
\label{sec:curl}

In this section we present nonlinear duality arguments which show that the elliptic estimates of Section \ref{sec:div} yield elliptic estimates for the pair
$$\A = \Curl, \qquad \mathcal K= \Gamma_k.$$
In fact, we will show that there are estimates
\begin{equation}
    \label{eq:goalcurl}
    \|A\|_{L^q(\B^n)}\leq C \|\Curl A \|_{L^p(\B^n)}\quad \text{for } A\in C^\infty_c(\B^n,\mathcal K),
\end{equation}
for some set of exponents $q$ such that $1^*\leq q\leq q_{\max}$, $p>q_*$, whenever $\mathcal K\subset \textup{int}\, \Gamma_k\cup\{0\}$ is a closed cone.
Moreover, in the terminology of Conjecture \ref{conj:big}, we will prove that
\begin{equation}
    \label{eq:qmaxcurl}
q_{\max}(\Curl,\Gamma_k)=k.
\end{equation}
In particular, combining the above results, we prove Theorem \ref{thm:maincurl}.

\subsection{Elliptic estimates}\label{subsec:curl}
We begin by establishing the critical version of \eqref{eq:goalcurl} at the endpoint $q=1^*$.

\begin{proposition}\label{prop:curlp=1}
For $2\leq k\leq n$, there is a constant $C$ such that
$$\|\rho_k(A)\|_{L^{1^*}(\B^n)}\leq C \|\Curl A \|_{L^1(\B^n)}, \quad \text{for }A\in C^\infty_c(\B^n,\Gamma_k).$$
\end{proposition}

\begin{proof}
By Lemma \ref{lemma:rhokorder} it suffices to prove the proposition when $k=2$, as $A\in\Gamma_k\subseteq\Gamma_2$ and we have $0\leq\rho_k(A)\leq\rho_2(A)$. We first note that, by Lemmas \ref{lemma:attainment} and \ref{lemma:rhokorder},
\begin{align*}
\rho_2(A)^{\frac{n}{n-1}} & =c_{n}\rho_2^*\big(\rho_2(A)\nabla\rho_2(A)\big)^{\frac{n}{n-1}}\\ 
& \leq c_{n}\rho_n^*\big(\rho_2(A)\nabla\rho_2(A)\big)^{\frac{n}{n-1}}=c_n\det\big(\rho_2(A)\nabla\rho_2(A)\big)^{\frac{1}{n-1}},
\end{align*}
where $c_n=n^{\frac{n}{n-1}}$.
We apply \eqref{eq:serre} and Proposition \ref{prop:divrhok*} to estimate
\beqa
\int_{\B^n}\rho_2(A)^{\frac{n}{n-1}}{\leq C}&\,\int_{\B^n}\det\big(\rho_2(A)\nabla\rho_2(A)\big)^{\frac{1}{n-1}}\\
{\leq C} &\,\bigg(\int_{\B^n}\Big|\Div\Big(\rho_2(A)\nabla\rho_2(A)\Big)\Big|\bigg)^{\frac{n}{n-1}}
{\leq C} \,\bigg(\int_{\B^n}|\Curl A|\bigg)^{\frac{n}{n-1}}.
\eeqa
Note that $\rho_2(A)\nabla \rho_2(A)\in \Gamma_2^*\subset \Gamma_n^*$ by Lemma \ref{lemma:attainment}, so \eqref{eq:serre} is indeed applicable.
\end{proof}

We next establish \eqref{eq:goalcurl} at the other endpoint, i.e.\ when $q=k$. Unfortunately, in this case we are unable to estimate the nonlinear quantity $\rho_k(A)$ and are therefore forced to work with fields which lie uniformly inside $\Gamma_k$.

\begin{proposition}\label{prop:curlmain}
     Let $2\leq k \leq n$ and $p>\frac{nk}{n+k}=k_*$ if $k\leq n-1$ or $p=\frac{n}{2}=n_*$ if $k=n$. For a closed convex cone $\mathcal K\subset \mathrm{int}\,\Gamma_k\cup\{0\}$, we have
the estimate
\begin{align}
    \label{ineq:curlk}
    \|A\|_{L^{k}(\B^n)}&\leq C\|\Curl A\|_{L^p(\B^n)}\quad\text{for }A\in C_c^\infty(\B^n,\mathcal K)
\end{align}
In fact, this improves to
\begin{equation}
    \label{eq:LlogLsym}
    \|A\|_{L^k\log L(\B^n)}\leq C\|\Curl A\|_{L^p(\B^n)}\quad\text{for }A\in C_c^\infty(\B^n,\mathcal K).
\end{equation}
\end{proposition}

The exponent $k$ is optimal, as will be shown below in Proposition \ref{prop:Curl_optimal}. As in Proposition \ref{prop:curlp=1}, the proof of Proposition \ref{prop:curlmain} relies on Proposition \ref{prop:divrhok*}, which allows us to exploit a non-linear duality between the operator-cone pairs $(\Div,\Gamma_k^*)$ and $(\Curl,\Gamma_k)$.

\begin{proof}
Let $A\in C_c^\infty(\B^n,\mathcal K)$ and recall the identity $\rho_k^*(n\nabla\rho_k(A))=1$ from Lemma \ref{lemma:attainment}. By the positive 1-homogeneity of $\rho_k^*$, we find 
\begin{equation}
    \label{eq:auxrhokcurl}
\rho_k(A)^k=n^{\frac{k}{k-1}}\rho_k^*\big(\rho_k^{k-1}(A)\nabla\rho_k(A)\big)^{\frac{k}{k-1}}.
\end{equation}
Suppose first that $k\leq n-1$. We let $r$ be such that $$\frac{1}{p}+\frac{k-2}{k}=\frac{1}{r};$$ then $r>\big(\frac{k}{k-1}\big)_*$, since by assumption $p>k_*$. This ensures that Proposition \ref{prop:Div} is applicable, since in addition  we have $\rho_k^{k-1}(A)\nabla \rho_k(A)\in \Gamma_k^*$ by Lemma \ref{lemma:attainment}. Thus, applying in order \eqref{eq:auxrhokcurl}, Proposition \ref{prop:Div}, Corollary \ref{cor:divrhok*} and H\"older's inequality, we get
\beqa
\int_{\B^n}\rho_k(A)^k{\leq C}&\,\int_{\B^n}\rho_k^*\big(\rho_k^{k-1}(A)\nabla\rho_k(A)\big)^{\frac{k}{k-1}}\\
{\leq C} &\,\bigg(\int_{\B^n}\Big|\Div\Big(\rho_k^{k-1}(A)\nabla\rho_k(A)\Big)\Big|^r\bigg)^{\frac{k}{r(k-1)}}\\
{\leq C} &\,\Big(\int_{\B^n}|\Curl A|^r|A|^{r(k-2)}\Big)^{\frac{k}{r(k-1)}}\\
{\leq C} &\,\|\Curl A\|_{L^p}^{\frac{k}{k-1}}\|A\|_{L^k}^{\frac{k(k-2)}{k-1}}.
\eeqa
 The conclusion follows by taking the $\frac{k-1}{k}$-th power on both sides and recalling that, in $\mathcal K$, we have $|A|{\leq C} \rho_k(A)$ in order to absorb the $\|A\|_{L^k}$ term on the left-hand side.

In the case $k=n$, observe that inequality \eqref{eq:serre} holds for the critical exponents and not just in the sub-critical range. The same proof as above now allows us to conclude with exactly $p=\frac{n}{2}$ on the right.

We now establish the improvement in \eqref{eq:LlogLsym}. Recall that $F_k=\rho_k^k$ and note that $F_k$ is a linear combination of $k\times k$ minors. It follows that $F_k$ is $\Curl$-quasiaffine (or a null Lagrangian), in the terminology of Appendix \ref{app:LlogL}. Thus we may apply Corollary \ref{cor:GRS}, together with the estimate $|A|^k{\leq C} F_k(A)$, valid for $A\in \mathcal K$,  to obtain
\begin{align*}
    \|A\|_{L^k\log L(\B^n)}\leq C \left(\|A\|_{L^k(\B^n)}+\|\Curl A\|_{{\dot{W}}{^{-1,p^*}}(\B^n)}\right)\leq C\|\Curl A\|_{L^p(\B^n)},
\end{align*}
where the last estimate follows from \eqref{ineq:curlk} and the embedding $L^p(\B^n)\hookrightarrow W^{-1,p^*}(\B^n)$.
\end{proof}

\begin{remark}[Interpolation]\label{rmk:curlinterp}
Similarly to Remark \ref{rmk:interpol}, for $1<l<k$, $l\in \mathbb N$, we are able to deduce that
$$\|A\|_{L^l(B)}\leq \|\Curl A\|_{L^p(B)} \quad \text{for } A\in C^\infty_c(B,\Gamma_k),$$
for $p>l_*$. If we knew that \eqref{eq:fullrange} holds then, by repeating the arguments above, we would obtain a full range of elliptic estimates:
$$\|A\|_{L^{q}(B)}\leq C \|\Curl A \|_{L^p(B)}, \qquad \text{whenever } 1^*\leq q \leq k,\, q_*<p.$$
\end{remark}

In the special case when $k=2$, we can actually obtain a slightly sharper form of Proposition \ref{prop:curlmain} and in fact we manage to estimate the quantity $\rho_2(A)$ directly in $L^2$ in terms of a weaker operator than $\Curl$. To do so it is convenient to first fix some notation.

\begin{definition}\label{notation:A_2}
We introduce the notation $\mathcal Q\equiv \Div-\D \tr$, for which we have the relation $\mathcal Q A =L \Curl (A^\top)$ for a linear map $L\colon \V \to \R^n$, where $\V$ is an appropriate finite-dimensional vector space. We also have that 
$$\mathcal Q A=\Div(L_{\mathcal Q} A), \quad \text{where } L_{\mathcal Q}A=A-(\tr A)I_n \text{ for }A\in\R^{n\times n}.$$
\end{definition}
The linear map $L$ is easily computed by working in coordinates, see \eqref{eq:L2coordinates}. We also have the bilinear inequality 
\begin{align}\label{eq:bilinear}
    \rho_2(A)\rho_2(B)\leq C\langle A,L_{\mathcal Q}B\rangle\quad\text{for }A,\,B\in\Gamma_2,
\end{align}
see for instance \cite{Wang2009}. We next obtain the following refinement of Proposition \ref{prop:curlmain}.

\begin{proposition}\label{prop:curl}
Suppose that $p>2_*=\frac{2n}{n+2}$.
Then
    \begin{align}\label{eq:rho_curl_inequ1}
    \|\rho_2(A)\|_{L^{2}(\B^n)}\leq C\|{\mathcal Q A}\|_{L^p(\B^n)}\leq C\|\Curl A\|_{L^p(\B^n)}\quad\text{for }A\in C_c^\infty(\B^n,\Gamma_2).
    \end{align}
\end{proposition}

\begin{proof}
As in the proof of Proposition \ref{prop:Div}, extend $A\in C_c^\infty(\B^n,\Gamma_2)$ by zero to $2\B^n$ and define the function 
$$f=\rho_2(A).$$
Applying Theorem \ref{thm:kHessExistence}, let $\psi\in C^{1,1}(\overline{2\B^n})$ be the $k$-admissible solution of
$$\begin{cases}
\rho_2(\D^2\psi)=f & \text{ in }2\B^n,\\
\psi=0 & \text{ on }\d(2\B^n).
\end{cases}$$
By Proposition \ref{prop:WS}, as $p'<\frac{2n}{n+2}$,
$$\|\D\psi\|_{L^{p'}(\B^n)}{\leq C} \int_{2\B^n}|\psi|{\leq C} \|f\|_{L^2(2\B^n)}=\|f\|_{L^2(\B^n)},$$
where the second inequality follows from Lemma \ref{lemma:Lpest} and we have used compact support of $f$ in $\B^n$ in the last equality. As $0\leq f=\rho_2(A)$, we have obtained
$$
\|\D\psi\|_{L^{p'}(\B^n)}{\leq C} \|\rho_2(A)\|_{L^{2}(\B^n)}.
$$
We then estimate using the inequality of \eqref{eq:bilinear} and integration by parts:
\begin{align*}
    \int_{\B^n} \rho_2(A)^2&=
    \int_{\B^n} \rho_2(A)\rho_2(\D^2\psi){\leq C} \int_{\B^n} \langle L_{\mathcal Q}A, \D^2\psi\rangle  =-\int_{\B^n} \Div(L_{\mathcal Q}A)\cdot \D\psi\\
    &{\leq C} \|\D\psi\|_{L^{p'}(\B^n)}\|{\mathcal Q}A\|_{L^p(\B^n)}{\leq C} \left(\int_{\B^n}\rho_2(A)^2\right)^{1/2} \|{\mathcal Q}A\|_{L^p(\B^n)}.
\end{align*}
Rearranging, the conclusion follows.
\end{proof}

\subsection{The case of exact constraints: optimality and higher integrability}

As in Section \ref{sec:curloptimal}, by considering $\A$-free fields we will show that the exponent $k$ in Proposition \ref{prop:curlmain} is optimal.

\begin{proposition}\label{prop:Curl_optimal}
    Fix $2\leq k \leq n$ and $\e>0$. There is a closed convex cone $\mathcal K^\e\subset \textup{int\,} \Gamma_k$ and  a function $\phi_\e \in C^\infty(\B^n\backslash\{0\})$ such that
$$\D^2 \phi_\e\in\mathcal K^\e \text{ in }\B^n\backslash\{0\}, \qquad  \D^2\phi_\e\in L^k(\B^n)\setminus L^{k+\varepsilon}(\B^n).$$
\end{proposition}

The proof relies on the following:
\begin{lemma}\label{lem:radial}
For each $\varepsilon>0$ there exists $K=K(\e)>1$ and a function $u_\e\in C^\infty(\B^n\backslash\{0\})$ such that
$$\D^2 u_\e \in \textup{Sym}^+_n\cap Q_n^+(K) \text{ a.e.\ in } \B^n,\qquad 
\D^2u_\e\in L^n(\B^n)\setminus L^{n+\varepsilon}(\B^n).
$$
\end{lemma}
\begin{proof}
Let $u_\e=u_\e(r)$ be radial. Then $\D u_\e(x) = \dot u_\e(r) \nu$ (where $\nu=\frac{x}{r}$) and
$$\D^2 u_\e(x) = \frac{\dot u_\e(r)}{r} \textup{Id} + \Big(\ddot u_\e(r)-\frac{\dot u_\e(r)}{r}\Big) \nu \otimes \nu.$$
It follows that $\det \D^2 u_\e (x) = \ddot u_\e(r) (\frac{\dot u_\e(r)}{r})^{n-1}$ and $\|\D^2 u_\e(x)\|= \max\{|\ddot u_\e|,|\frac{\dot u_\e}{r}|\}$. Fix $\e>0$ and take $u_\e(r)=r^{\frac{\e}{n+\e}+1}$; note that $u_\e$ is smooth away from the origin, as claimed. Substituting the definition of $u_\e$ in the above formulas we obtain
$$\frac{\|\D^2 u_\e(x)\|^n}{\det \D^2 u_\e(x)} = 1+ \frac{n}{\e}$$
where we used the fact that $\|\D^2 u_\e(x)\|=\frac{\dot u_\e(r)}{r}.$ In particular, $\D^2 u_\e \in Q_n^+(1+\frac n \e).$
Moreover, we compute
$$\int_{\B^n} \| \D^2 u_\e\|^p \dif x \approx_n
\int_0^1 r^{n-1} \left|\frac{\dot u_\e(r)}{r}\right|^p\dif r
= \Big(1+\frac{\e}{\e+n}\Big) \int_0^1 r^{\frac{n(\e+n-p)}{\e+n}-1} \dif r,
$$
and this integral is finite if and only if $p<n+\e.$
\end{proof}

\begin{proof}[Proof of Proposition \ref{prop:Curl_optimal}]
 Let $\phi_\e(x)=u_\e(x_1,\ldots,x_k)$, where $u_\e$ is given by Lemma \ref{lem:radial} for $n=k$. Then clearly $\D^2\phi_\e\in L^k\setminus L^{k+\varepsilon}(\B^n,\mathcal K^\e)$, where
$$
\mathcal K^\e\equiv \left\{
\left(\begin{matrix}
A_0 &0\\
0&0
\end{matrix}\right)\colon A_0\in \textup{Sym}^+_k \cap Q_k(K)
\right\}
$$
is a closed convex cone contained in $\mathrm{int\,}\Gamma_k$.
The proof of \eqref{eq:qmaxcurl} is similar to the proof of Proposition \ref{prop:optimal_div} and we omit the details.
\end{proof}

Despite Proposition \ref{prop:Curl_optimal}, irrotational fields uniformly inside $\Gamma_k$ have higher integrability:

\begin{corollary}\label{cor:curlUI}
Let $2\leq k \leq n$ and $\mathcal K\subset\mathrm{int}\,\Gamma_k\cup\{0\}$ be a closed convex cone. 
Consider a sequence $A_j\in C^\infty(\R^n,\mathcal K)$ which is bounded in $L^1_{\locc}(\R^n)$ and such that $\Curl A_j=0$, so $A_j=\D^2\phi_j$ for some $\phi_j\in C^\infty(\R^n)$.  For a ball $B=B_R(x_0)$ and $\theta\in (0,1)$, we have
$$\left(\frac{1}{|\theta B|}\int_{\theta B}|A_j|^{k} \dif x\right)^{\frac 1 k}\leq C(\mathcal K, R,\theta ) \frac{1}{|B|}\int_{B\setminus \theta B} |A_j| \dif x.$$
In particular, there is $\delta=\delta(\mathcal K)$ such that $(A_j)$ is bounded in $L^{k+\delta}_{\locc}(\R^n)$ and so $(A_j)$ is locally $k$-uniformly integrable.
\end{corollary}
\begin{proof}
We begin by noting that, as $\Gamma_k\subset\Gamma_l$ for each $2\leq l\leq k$, we may define a field $B_l=\nabla F_l(A)$, which by Proposition \ref{prop:diffeo} takes values in a closed cone $\mathcal K^*\subset\textup{int}\,\Gamma_l^*\cup\{0\}$. In addition, Proposition \ref{prop:diffeo} and a homogeneity argument show that
$$c(\mathcal K, l) |A|^{l-1}\leq |\nabla F_l(A)|\leq C(n,l) |A|^{l-1}.$$
By Proposition \ref{prop:divrhok*}, $\Div B_l= 0$.  Thus, applying Corollary \ref{cor:divUI} with $\theta=\frac 1 2$, we have
$$\|A\|_{L^{l}(B)}{\leq C} \|B_l\|_{L^{\frac{l}{l-1}}(B)}{\leq C} \|B_l\|_{L^1(2B\setminus B)}{\leq C} \|A\|_{L^{l-1}(2B\setminus B)},$$
for any ball $B\subset \R^n$.
Iterating these estimates from $l=2$ to $l=k$, we obtain the desired inequality with $\theta=2^{-k}$, after relabeling the balls; the case of general $\theta$ follows similarly. The last claim follows from Gehring's Lemma.
\end{proof}
Corollary \ref{cor:curlUI} is a higher integrability result for sequences of Hessians of uniformly $k$-convex functions.
In the case $k=n$, the fields in Corollary \ref{cor:curlUI} are in fact the gradients of so-called \textit{$\delta$-monotone mappings} and in this case the improved integrability claimed above follows from the general theory of these mappings, as developed in \cite{Kovalev2007} and \cite[\S 3.11]{Astala2009}. 

We also deduce the reverse H\"older inequality:
\begin{corollary}
Let $\theta\in(0,1)$  and $\mathcal K\subset\mathrm{int}\,\Gamma_k\cup\{0\}$ be a closed convex cone. Then
$$
\|\D^2u\|_{L^{{k}}(\theta\B^n)}\leq C\|\D^2u\|_{L^1(\B^n\setminus(\theta\B^n))}\quad\text{for }u\in C^\infty(\bar\B^n)\text{ with }\D^2 u\in\mathcal{K}.
$$
\end{corollary}

Finally, we observe that a weaker version of Corollary \ref{cor:curlUI} holds outside the set of symmetric matrices. Indeed, extend the integrands $F_k$, defined  in \eqref{def:Fk} for matrices in $\textup{Sym}_n$, to $\R^{n \times n}$ through the formula
$$F_k(A)=\frac{1}{{n \choose k}}\sum_M M(A),$$
where the sum is over all $k\times k$ principal minors $M$. Applying Corollary \ref{cor:GRS}, we obtain:

\begin{corollary}\label{cor:LklogL}
Let $2\leq k \leq n$, $p>k_*$ and $\mathcal K \subset  \{A\in \R^{n\times n}: F_k(A)> 0\}$ be a closed cone. Then
$$
\|A\|_{L^k\log L(\B^n)}\leq C\left(\|A\|_{L^k(\B^n)}+\|\Curl A\|_{L^p(\B^n)}\right)\quad\text{for }A\in C_c^\infty(\B^n,\mathcal K).
$$
In particular, if a sequence $\D u_j\in C^\infty(\R^n,\mathcal K)$ is bounded in $L^k_{\locc}(\R^n)$ then it is locally $k$-uniformly integrable.  
\end{corollary}

Clearly $\Gamma_k\subset \{A\in \R^{n\times n}: F_k(A)\geq 0\}$ and note that, when $k=n$, this last cone coincides with the closure of $\{\det>0\}$. 

\begin{proof}[Proof of Theorem \ref{thm:maincurl}]
 The proof of Theorem \ref{thm:maincurl} now follows similar lines to the proof of Theorem \ref{thm:maindiv}, where we now combine Propositions \ref{prop:curlmain} and \ref{prop:Curl_optimal}. In particular, \eqref{eq:qmaxcurl} holds.
\end{proof}

\section{Other examples}
\label{sec:otherexamples}
In this final section, we treat three  examples slightly different from those considered above. The first is an example in which we obtain the maximal gain of integrability exponent $q_{\max}=\infty$. This operator, while related to the pair $(\Div,\Curl)$, is not elliptic, but has a wave cone that is not spanning. 

The second example  is a second order variant of the $\Curl$, similar to the $\Div^2$ operator considered above. This is the composition of $\di$ with $\mathcal Q$, the sub-curl type operator defined above in Definition \ref{notation:A_2}, and we obtain the same $q_{\max}$ as in that case. 

The final example is a sub-curl operator strictly weaker than the $\Curl$, acting on functions constrained to $\textup{Sym}^+$.

\subsection{A degenerate example}
\label{sec:degenerate}
In this subsection we give a first order example. The cone we consider is very large, $\mathcal K=\{\tr\geq 0\}$, and the operator is almost elliptic. The example is somewhat degenerate in the sense that the wave cone of the operator is not spanning. However, in the notation of Conjecture \ref{conj:big}, we obtain $q_{\max}=\infty$.

\begin{proposition}
\label{prop:divcurl}Let $\mathcal K\equiv \{A\in \R^{n\times n}:\tr(A)\geq 0\}$.
\begin{enumerate}
    \item\label{it:p<n} Let $1\leq p<n$. We have that 
$$
\|\tr A\|_{\lebe^{p^*}(\R^n)}\leq C\left(\|\Div A\|_{\lebe^p(\R^n)}+\|\Curl (A^\top) \|_{\lebe^p(\R^n)}\right) \quad \text{for } A\in C^\infty_c(\R^n,\mathcal K).
$$
\item\label{it:p>n}  Let $p> n$. We have that
$$
\|\tr A\|_{\lebe^{\infty}(\B^n)}\leq C\left(\|\Div A\|_{\lebe^p(\B^n)}+\|\Curl (A^\top) \|_{\lebe^p(\B^n)}\right) \quad \text{for } A\in C^\infty_c(\B^n,\mathcal K).
$$
\end{enumerate}
\end{proposition}
We remark that in the case that $\widetilde{\mathcal K}\subset \mathrm{int}\,\mathcal K\cup\{0\}$ is a  closed convex cone, we have
the higher integrability estimates 
$$
\|A\|_{\lebe^{p^*}(\R^n)}\leq C\left(\|\Div A\|_{\lebe^p(\R^n)}+\|\Curl (A^\top) \|_{\lebe^p(\R^n)}\right) \quad \text{for } A\in C^\infty_c(\R^n,\widetilde{\mathcal K}).
$$
We note that the operator $\mathcal P A\equiv (\Div A, \Curl (A^\top))$ is both elliptic and canceling whenever $A\in \textup{Sym}_n$, in which case Proposition \ref{prop:divcurl} follows from the work of Van Schaftingen \cite{VanSchaftingen2011} and Stolyarov \cite{Stolyarov2020}. However, for non-symmetric matrices, this operator is not elliptic: it is easy to verify that $\mathcal P(\xi)( \xi \otimes \xi^\bot)=0$, where $\xi^\bot$ is any vector orthogonal to $\xi$.  Therefore Proposition \ref{prop:divcurl} is genuinely different from the previously mentioned results.

\begin{proof}
Consider first the case $1\leq p<n$. Let us take $A\in C^\infty_c(\R^n,\mathcal K)$. Writing $A=(a_{ij})$ and $A^\top=(a^\top_{ij})$, and $A_i, A_i^\top$ for their rows, we have
\begin{align}\label{eq:L2coordinates}
({\Div A - \D\tr(A)})_i
= \sum_j \p_j a_{ij} -\p_i a_{jj} 
= \sum_j \p_j a^\top_{ji} -\p_i a^\top_{jj}
= -\sum_j (\curl A^\top_j)_{ij}.
\end{align}
It then follows from the usual Gagliardo--Nirenberg inequality that
$$
\Vert \tr(A)\Vert_{L^{p^*}}
{\leq C} 
\Vert \D \tr(A)\Vert_{L^p}
\leq 
\Vert \Div A \Vert_{L^p}+\Vert \Curl A^\top \Vert_{L^p},
$$
as wished.
An  analogous argument proves part \ref{it:p>n}.
\end{proof}

We note that, by using the refined Gagliardo--Nirenberg--Sobolev inequality  we can actually deduce a stronger form of Proposition \ref{prop:divcurl}, where the $\lebe^{p^*}$-norm can be replaced by the Lorentz $\lebe^{p^*,p}$-norm \cite{Alvino1977,Oneil1963}.

\subsection{A second order example}
\label{c}
In this subsection we consider an example of a second order operators and a cone for which we have the full resolution of Conjecture \ref{conj:big}. The results are similar in spirit to Theorem \ref{thm:div^2}.
\begin{proposition}\label{prop:div^2L_2A}
    Let $p,\,q\geq 1$. 
We have that 
    $$
    \|\rho_2(A)\|_{L^q(\B^n)}\leq C\|\Div^2A-\Delta \tr A\|_{L^p(\B^n)}\quad\text{for }A\in C_c^\infty(\B^n,\Gamma_2),
    $$
whenever
\begin{enumerate}
   \item\label{itm:div2L2a} $p>1$, $n\geq 4$, and $q=\min\{p^{**},2\}$,
      \item $p=1$, $q<1^{**}$, $q\leq 2$.
\end{enumerate}
As before, if $\mathcal K\subset \mathrm{int}\,\Gamma_2\cup\{0\}$ is a closed convex cone then, for the same ranges of $p$ and $q$, we have
$$
    \|A\|_{L^q(\B^n)}\leq C\|\Div^2A-\Delta \tr A\|_{L^p(\B^n)}\quad\text{for }A\in C_c^\infty(\B^n,\mathcal K).
    $$
\end{proposition}
We recall that here $p^{**}=(p^*)^*=\frac{np}{n-2p}$ is the Sobolev exponent. Note that in the special cases $n=2,\,3$, we can take $p=1$, $q=2$.
 We also have 
\begin{equation}
    \label{eq:divQ}
\Div^2A-\Delta \tr A= \Div^2(L_{\mathcal Q} A)=\di\mathcal Q A,
\end{equation}
where we recall Definition \ref{notation:A_2}.

\begin{proof}
Let $A\in C_c^\infty(\B^n,\mathcal K)$ and let $\psi\in C^{1,1}(\overline{\B^n})$  be the $k$-admissible solution of
$$\begin{cases}
\rho_2(\D^2\psi)=f\equiv \rho_2(A)^{q-1} & \text{ in }\B^n,\\
\psi=0 &\text{ on }\d \B^n.
\end{cases}
$$
In particular, $\D^2\psi\in\Gamma_2$ a.e.\ in $\B^n$ and we have the estimate
$$
\|\psi\|_{L^{p'}(\B^n)}{\leq C} \|f\|_{L^{q'}(\B^n)}= \left(\int_{\B^n}\rho_2(A)^q\right)^{1/q'}
$$
by  Lemma \ref{lemma:Lpest}, where the restrictions on $p$ and $q$ follow from those in the statement of the lemma.
We then estimate using inequality \eqref{eq:bilinear} and integration by parts twice:
\begin{align*}
    \int_{\B^n} \rho_2(A)^{q}=&\int_{\B^n} \rho_2(A)\rho_2(\D^2\psi){\leq C} \int_{\B^n} \langle L_{\mathcal Q}A, \D^2\psi \rangle =\int_{\B^n} \psi\Div^2(L_{\mathcal Q}A)\\
    &{\leq C} \|\psi\|_{L^{p'}(\B^n)}\|\Div^2(L_{\mathcal Q}A)\|_{L^p(\B^n)}{\leq C} \left(\int_{\B^n}\rho_2(A)^q\right)^{1/q'} \|\Div^2(L_{\mathcal Q}A)\|_{L^p(\B^n)}.
\end{align*}
The first inequality follows from \eqref{eq:divQ} and the second is also proved after noting that for $A\in\mathcal{K}$ we have $|A|{\leq C} \rho_2(A)$. 
\end{proof}

The next result shows that the exponent $q=2$ in Proposition \ref{prop:div^2L_2A} is optimal.

\begin{proposition}\label{prop:optimality}
    Let $n\geq2$. For each $\varepsilon>0$ there is a closed convex cone $\mathcal K^\e\subset \mathrm{int}\,\Gamma_2\cup\{0\}$
such that the inequality
$$
\|v\|_{L^{2+\varepsilon}(\mathbb B^n)}\leq C\|\cala v\|_{L^\infty(\mathbb B^n)}\quad\text{for }v\in C_c^\infty(\mathbb B^n,\mathcal K^\e)
$$
fails, where
 $\A=\di\mathcal Q=\Div^2-\Delta\tr$.
\end{proposition}
 In particular, in Proposition \ref{prop:optimality}, we can take $n=2$.
In this case, the estimate predicted by \cite[Conjecture 1.9]{Arroyo-Rabasa2021a} should hold for any $L^q$, $q<\infty$, on  the right hand side, and so the proposition falsifies that conjecture.

\begin{proof}
Let $u(x)=u_\e(x_1,x_2)$, where $u_\e$ is given by Lemma \ref{lem:radial} for $n=2$. Then clearly $\D^2u\in L^2\setminus L^{2+\varepsilon}_{\locc}(\mathbb B^n,\mathcal K^\e)$, where
$$
\mathcal K^\e=\left\{
\left(\begin{matrix}
A_0 &0\\
0&0
\end{matrix}\right)\colon A_0\in \textup{Sym}_2\cap Q_2^+(K)
\right\}
$$
is contained in $\mathrm{int\,}\Gamma_2\cup\{0\}$ and $K=K(\e)$. By \eqref{eq:divQ} and Definition \ref{notation:A_2}, we have that $\A \D^2u=0$. The same cut-off argument as in the proof of Theorem \ref{thm:maindiv}, noting that $$
\Div^2L_{\mathcal Q}(\rho \D^2u)=\langle L_{\mathcal Q}\D^2\rho,\D^2u \rangle,
$$
allows us to conclude.
\end{proof}

The consequence pertaining to Hessians in Corollary \ref{cor:curlUI} for $k=2$ can be improved using Proposition \ref{prop:div^2L_2A} instead of Proposition \ref{prop:curlmain}. The improvement concerns the operator; rather than imposing that we have a sequence of exact Hessians, we instead require only that $\mathcal Q A_j=\Div L_{\mathcal Q}A_j=0$ for each $j$.
\begin{corollary}\label{cor:divL2UI}
 Let 
 $\mathcal K\subset \mathrm{int}\,\Gamma_2\cup\{0\}$ be a closed convex cone. Consider a sequence $A_j\in C^\infty(\R^n,\mathcal K)$ which is bounded in $L^1_{\locc}(\R^n)$ and such that $\mathcal Q A_j=0$ for all $j$. There is $\delta=\delta(\mathcal K)$ such that $A_j$ is bounded in $L^{2+\delta}_{\locc}(\R^n).$ In particular, $(A_j)$ is locally $2$-uniformly integrable.
\end{corollary}

\begin{proof}
For a scalar field $\rho\in C_c^\infty(2B)$ we have
\begin{align*}
\Div^2(L_{\mathcal Q}(\rho A))& =\di(\rho \Div(L_{\mathcal Q}A)+(L_{\mathcal Q}A)\D\rho)\\ 
& =\Div(L_{\mathcal Q}A)\cdot \D\rho+\langle \D^2\rho, L_{\mathcal Q}A\rangle =\langle \D^2\rho,L_{\mathcal Q}A\rangle.
\end{align*}
Applying Proposition \ref{prop:div^2L_2A} iteratively, the rest of the proof now follows the same lines as that of Corollary \ref{cor:divUI}.
\end{proof}

\subsection{A first order example}
It is natural to extend the slicing methods of Section \ref{sec:GN} to higher dimensions, and so we conclude this section with an inequality which generalizes Corollary \ref{cor:curlinequality2d} to dimensions $n\geq 3$. In particular, we prove the end-point estimate for $p=1$ in the setting of \eqref{eq:main_inequ}, for fields restricted to $\textup{Sym}^+_n$ and an operator which is related to but strictly weaker than $\Curl$. To fix the notation for this operator, we distinguish between the cases $n=2k$ is even and $n=2k+1$ is odd.

In even dimensions $n=2k$, we define the operator $\mathcal  P_n$ which takes a matrix field $A=(a_{ij})$ to the $2\times k$ matrix formed of curls acting in $\p_j$, $\p_{j+1}$ on principal sub-matrices $A_{j,j+1}$ with $j$ odd:
 $$\mathcal  P_n A=\begin{pmatrix}
 \p_2 a_{11}-\p_1 a_{12} & \cdots & \p_{n} a_{n-1,n-1}-\p_{n-1} a_{n-1,n}\\
 \p_2 a_{21}-\p_1 a_{22} & \cdots & \p_n a_{n,n-1}-\p_{n-1} a_{nn}
 \end{pmatrix}.$$
Note that this has exactly $2k=n$ entries, i.e.\ it has the same number of entries as  $\Div$. One easily checks that $\Lambda_{\mathcal  P_n}$ consists of matrices of the form
$$\begin{pmatrix}
a_1\otimes (\xi_1, \xi_2) & * & * \\
* & \ddots & *\\
* & * & a_k\otimes(\xi_{n-1}, \xi_n) 
\end{pmatrix}$$
where $a_1, \dots, a_k \in \R^2$ and $*$ denotes arbitrary entries.

When $n=2k+1$, we define $\mathcal P_n$ as the operator that takes a matrix field $A=(a_{ij})$ to the $2\times n$ matrix formed of curls acting in $\p_j$, $\p_{j+1}$ on principal sub-matrices $A_{j,j+1}$ with $j=1,\dots,n$, the last pair being taken modulo $n$:
 $$\mathcal  P_n A=\begin{pmatrix}
 \p_2 a_{11}-\p_1 a_{12} & \p_3 a_{22}-\p_2 a_{23} \cdots & \p_{1} a_{n,n}-\p_{n} a_{n,1}\\
 \p_2 a_{21}-\p_1 a_{22} & \p_3 a_{32}-\p_2 a_{33} \cdots & \p_1 a_{1,n}-\p_{n} a_{11}
 \end{pmatrix}.$$

\begin{proposition}
\label{prop:curlinequality}
Take $K\geq 1$ and let $A\in C_c^\infty(\R^n, \textup{Sym}^+_n\cap Q_n^+(K))$. Then
$$\|A\|_{L^{1^*}(\R^n)}\leq C_K \|\mathcal P_n A \|_{L^1(\R^n)}.$$
\end{proposition}

Note that Proposition \ref{prop:curlinequality} yields an example of the situation posed in Conjecture \ref{conj:big} in which we may take $p=1$ and also have the critical Sobolev gain of integrability $q=1^*$. This result is therefore complementary to many of the other results of this paper, in which we typically obtain the sub-critical estimate $q=q_{\max}$ with $p>(q_{\max})_*$. However, Proposition \ref{prop:curlinequality} should be compared to Proposition \ref{prop:curlmain} in which we obtain the same estimate for a stronger operator in a larger family of cones.

The proof of Proposition \ref{prop:curlinequality} is based on a higher dimensional slicing argument.

\begin{proof}
Throughout, we will use the notation $A_{j,k}$ for the $2\times 2$ matrix obtained by deleting all rows and columns from $A$ except the $j$-th and $k$-th and $M_{j,k}$ for the determinant of this matrix.

We begin with two simple observations. Firstly, any matrix $A\in \textup{Sym}^+_n$ has positive diagonal entries; secondly, under our assumptions, given any $2\times 2$ principal minor $M$, there is a constant $c=c(K)$ such that
$M(A)\geq c |A|^2$.
The rest of the proof is split into two cases. 

\textit{Case 1: even dimension.} Assume $n=2k$. For each sub-matrix $A_{j,j+1}$ with $j$ odd, $j=1,3,...,n-1$, we may apply Corollary \ref{cor:curlinequality2d} to get
\begin{align*}
    \int_{\R^2}M_{j,j+1}(A)\dif  x_{j}\dif  x_{j+1}\leq \Big(\int_{\R^2}|\mathcal P_n A|(x)\dif  x_{j}\dif  x_{j+1}\Big)^2.
\end{align*}
Using the lower bound $M_{j,j+1}\geq c|A|^2$, we take square root of this inequality and integrate in the other variables to obtain
\begin{align*}
    &\int_{\R^{n-2}}\Big(\int_{\R^2}|A|^2(x)\dif x_{j}\dif x_{j+1}\Big)^{\frac12}\dif x_1\cdots\dif  x_{j-1}\dif  x_{j+2}\cdots\dif  x_n\leq C\|\mathcal P_n A\|_{L^1(\R^n)}.
\end{align*}
Applying Minkowski's integral inequality to re-order the integrals, this implies
\begin{align*}
\|A\|_{L^{\mathbf{p}_j}(\R^n)}& \equiv \,\int_{\R^{n-(j+1)}}\Big(\int_{\R^2}\Big(\int_{\R^{j-1}}|A|\dif  x_1...\dif  x_{j-1}\Big)^2\dif  x_j\dif  x_{j+1}\Big)^{\frac{1}{2}}\dif  x_{j+2}...\dif  x_n\\
& \leq C\|\mathcal P_n A\|_{L^1(\R^n)};
\end{align*}
here $\mathbf{p}_j\in \R^n$ has the $j$-th and $(j+1)$-th entries equal to 2 and all other entries equal to 1. We refer the reader to the book of Adams and Fournier for further details concerning these mixed-norm spaces \cite{Adams2003}.

We take the product of all these $k$ inequalities (for $j=1,3,5,...,n-1$) and apply the generalised H\"older's inequality, see e.g.\ \cite[Theorem 2.49]{Adams2003},
to obtain
$$\| |A|^k\|_{L^{\mathbf{r}}}\leq\prod_{\substack{j=1\\ j\text{ odd}}}^{n-1}\|A\|_{L^{\mathbf{p}_j}}\leq C^k\|\mathcal P_n A\|_{L^1}^k,$$
where $\mathbf{r}$ is computed component-wise  by 
$$\frac{1}{\mathbf{r}}=\sum_{\substack{j=1 \\ j\text{ odd}}}^{n-1}\frac{1}{\mathbf{p}_j}=(k-1+\frac{1}{2},\ldots,k-1+\frac{1}{2}).$$
Thus,
the norm on the left-hand side is simply the $L^r$-norm with $r=\frac{2}{2k-1}=\frac{2}{n-1}$, and so we may rewrite the previous inequality as
$$\Big(\int_{\R^n}|A|^\frac{n}{n-1}\dif  x\Big)^{\frac{n-1}{2}}\leq C\|\mathcal P_n A\|_{L^1(\R^n)}^k.$$
Taking the $k$-th root, the desired inequality follows.

\textit{Case 2: odd dimension.}  Assume $n=2k+1$. Arguing as above, i.e.\ by considering the $M_{j,j+1}$ minor, we easily find, for $j=1,2,3,\ldots,n-1$, the estimate
 $$\|A\|_{L^{\mathbf{p}_j}}\leq C\|\mathcal P_n A\|_{L^1},$$
where $\mathbf{p}_j$ is as before.
We now work with the sub-matrix $A_{1,n}$ on the $x_1$-$x_n$ plane to obtain
\begin{align*}
    &\int_{\R^{n-2}}\Big(\int_{\R^2}|A|^2(x)\dif x_{1}\dif x_{n}\Big)^{\frac12}\dif x_2\cdots\dif x_{n-1}\leq C\|\mathcal P_n A\|_{L^1},
\end{align*}
leading to 
$$\|A\|_{L^{\mathbf{p}_n}}\leq C\|\mathcal P_n A\|_{L^1},$$
where $\mathbf{p}_n\in \R^n$ has the first and last entries equal to 2 and all other entries equal to 1.
We note that each component of the sum of the terms $\frac{1}{\mathbf{p}_j}$ consists of $n-2+2\times \frac 1 2 $, so that
$$\sum_{j=1}^{n}\frac{1}{\mathbf{p}_j}=(n-1,\ldots,n-1).$$
Thus, applying the generalised H\"older's inequality,
\begin{align*}
    \||A|^n\|_{L^{\frac{1}{n-1}}(\R^n)}\leq\prod_{j=1}^n\|A\|_{L^{\mathbf{p}_j}(\R^n)}\leq C^n\|\mathcal P_n A\|_{L^1(\R^n)}^n.
\end{align*}
Taking the $n$-th root, we obtain the desired inequality.
\end{proof}

\appendix

\addtocontents{toc}{\protect\setcounter{tocdepth}{1}}
\section{Proof of Theorem \ref{thm:khessian}}
\label{appendix}

In this Appendix we include a proof of Theorem \ref{thm:khessian}. 
We begin by proving the uniqueness claim: we observe that if $u$ and $v$ are both solutions, then $u-v$ satisfies the uniformly elliptic equation
$$\rho_k(\D^2 u)-\rho_k(\D^2 v)=0,$$
to which the comparison principle applies. It follows that $u-v$ is constant.

To prove the existence claim we rely on argument similar to the one used by YanYan Li in \cite{Li1990}. The crucial point is to obtain an a priori estimate on the oscillation of $\psi$, where we write $\psi \equiv \frac 1 2 S x\cdot x + \varphi$. 

Let $x_0,x_1$ be such that $\psi(x_0)=\min_{\mathbb T^n} \psi$, $\psi(x_1)=\max_{\mathbb T^n}\psi$. Let $\tilde \psi = \psi-\max_{2 \mathbb T^n}\psi$, which is non-positive in $2\mathbb T^n$. With $R=|x_0-x_1|\leq 1$, and according to Theorem 9.1 in \cite{Wang2009}, 
$$-\tilde\psi(x_0){\leq C} \Lambda -\tilde\psi(x_1) \implies \tilde \psi(x_1)-\tilde \psi(x_0)=\textup{osc}_{\mathbb T^n}\psi {\leq C} \Lambda,$$
    where $\lambda\leq f \leq \Lambda$. 
Thus we deduce the estimate
$$\sup |\psi| {\leq C} \Lambda.$$
This implies a $C^1$ estimate on $\psi$ by \cite[Thm. 4.1]{Wang2009}, which in turn implies a $C^2$ estimate on $\psi$ (see for instance \cite{Guan2014}). From here we can apply the Evans--Krylov Theorem, see e.g.\ \cite{Caffarelli1995}, to deduce that
$$\|\psi \|_{C^{2,\alpha}}\leq C(\Lambda,\lambda) \implies \|\varphi\|_{C^{2,\alpha}}\leq C(\Lambda,\lambda)+C\|S\|. $$
We conclude the proof by applying the continuity method, as in \cite{Li1990}: we let
$$S_t(\varphi)=F_k(S+D^2 \varphi)-t f^k -(1-t)F_k(S)$$
which maps $C^{2,\alpha}_\sharp$ to  $C^\alpha_\sharp$; here the subscript denotes mean-free functions. The crucial point now is that, for $\varphi$ in $C^{2,\alpha}$, if $S+D^2\varphi\in \Gamma_k$, the linearized operator $S'_t$ is elliptic. It follows in particular from standard Schauder theory that $S_t'\colon C^{2,\alpha}_\sharp \to C^\alpha_\sharp$ is bijective, i.e.\ the linearized equation can be solved uniquely for a given datum. Thus the continuity method applies.

\section{$\A$-quasiaffine integrands and $L\log L$ estimates}\label{app:LlogL}
Given finite dimensional inner product spaces $\mathbb{V}$ and $\mathbb{W}$ and an $\ell$-th order homogeneous differential operator
$$\A=\sum_{|\al|=\ell}L_\al\d^\al\quad\text{ where }L_\al\in\textup{Lin}(\mathbb{V},\mathbb{W}),$$
we say that a locally bounded, Borel measurable integrand $F\colon \mathbb V\rightarrow\R$  is \textit{$\A$-quasiconvex} if
$$
\int_{\T^n} F(v(x))\dif x\geq F\left(\int_{\T^n}v(x)\dif x\right)\quad\text{for all } v\in C^\infty(\T^n,\V) \text{ with } \A v=0,
$$
and we say that $F$ is \textit{$\A$-quasiaffine} if both $\pm F$ are $\A$-quasiconvex. The latter is the class of weakly continuous functionals defined on $\A$-free sequences 
\cite{Fonseca1999,Guerra2019,Murat1981}, which benefits from surprising improvements in integrability \cite{Coifman1993,Guerra2019,GuerraRaitaSchrecker2020,Muller1990}; we will recall a general result below. To this end, recall that the differential operator $\mathcal{A}$, with symbol
$$
 \A(\xi)=\sum_{|\alpha|=\ell}\xi^\alpha L_\alpha \in\lin(\V,\W),
$$
is said to have \textit{constant rank} if $\mathrm{rank\,}\A(\xi)$ is independent of $\xi\in\R^n\setminus\{0\}$, and is said to have a \textit{spanning wave cone} if
$$
\mathrm{span\,}\Lambda_\A=\V,\quad\text{where }\Lambda_\A=\bigcup_{\xi\in\R^n\setminus\{0\}}\ker\A(\xi).
$$
Under the spanning wave cone condition it is classical that $\A$-quasiaffine integrands are polynomials. We then have:
\begin{theorem}[{\cite[Thm.\ C]{GuerraRaitaSchrecker2020}}]\label{thm:GRS}
Let  $\A$ be a constant rank operator with spanning wave cone, and let $F\colon \mathbb V\rightarrow\R$ be an $\A$-quasiaffine polynomial of degree $s\in\{2,\ldots,\min\{n,\dim \V\}\}$. Let $r>s$ be a real number. Then
$$
\|F(v)\|_{L\log L(\B^n)}\leq C\left(\|v\|_{L^s(\B^n)}+\|\A v\|_{{\dot{W}}{^{-\ell,r}}(\B^n)}\right)\quad\text{for }v\in C_c^\infty(\B^n,\V) \text{ with }F(v)\geq 0.
$$
\end{theorem}
This result can be used to give inequalities in the spirit of \eqref{eq:main_inequ}:
\begin{corollary}\label{cor:GRS}
In the setting of Theorem \ref{thm:GRS}, let $\mathcal K=\{z\in\mathbb V\colon  c|z|^s\leq F(z)\}$, for some constant $c>0$. Then
$$
\|v\|_{L^s\log L(\B^n)}\leq C\left(\|v\|_{L^s(\B^n)}+\|\A v\|_{{\dot{W}}{^{-\ell,r}}(\B^n)}\right)\quad\text{for }v\in C_c^\infty(\B^n,\mathcal K).
$$
\end{corollary}
We therefore obtain sharp uniform integrability statements:
\begin{corollary}
In the setting of Corollary \ref{cor:GRS}, if a sequence $(v_j)\subset C_c^\infty(\B^n,\mathcal K)$ is bounded in $L^s(\B^n)$ and $(\A v_j)$ is bounded in ${\dot{W}}{^{-\ell,r}}(\B^n)$, we have that $(v_j)$ is $s$-uniformly integrable.
\end{corollary}

\addtocontents{toc}{\protect\setcounter{tocdepth}{-1}}

\bibliographystyle{siam}

\end{document}